\documentclass[10pt, a4paper]{article}
\usepackage[utf8]{inputenc} 
\usepackage[english]{babel} 
\usepackage[usenames, dvipsnames]{xcolor}
\usepackage[a4paper, tmargin=1.0in,bmargin=1.0in, lmargin=0.7in, rmargin=0.7in]{geometry} 
\usepackage{bm} 
\usepackage{amsmath} 
\usepackage{amsthm}
\usepackage{amssymb}
\usepackage{cases} 
\usepackage{tikz} 
\usepackage{stmaryrd} 
\usepackage{multirow} 
\usepackage{lipsum} 
\usepackage{wrapfig} 
\usepackage{oplotsymbl} 
\usepackage{scalerel}
\usepackage{mathrsfs}  
\usepackage{dsfont}
\usepackage{breakcites} 
\usepackage{graphicx}

\def\Xint#1{\mathchoice
   {\XXint\displaystyle\textstyle{#1}}%
   {\XXint\textstyle\scriptstyle{#1}}%
   {\XXint\scriptstyle\scriptscriptstyle{#1}}%
   {\XXint\scriptscriptstyle\scriptscriptstyle{#1}}%
   \!\int}
\def\XXint#1#2#3{{\setbox0=\hbox{$#1{#2#3}{\int}$}
     \vcenter{\hbox{$#2#3$}}\kern-.5\wd0}}

\def\dashint{\Xint-}

\newcommand{\labelWest}{\scalerel*{\circletfillhl}{\labPosX}}
\newcommand{\labelEast}{\scalerel*{\circletfillhr}{\labPosX}}
\newcommand{\labelSouth}{\scalerel*{\circletfillhb}{\labPosX}}
\newcommand{\labelNorth}{\scalerel*{\circletfillha}{\labPosX}}

\usepackage{empheq}

\usepackage{placeins} 

\usepackage{environ}

\usepackage[colorlinks=true,linkcolor=OrangeRed,citecolor=RoyalBlue]{hyperref} 
\usepackage{cleveref} 

\providecommand{\keywords}[1]{\textbf{\textit{Keywords: }} \emph{#1}}
\providecommand{\amsCat}[1]{\textbf{\textit{MSC: }} #1}

\theoremstyle{plain}
\newtheorem{proposition}{Proposition}%
\newtheorem{lemma}{Lemma}%
\newtheorem{theorem}{Theorem}%
\newtheorem*{theorem*}{Theorem}
\theoremstyle{definition}
\newtheorem{definition}{Definition}%
\newtheorem{assumption}{Assumption}%
\theoremstyle{remark}
\newtheorem{remark}{Remark}%

\title{Equilibrium boundary conditions for vectorial multi-dimensional lattice Boltzmann schemes}

\author{\textsc{Denise Aregba-Driollet}\footnote{Université de Bordeaux, CNRS, Bordeaux INP, IMB, UMR 5251, 33400 Talence, France.}\and{} \textsc{Thomas Bellotti}\footnote{Université Paris-Saclay, CNRS, CentraleSupélec, Laboratoire EM2C \& Fédération de Mathématiques de CentraleSupélec, 91190, Gif-sur-Yvette, France.}}

\ifpdf
\hypersetup{
  pdftitle={Equilibrium BC for LBM},
  pdfauthor={D. Aregba-Driollet and T. Bellotti}
}
\fi

\newcommand{\numberCells}{J}
\newcommand{\numberConservationLaws}{M}

\newcommand{\timeVariable}{t} 
\newcommand{\spaceVariable}{x} 
\newcommand{\vectorial}[1]{\bm{#1}} 
\newcommand{\matricial}[1]{\bm{#1}} 
\newcommand{\conservedVariable}{u} 
\newcommand{\spatialDimensionality}{d} 
\newcommand{\flux}{\varphi} 
\newcommand{\numberVelocities}{q} 
\newcommand{\idEst}{\emph{i.e.}}
\newcommand{\confer}{\emph{cf.}}
\newcommand{\exempliGratia}{\emph{e.g.}}
\newcommand{\spaceStep}{\Delta \spaceVariable} 
\newcommand{\timeStep}{\Delta \timeVariable} 
\newcommand{\latticeVelocity}{\lambda} 
\newcommand{\discreteVelocityLetter}{c} 
\newcommand{\indexVelocity}{i} 
\newcommand{\naturals}{\mathbb{N}} 
\newcommand{\naturalsWithoutZero}{\mathbb{N}^*} 
\newcommand{\integerInterval}[2]{\llbracket#1, #2\rrbracket} 
\newcommand{\discrete}[1]{\mathsf{#1}} 
\newcommand{\distributionFunctionLetter}{f}
\newcommand{\distributionFunction}{\discrete{\distributionFunctionLetter}} 
\newcommand{\indexSpace}{j} 
\newcommand{\indexTime}{n} 

\newcommand{\definitionEquality}{:=} 
\newcommand{\collided}{\star} 

\newcommand{\relaxationParameter}{\omega} 
\newcommand{\relaxationParameterSymmetric}{\omega_{\discrete{s}}} 
\newcommand{\relaxationParameterAntiSymmetric}{\omega_{\discrete{a}}} 

\newcommand{\atEquilibrium}{\textnormal{eq}} 
\newcommand{\conservedVariableDiscrete}{\discrete{\conservedVariable}} 
\newcommand{\equilibriumCoefficientLinear}{\mathscr{L}} 
\newcommand{\differential}{\textnormal{d}} 
\newcommand{\maximumInitialDatum}{u_{\infty}} 
\newcommand{\invariantCompactSetDistributions}{K} 
\newcommand{\lbmScheme}[2]{$\textnormal{D}_{#1}\textnormal{Q}_{#2}$} 
\newcommand{\reals}{\mathbb{R}} 
\newcommand{\realsPositive}{\mathbb{R}_+^{*}}
\newcommand{\initial}{\circ} 
\newcommand{\lebesgueSpace}[1]{L^{#1}} 
\newcommand{\boundedVariationSpace}{\textnormal{BV}} 
\newcommand{\continuousSpace}[1]{C^{#1}} 
\newcommand{\timeGridPoint}[1]{\timeVariable^{#1}} 
\newcommand{\collisionOperator}{\mathcal{R}} 
\newcommand{\cell}[1]{C_{#1}} 
\newcommand{\discreteMark}{\Delta} 
\newcommand{\distributionFunctionsAsFunction}{\vectorial{\distributionFunctionLetter}_{\discreteMark}} 
\newcommand{\distributionFunctionsAsFunctionComponent}[1]{{\distributionFunctionLetter}_{\discreteMark, #1}} 
\newcommand{\distributionFunctionsAsFunctionWithStep}[1]{\vectorial{\distributionFunctionLetter}_{#1}} 
\newcommand{\conservedVariableDiscreteAsFunction}{\conservedVariable_{\discreteMark}} 
\newcommand{\conservedVariableDiscreteAsFunctionWithStep}[1]{\conservedVariable_{#1}} 

\newcommand{\conservedVariableDiscreteAsFunctionSpace}[1]{\conservedVariable_{\discreteMark}^{#1}} 
\newcommand{\transpose}[1]{#1^{\textsf{T}}} 
\newcommand{\indicatorFunction}[1]{\mathds{1}_{#1}} 
\newcommand{\lebesgueInSpace}[1]{L_{\vectorial{\spaceVariable}}^{#1}}

\newcommand{\lebesgueTime}[1]{L_{\timeVariable}^{#1}} 
\newcommand{\petitLebesgueSpace}[1]{\ell^{#1}} 

\newcommand{\strong}[1]{\emph{#1}} 
\newcommand{\bigO}[1]{\mathcal{O}(#1)}
\newcommand{\canonicalBasisVector}[1]{\vectorial{e}_{#1}}
\newcommand{\limitConservedMoment}{\overline{\conservedVariable}}
\newcommand{\limitDistributionFunction}{\overline{\distributionFunctionLetter}}
\newcommand{\finalTime}{T}
\newcommand{\testFunction}{\psi}

\newcommand{\testFunctionDiscrete}{\discrete{\psi}}

\newcommand{\kineticEntropy}{s}
\newcommand{\krushkovParameter}{\kappa}
\newcommand{\sign}{\textnormal{sgn}}

\newcommand{\labZeroVel}{\times}
\newcommand{\labPosX}{\Yright}
\newcommand{\labPosY}{\Yup}
\newcommand{\labNegX}{\Yleft}
\newcommand{\labNegY}{\Ydown}
\newcommand{\setVelIndexes}{I}

\newcommand{\totalVariation}[2]{\mathscr{V}(#1, #2)}
\newcommand{\totalVariationAlongAxis}[3]{\mathscr{V}_{#3}(#1, #2)}

\newcommand{\xLabel}{x}
\newcommand{\yLabel}{y}
\newcommand{\boundaryDatum}[2]{\tilde{\conservedVariableDiscrete}_{#1, #2}}
\newcommand{\boundaryDatumVectorial}[2]{\tilde{\vectorial{\conservedVariableDiscrete}}_{#1, #2}}
\newcommand{\boundaryFunction}[1]{\tilde{\conservedVariable}_{#1}}
\newcommand{\boundaryDatumSouth}[1]{\boundaryDatum{\labelSouth}{#1}}
\newcommand{\boundaryDatumWest}[1]{\boundaryDatum{\labelWest}{#1}}

\newcommand{\boundaryDatumSouthVectorial}[1]{\tilde{\vectorial{\conservedVariableDiscrete}}_{\labelSouth, #1}}
\newcommand{\boundaryDatumWestVectorial}[1]{\tilde{\vectorial{\conservedVariableDiscrete}}_{\labelWest, #1}}
\newcommand{\boundaryDatumEastVectorial}[1]{\tilde{\vectorial{\conservedVariableDiscrete}}_{\labelEast, #1}}
\newcommand{\boundaryDatumNorthVectorial}[1]{\tilde{\vectorial{\conservedVariableDiscrete}}_{\labelNorth, #1}}

\newcommand{\traceOperator}{\gamma}
\newcommand{\smoothFunctionsSpace}{C_c^{2}}
\newcommand{\reduceSpaceDoubleInt}{\hspace{-0.3em}}
\newcommand{\numberTimeSteps}{N}
\newcommand{\advectionVelocity}{V}
\newcommand{\courantNumber}{C}
\newcommand{\identityMatrix}[1]{\matricial{I}_{#1}}
\newcommand{\timeShiftOperator}{z}
\newcommand{\fourierShift}{\kappa}
\newcommand{\solutionCharStable}{\fourierShift_-}
\newcommand{\solutionCharUnstable}{\fourierShift_+}
\newcommand{\cKClass}[1]{C^{#1}}
\newcommand{\sobolevSpace}[2]{W^{#1, #2}}
\newcommand{\smoothFunctionsSpaceWithReg}[1]{C_c^{#1}}
\newcommand{\divergence}{\text{div}}

\allowdisplaybreaks 

\begin{document}


\maketitle

\begin{abstract}
    The concept of equilibrium is a general tool to fill the gap between macroscopic and mesoscopic information, both within kinetic systems and kinetic schemes.
    This work explores the use of equilibria to devise numerical boundary conditions for multi-dimensional vectorial lattice Boltzmann schemes tackling systems of hyperbolic conservation laws.
    In the scalar case, we prove convergence for schemes with monotone relaxation to the weak entropy solution by Bardos, Leroux, and Nédelec [Commun. Partial Differ. Equ., 4 (9), 1979], following the path by Crandall and Majda [Math. Comput., 34, 149 (1980)].
    Numerical experiments are conducted both for scalar and vectorial problems, and demonstrate the effectiveness of equilibrium boundary conditions in capturing significant physical phenomena.
\end{abstract}

\keywords{lattice Boltzmann, boundary conditions, equilibrium, convergence, monotonicity}\\
\amsCat{65M12, 76M28}

\section*{Introduction}

In the context of lattice Boltzmann schemes, where boundaries in time-space are present---that is in \strong{actual numerical simulations}---there is a need to fill the gap between the ``macroscopic'' (the one of the unknowns of the PDEs) and the ``mesoscopic'' world (concerning the distribution functions, or the moments, of the numerical scheme). 
The concept of equilibrium---depending only on macroscopic quantities---offers a universal bridge between macroscopic data, which are generally known from the problem, and mesoscopic ones, which are employed by the numerical scheme.

In the case of ``time-boundaries'', \idEst{} initial conditions, the advantages and limitations of using equilibria to extract the numerous mesoscopic data from the scarce initial data of the Cauchy problem have already been addressed, see for example \cite{van2009smooth, bellotti2024initialisation}.

As far as ``space-boundaries'' are concerned, one lacks data for the distribution functions associated to incoming discrete velocities. Notice that, in principle, each part of the boundary features incoming velocities, whereas the macroscopic PDEs do not necessarily need boundary conditions everywhere, \confer{} the dichotomy between inflows and outflows. 
In this context, the use of equilibria has already been investigated from the practical standpoint, see  \cite[Chapter 5]{timm2016lattice}.
This procedure is indeed called ``wet-node'' approach with ``equilibrium scheme'', and is a quite old one dating back at least to \cite{inamuro1995non}. Further discussions on the applicability of this approach in viscous flows with moderate Reynolds number can be found in \cite{mohamad2009note}.

The present work focuses on lattice Boltzmann schemes, in particular the so-called ``\strong{vectorial schemes}'', to approximate the solution of systems of conservation laws, see for instance \cite{dubois2014simulation, graille2014approximation, bellotti2022multidimensional, bellotti2025fourth, wissocq2024positive}.
In this setting, we investigate the usefulness of boundary conditions based on equilibria and---in the scalar case---we prove that they ensure, under suitable constraints generalizing the celebrated ``sub-characteristic condition'', \strong{convergence} of the numerical solution to the ``physical'' solution of the PDE, namely the one satisfying an entropy inequality introduced by \cite{bardos1979first}.

This attempt has been propted by two concomitant facts. 
The first one is that we have recently been able to prove the convergence of a vast class of two-relaxation-times lattice Boltzmann schemes for scalar non-linear problems on the whole space \cite{aregba2025monotonicity}, and this proof crucially relies on the fact that the discrete solution geometrically relaxes to a $\bigO{\spaceStep}$-neighborhood (with $\spaceStep$ being the grid step) of the equilibrium manifold. It seems therefore natural to use this very equilibrium to devise boundary conditions.
The second fact is that a proof of convergence has been available for a long time \cite{aregba2004kinetic} for scalar 1D relaxation schemes, where the numerical solution is instantaneously (\idEst{} iteration-wise) at equilibrium, when endowed with equilibrium boundary conditions.
The interest of using lattice Boltzmann schemes which do not relax on the equilibrium lays in their reduced numerical diffusion, compared to relaxation schemes.  
The present extension of the work by Aregba-Driollet and Mili{\v{s}}i{\'c} to lattice Boltzmann schemes is non-trivial because it is conducted in a multi-dimensional setting, which brings additional difficulties to overcome.
We also emphasize that ours is, to the best of our knowledge, the first result on the convergence of lattice Boltzmann schemes in presence of non-periodic boundaries for weak solutions of hyperbolic problems. 
Results are available for lattice Boltzmann schemes addressing smooth solutions (\idEst{} at least with a velocity field of class $\continuousSpace{5}$) of the incompressible Stokes \cite{junk2008convergence} and Navier-Stokes \cite{junk2009convergence} systems, when these schemes admit a so-called ``stability-structure'' and are endowed with bounce-back numerical boundary conditions.

It is worthwhile noting that, in order to obtain convergence, we show stability properties (\idEst{} estimates from the data) for the $\lebesgueSpace{\infty}$-norm and the total variation semi-norm, when the boundary conditions for the $\indexVelocity$-th discrete velocity---here assumed to enter the domain---are obtained by preparing values outside the domain: 
\begin{equation}\label{eq:bcIntro}
    \distributionFunction_{\indexVelocity, \text{out}} = \distributionFunctionLetter_{\indexVelocity}^{\atEquilibrium}(\conservedVariableDiscrete_{\text{out}}),
\end{equation} 
without much rigor, where the external state $\conservedVariableDiscrete_{\text{out}}$ in \strong{independent} of the solution inside the computational domain.
Therefore, \eqref{eq:bcIntro} can be regarded as a \strong{non-homogeneous Dirichlet} boundary condition on the $\indexVelocity$-th distribution function, and $\distributionFunctionLetter_{\indexVelocity}^{\atEquilibrium}(\conservedVariableDiscrete_{\text{out}})$ is a boundary source term.
The so-called ``strong stability'' (or GKS, for Gustafsson, Kreiss, and Sundstr\"om) stability of \strong{Dirichlet} boundary conditions for $\lebesgueSpace{2}$-stable 1D Finite Difference schemes, tackling scalar linear problems, has been established a long time ago by Goldberg and Tadmor \cite{goldberg1981scheme} (see also \cite{coulombel2011semigroup}), and holds \strong{regardless} of whether the boundary is an inflow or an outflow. 
All this to say that the kind of results that we prove on lattice Boltzmann methods with \eqref{eq:bcIntro} shares similarities with that of Goldberg and Tadmor on Finite Difference methods. 
Although we must apply \eqref{eq:bcIntro} to incoming distribution functions due to the structure of the numerical scheme, this has---as previously stressed---\strong{nothing to do} with the fact that the boundary at hand is either an inflow or an outflow\footnote{Lattice Boltzmann schemes are essentially \strong{centered} schemes \cite{kurganov2000new}.} for $\conservedVariableDiscrete$---the unknown we are eventually interested in. We indeed see that stability (and convergence) holds regardless of the inflow/outflow nature of the boundary. The worst-case scenario---already discussed in \cite{coulombel:cel-00616497}---is the presence of numerical boundary layers, which we sometimes observe in practice.
A difference compared to the work of Goldberg and Tadmor is that our results are genuinely non-linear: in their work, stability holds \strong{without limitations} on the boundary source terms, which is symptomatic of a linear problem, whereas in our context, limitations on $\conservedVariableDiscrete_{\text{out}}$ must be taken into account.

Let us now be more precise concerning the problem we tackle here.
We consider $\vectorial{\conservedVariable}\in\reals^{\numberConservationLaws}$, where $\numberConservationLaws\in\naturalsWithoutZero$ is the number of conservation laws, such that 
\begin{align}
    &\partial_{\timeVariable}\vectorial{\conservedVariable}(\timeVariable, \vectorial{\spaceVariable}) + \partial_{\xLabel}(\vectorial{\flux}_{\xLabel}(\vectorial{\conservedVariable}(\timeVariable, \vectorial{\spaceVariable}))) + \partial_{\yLabel}(\vectorial{\flux}_{\yLabel}(\vectorial{\conservedVariable}(\timeVariable, \vectorial{\spaceVariable}))) = 0, \qquad &&(\timeVariable, \vectorial{\spaceVariable})\in (0, \finalTime)\times (0, 1)^2,\label{eq:systemCons} \\
    &\vectorial{\conservedVariable}(0, \xLabel, \yLabel) = \vectorial{\conservedVariable}^{\initial}(\xLabel, \yLabel), \qquad &&\xLabel, \yLabel\in(0, 1)^2, \label{eq:initialData}\\
    &\vectorial{\conservedVariable}(\timeVariable, \xLabel = 0, \yLabel) = \tilde{\vectorial{\conservedVariable}}_{\labelWest}(\timeVariable, \yLabel), \qquad &&(\timeVariable, \yLabel)\in(0, \finalTime)\times(0, 1),\label{eq:BCW}\\
    &\vectorial{\conservedVariable}(\timeVariable, \xLabel = 1, \yLabel) = \tilde{\vectorial{\conservedVariable}}_{\labelEast}(\timeVariable, \yLabel), \qquad &&(\timeVariable, \yLabel)\in(0, \finalTime)\times(0, 1),\label{eq:BCE}\\
    &\vectorial{\conservedVariable}(\timeVariable, \xLabel, \yLabel = 0) = \tilde{\vectorial{\conservedVariable}}_{\labelSouth}(\timeVariable, \xLabel), \qquad &&(\timeVariable, \xLabel)\in(0, \finalTime)\times(0, 1),\label{eq:BCS}\\
    &\vectorial{\conservedVariable}(\timeVariable, \xLabel, \yLabel = 1) = \tilde{\vectorial{\conservedVariable}}_{\labelNorth}(\timeVariable, \xLabel), \qquad &&(\timeVariable, \xLabel)\in(0, \finalTime)\times(0, 1).\label{eq:BCN}
\end{align}
We often indicate $\vectorial{\spaceVariable} = (\xLabel, \yLabel)$ for conciseness.
Although we consider a 2D problem for the sake of presentation, everything adapts to the 1D context by ignoring the $\yLabel$-axis.
One can analoguously extend the discussion to 3D problems, at the price of more involved notations.
Notice that, in general, \eqref{eq:BCW}, \eqref{eq:BCE}, \eqref{eq:BCS}, and \eqref{eq:BCN} cannot be enforced on the entire boundary of the domain at any time.
Indeed, the trace of the solution on the boundary is only needed when and where there are incoming characteristics associated with the system of conservation laws, \confer{} \cite{serre1999systems} for example. 
To give a simple illustration of this fact, consider the scalar case $\numberConservationLaws = 1$, with $\flux_{\xLabel}(\conservedVariable) = \conservedVariable$ and $\flux_{\yLabel}(\conservedVariable) = \conservedVariable$. In this case, the solution is completely known as long as \eqref{eq:BCW} and \eqref{eq:BCS} are specified, \idEst{} the trace on the western ($\labelWest$) and southern ($\labelSouth$) walls is specified, whereas \eqref{eq:BCE} and \eqref{eq:BCN} cannot be ensured.

Several contributions and theories exist concerning the initial-boundary value problem for \eqref{eq:systemCons}.
In the scalar case $\numberConservationLaws = 1$, \cite{bardos1979first} consider the case of solution with bounded variation, whereas \cite{otto1996initial} (see also \cite{vovelle2002convergence}) consider the case of bounded measurable data.
For the case of systems ($\numberConservationLaws\geq 2$) in 1D, we mention the contributions by \cite{dubois1988boundary} and \cite{benabdallah1987problemes}. 
In the sequel, we rely on  \cite{bardos1979first}, because a theory based on bounded variation is well-suited for numerical methods, like lattice Boltzmann, which work on regular Cartesian grids.

The paper is structured as follows. 
\Cref{sec:vectorialLBM} introduces the numerical methods that this work analyzes.
Convergence in the scalar case is proved in \Cref{sec:convergenceScalar}.
Numerical experiments in 1D and 2D, both for scalar problems and systems, are showcased in \Cref{sec:numericalExperiments}.

\section{Vectorial lattice Boltzmann schemes}\label{sec:vectorialLBM}

\subsection{Time and space discretizations}

The time axis is sampled with a uniform discretization made up of points $\timeGridPoint{\indexTime} \definitionEquality\indexTime\timeStep$, with $\indexTime\in\naturals$ and $\timeStep>0$.
To reach the final time $\finalTime>0$, which is fixed once for all in what follows, we assume that there exists $\numberTimeSteps\in\naturalsWithoutZero$ such that $\numberTimeSteps\timeStep=\finalTime$.

For the space discretization, we consider $\numberCells\in\naturalsWithoutZero$ cells-per-direction and form the control volumes 
\begin{equation}\label{eq:paving}
    \cell{\vectorial{\indexSpace}}\definitionEquality (\indexSpace_{\xLabel}\spaceStep, (\indexSpace_{\xLabel}+1)\spaceStep)\times (\indexSpace_{\yLabel}\spaceStep, (\indexSpace_{\yLabel}+1)\spaceStep), \qquad \vectorial{\indexSpace} = (\indexSpace_{\xLabel}, \indexSpace_{\yLabel})\in\integerInterval{0}{\numberCells-1}^2,
\end{equation}
where the space step is given by $\spaceStep \definitionEquality 1/\numberCells$.
We hence pave the domain $(0, 1)^2$, in the sense that $[0, 1]^2 = \bigcup_{\vectorial{\indexSpace}\in\integerInterval{0}{\numberCells-1}^2}\overline{\cell{\vectorial{\indexSpace}}}$.

Since we utilize explicit schemes to handle problems where the speed of propagation of information is finite, it is natural to request that the speed of propagation by the numerical scheme, \idEst{} the ratio $\spaceStep/\timeStep$, be finite.
We thus consider $\latticeVelocity\definitionEquality\spaceStep/\timeStep>0$ fixed.

\subsection{Lattice Boltzmann schemes}

The most general lattice Boltzmann scheme that we treat in the present paper is a \lbmScheme{2}{5} scheme\footnote{Generally, a $\spatialDimensionality$-dimensional scheme featuring $\numberVelocities$ discrete velocities is called \lbmScheme{\spatialDimensionality}{\numberVelocities}.} in terms of velocity stencil, with a two-relaxation-times (TRT) relaxation model.
We do not consider velocity stencils featuring diagonal velocities (\exempliGratia{} \lbmScheme{2}{9} schemes), since they are excessively costly---in terms of storage---to address \eqref{eq:systemCons}.
Notice that, when handling systems, we could use different lattice Boltzmann schemes, coupled through equilibria, to address each conservation law, as in \cite{dubois2014simulation}. We do not consider this option just for the sake of simplifying notations.
Finally, let us notice that we discuss how to adapt our discussion to simpler velocity stencils (without zero velocity in particular, see \lbmScheme{2}{4}), to the 1D context (\lbmScheme{1}{3} and \lbmScheme{1}{2} stencils), and to simpler relaxation models, BGK for instance.
The number of discrete velocity is indicated by $\numberVelocities\in\integerInterval{2}{5}$ (and is usually $\numberVelocities = 5$ throughout the paper).

The indices for the discrete velocities belong to
\begin{equation*}
    \setVelIndexes\definitionEquality\{\labZeroVel, \labPosX, \labNegX, \labPosY, \labNegY\}.
\end{equation*}
The first one concerns the zero velocity, the second (resp. third) one the positive (resp. negative) velocity along the $\xLabel$-axis, and the fourth (resp. fifth) one the positive (resp. negative) velocity along the $\yLabel$-axis.
We use these notations \emph{in lieu} of integer indices so that the directions of propagation are clear---we do hope---at a glance.

For each discrete velocity indexed in $\setVelIndexes \ni \indexVelocity$, at each point of the time-space grid indexed by $(\indexTime, \vectorial{\indexSpace}) \in \integerInterval{0}{\numberTimeSteps}\times \integerInterval{0}{\numberCells-1}^2$, we attach a discrete distribution function $\vectorial{\distributionFunction}_{\indexVelocity, \vectorial{\indexSpace}}^{\indexTime} \in\reals^{\numberConservationLaws}$, which evolves through the following algorithm, made up of a local non-linear relaxation (sometimes called ``collision'') and a non-local linear transport phase.

\subsubsection{Relaxation}

Let $\indexTime\in\integerInterval{0}{\numberTimeSteps-1}$ and $\vectorial{\indexSpace}\in\integerInterval{0}{\numberCells-1}^2$.
Denoting $\vectorial{\conservedVariableDiscrete}_{\vectorial{\indexSpace}}^{\indexTime}\definitionEquality \sum_{\indexVelocity\in\setVelIndexes} \vectorial{\distributionFunction}_{\indexVelocity, \vectorial{\indexSpace}}^{\indexTime} \in \reals^{\numberConservationLaws}$, which has to be thought as an approximation of the solution $\vectorial{\conservedVariable}$, the relaxation reads:
\begin{align*}
    \vectorial{\distributionFunction}_{\labZeroVel, \vectorial{\indexSpace}}^{\indexTime, \collided} &= (1-\relaxationParameterSymmetric) \vectorial{\distributionFunction}_{\labZeroVel, \vectorial{\indexSpace}}^{\indexTime} + \relaxationParameterSymmetric\vectorial{\distributionFunctionLetter}_{\labZeroVel}^{\atEquilibrium}(\vectorial{\conservedVariableDiscrete}_{\vectorial{\indexSpace}}^{\indexTime}), \\
    \vectorial{\distributionFunction}_{\labPosX, \vectorial{\indexSpace}}^{\indexTime, \collided} &= (1-\tfrac{1}{2}(\relaxationParameterSymmetric+\relaxationParameterAntiSymmetric)) \vectorial{\distributionFunction}_{\labPosX, \vectorial{\indexSpace}}^{\indexTime} + \tfrac{1}{2}(\relaxationParameterSymmetric+\relaxationParameterAntiSymmetric)\vectorial{\distributionFunctionLetter}_{\labPosX}^{\atEquilibrium}(\vectorial{\conservedVariableDiscrete}_{\vectorial{\indexSpace}}^{\indexTime}) + \tfrac{1}{2}(\relaxationParameterSymmetric-\relaxationParameterAntiSymmetric) (\vectorial{\distributionFunctionLetter}_{\labNegX}^{\atEquilibrium}(\vectorial{\conservedVariableDiscrete}_{\vectorial{\indexSpace}}^{\indexTime}) - \vectorial{\distributionFunction}_{\labNegX, \vectorial{\indexSpace}}^{\indexTime}), \\
    \vectorial{\distributionFunction}_{\labNegX, \vectorial{\indexSpace}}^{\indexTime, \collided} &= (1-\tfrac{1}{2}(\relaxationParameterSymmetric+\relaxationParameterAntiSymmetric)) \vectorial{\distributionFunction}_{\labNegX, \vectorial{\indexSpace}}^{\indexTime} + \tfrac{1}{2}(\relaxationParameterSymmetric+\relaxationParameterAntiSymmetric)\vectorial{\distributionFunctionLetter}_{\labNegX}^{\atEquilibrium}(\vectorial{\conservedVariableDiscrete}_{\vectorial{\indexSpace}}^{\indexTime}) + \tfrac{1}{2}(\relaxationParameterSymmetric-\relaxationParameterAntiSymmetric) (\vectorial{\distributionFunctionLetter}_{\labPosX}^{\atEquilibrium}(\vectorial{\conservedVariableDiscrete}_{\vectorial{\indexSpace}}^{\indexTime}) - \vectorial{\distributionFunction}_{\labPosX, \vectorial{\indexSpace}}^{\indexTime}),\\
    \vectorial{\distributionFunction}_{\labPosY, \vectorial{\indexSpace}}^{\indexTime, \collided} &= (1-\tfrac{1}{2}(\relaxationParameterSymmetric+\relaxationParameterAntiSymmetric)) \vectorial{\distributionFunction}_{\labPosY, \vectorial{\indexSpace}}^{\indexTime} + \tfrac{1}{2}(\relaxationParameterSymmetric+\relaxationParameterAntiSymmetric)\vectorial{\distributionFunctionLetter}_{\labPosY}^{\atEquilibrium}(\vectorial{\conservedVariableDiscrete}_{\vectorial{\indexSpace}}^{\indexTime}) + \tfrac{1}{2}(\relaxationParameterSymmetric-\relaxationParameterAntiSymmetric) (\vectorial{\distributionFunctionLetter}_{\labNegY}^{\atEquilibrium}(\vectorial{\conservedVariableDiscrete}_{\vectorial{\indexSpace}}^{\indexTime}) - \vectorial{\distributionFunction}_{\labNegY, \vectorial{\indexSpace}}^{\indexTime}), \\
    \vectorial{\distributionFunction}_{\labNegY, \vectorial{\indexSpace}}^{\indexTime, \collided} &= (1-\tfrac{1}{2}(\relaxationParameterSymmetric+\relaxationParameterAntiSymmetric)) \vectorial{\distributionFunction}_{\labNegY, \vectorial{\indexSpace}}^{\indexTime} + \tfrac{1}{2}(\relaxationParameterSymmetric+\relaxationParameterAntiSymmetric)\vectorial{\distributionFunctionLetter}_{\labNegY}^{\atEquilibrium}(\vectorial{\conservedVariableDiscrete}_{\vectorial{\indexSpace}}^{\indexTime}) + \tfrac{1}{2}(\relaxationParameterSymmetric-\relaxationParameterAntiSymmetric) (\vectorial{\distributionFunctionLetter}_{\labPosY}^{\atEquilibrium}(\vectorial{\conservedVariableDiscrete}_{\vectorial{\indexSpace}}^{\indexTime}) - \vectorial{\distributionFunction}_{\labPosY, \vectorial{\indexSpace}}^{\indexTime}).
\end{align*}
The relaxation parameters $\relaxationParameterSymmetric$ and $\relaxationParameterAntiSymmetric$ belong to $(0, 2]$ and can be chosen freely in principle.
The equilibria must fulfill the consistency constraints
\begin{equation}\label{eq:consistency}
    \sum_{\indexVelocity\in\setVelIndexes} \vectorial{\distributionFunctionLetter}_{\indexVelocity}^{\atEquilibrium}(\vectorial{\conservedVariableDiscrete}) = \vectorial{\conservedVariableDiscrete}, \qquad \latticeVelocity (\vectorial{\distributionFunctionLetter}_{\labPosX}^{\atEquilibrium}(\vectorial{\conservedVariableDiscrete}) - \vectorial{\distributionFunctionLetter}_{\labNegX}^{\atEquilibrium}(\vectorial{\conservedVariableDiscrete})) = \vectorial{\flux}_{\xLabel}(\vectorial{\conservedVariableDiscrete}), \qquad \latticeVelocity (\vectorial{\distributionFunctionLetter}_{\labPosY}^{\atEquilibrium}(\vectorial{\conservedVariableDiscrete}) - \vectorial{\distributionFunctionLetter}_{\labNegY}^{\atEquilibrium}(\vectorial{\conservedVariableDiscrete})) = \vectorial{\flux}_{\yLabel}(\vectorial{\conservedVariableDiscrete}).
\end{equation}
For some notational ease, we sometimes write $\vectorial{\distributionFunction}_{\indexVelocity, \vectorial{\indexSpace}}^{\indexTime, \collided} = \vectorial{\collisionOperator}_{\indexVelocity}(\vectorial{\distributionFunction}_{\labZeroVel, \vectorial{\indexSpace}}^{\indexTime}, \dots, \vectorial{\distributionFunction}_{\labNegY, \vectorial{\indexSpace}}^{\indexTime})$ or $\vectorial{\distributionFunction}_{\indexVelocity, \vectorial{\indexSpace}}^{\indexTime, \collided} = \vectorial{\collisionOperator}_{\indexVelocity}(\vectorial{\distributionFunction}_{\vectorial{\indexSpace}}^{\indexTime})$ for $\indexVelocity\in\setVelIndexes$.

\begin{remark}[BGK relaxation]
    The BGK relaxation operator is obtained by simply setting $\relaxationParameterSymmetric = \relaxationParameterAntiSymmetric$ in the previous expressions.
\end{remark}

\subsubsection{Transport}

After the relaxation, the transport phase reads, for $\vectorial{\indexSpace}\in\integerInterval{0}{\numberCells-1}^2$
\begin{align*}
    \vectorial{\distributionFunction}_{\labZeroVel, \vectorial{\indexSpace}}^{\indexTime+1} =\vectorial{\distributionFunction}_{\labZeroVel, \vectorial{\indexSpace}}^{\indexTime, \collided}, \qquad 
    \vectorial{\distributionFunction}_{\labPosX, \vectorial{\indexSpace}}^{\indexTime+1} &=\vectorial{\distributionFunction}_{\labPosX, \vectorial{\indexSpace}-\canonicalBasisVector{\xLabel}}^{\indexTime, \collided}, \qquad 
    \vectorial{\distributionFunction}_{\labNegX, \vectorial{\indexSpace}}^{\indexTime+1} =\vectorial{\distributionFunction}_{\labNegX, \vectorial{\indexSpace}+\canonicalBasisVector{\xLabel}}^{\indexTime, \collided},\\
    \vectorial{\distributionFunction}_{\labPosY, \vectorial{\indexSpace}}^{\indexTime+1} &=\vectorial{\distributionFunction}_{\labPosY, \vectorial{\indexSpace}-\canonicalBasisVector{\yLabel}}^{\indexTime, \collided}, \qquad
    \vectorial{\distributionFunction}_{\labNegY, \vectorial{\indexSpace}}^{\indexTime+1} =\vectorial{\distributionFunction}_{\labNegY, \vectorial{\indexSpace}+\canonicalBasisVector{\yLabel}}^{\indexTime, \collided}.
\end{align*}

Notice that when $\vectorial{\indexSpace}$ is close to the boundary of $\integerInterval{0}{\numberCells-1}^2$, the shifted coordinate $\vectorial{\indexSpace}\mp\canonicalBasisVector{\xLabel}$ (and the same along $\yLabel$) falls outside $\integerInterval{0}{\numberCells-1}^2$.
This calls for the definition of post-relaxation distributions functions, at least for some discrete velocities, defined on ghost cells indexed by 
\begin{align*}
    (-1, \indexSpace_{\yLabel}) \quad  &\text{with}\quad  \indexSpace_{\yLabel}\in\integerInterval{0}{\numberCells-1}, \qquad \text{for} \quad \vectorial{\distributionFunction}_{\labPosX}\qquad \text{(western boundary }\labelWest).\\
    (\numberCells, \indexSpace_{\yLabel}) \quad  &\text{with}\quad  \indexSpace_{\yLabel}\in\integerInterval{0}{\numberCells-1}, \qquad \text{for} \quad \vectorial{\distributionFunction}_{\labNegX}\qquad \text{(eastern boundary }\labelEast).\\
    (\indexSpace_{\xLabel}, -1) \quad  &\text{with}\quad  \indexSpace_{\xLabel}\in\integerInterval{0}{\numberCells-1}, \qquad \text{for} \quad \vectorial{\distributionFunction}_{\labPosY}\qquad \text{(southern boundary }\labelSouth).\\
    (\indexSpace_{\xLabel}, \numberCells) \quad  &\text{with}\quad  \indexSpace_{\xLabel}\in\integerInterval{0}{\numberCells-1}, \qquad \text{for} \quad \vectorial{\distributionFunction}_{\labNegY}\qquad \text{(northern boundary }\labelNorth).
\end{align*}
This topic---which is central in our work---is discussed in the dedicated \Cref{sec:boundaryConditions}.

\subsubsection{Initialization}

Both to handle initialization and boundary conditions, equilibria are utilized as a way of constructing kinetic information in $\reals^{\numberConservationLaws\times\numberVelocities}$ from macroscopic one in $\reals^{\numberConservationLaws}$.
For the initialization, we therefore take
\begin{equation}\label{eq:initialization}
    \vectorial{\distributionFunction}_{\indexVelocity, \vectorial{\indexSpace}}^{0} = \vectorial{\distributionFunctionLetter}_{\indexVelocity}^{\atEquilibrium}\Bigl (\dashint_{\cell{\vectorial{\indexSpace}}}\vectorial{\conservedVariable}^{\initial}(\vectorial{\spaceVariable})\differential\vectorial{\spaceVariable}\Bigr ), \qquad \indexVelocity\in\setVelIndexes, \quad \vectorial{\indexSpace}\in\integerInterval{0}{\numberCells-1}^2.
\end{equation}

\subsubsection{Boundary conditions}\label{sec:boundaryConditions}

As done for the initialization, we propose to fill missing data outside the boundary to be streamed inside using equilibria:
\begin{align}
    \vectorial{\distributionFunction}_{\labPosX, -1, \indexSpace_{\yLabel}}^{\indexTime, \collided} = \vectorial{\distributionFunctionLetter}_{\labPosX}^{\atEquilibrium}(\boundaryDatumWestVectorial{\indexSpace_{\yLabel}}^{\indexTime}), \qquad \vectorial{\distributionFunction}_{\labNegX, \numberCells, \indexSpace_{\yLabel}}^{\indexTime, \collided} = \vectorial{\distributionFunctionLetter}_{\labNegX}^{\atEquilibrium}(\boundaryDatumEastVectorial{\indexSpace_{\yLabel}}^{\indexTime}), \qquad \indexSpace_{\yLabel}\in\integerInterval{0}{\numberCells-1}, \label{eq:boundaryEquilibrium1}\\
    \vectorial{\distributionFunction}_{\labPosY, \indexSpace_{\xLabel}, -1}^{\indexTime, \collided} = \vectorial{\distributionFunctionLetter}_{\labPosY}^{\atEquilibrium}(\boundaryDatumSouthVectorial{\indexSpace_{\xLabel}}^{\indexTime}), \qquad \vectorial{\distributionFunction}_{\labNegY, \indexSpace_{\xLabel}, \numberCells}^{\indexTime, \collided} = \vectorial{\distributionFunctionLetter}_{\labNegY}^{\atEquilibrium}(\boundaryDatumNorthVectorial{\indexSpace_{\xLabel}}^{\indexTime}), \qquad \indexSpace_{\xLabel}\in\integerInterval{0}{\numberCells-1}, \nonumber
\end{align}
where the boundary data $\boundaryDatumVectorial{\alpha}{\indexSpace}^{\indexTime}$ for $\alpha=\labelWest, \labelEast, \labelSouth, \labelNorth$ can be created in several ways according to the specific effect that one wants to obtain on the conserved moment $\vectorial{\conservedVariableDiscrete}_{\vectorial{\indexSpace}}^{\indexTime}\definitionEquality \sum_{\indexVelocity\in\setVelIndexes} \vectorial{\distributionFunction}_{\indexVelocity, \vectorial{\indexSpace}}^{\indexTime}$  close to the boundary. Several choices are illustrated in the numerical experiments of \Cref{sec:numericalExperiments}. 
For instance, the values of $\boundaryDatumVectorial{\alpha}{\indexSpace}^{\indexTime}$ can depend, if one is not interested in proving convergence, on $\vectorial{\conservedVariableDiscrete}_{\vectorial{\indexSpace}}^{\indexTime}$ inside the domain, which is know once enforcing boundary conditions.
The choice that we adopt in \Cref{sec:convergenceScalar}, aiming at analyzing convergence, is 
\begin{equation}\label{eq:boundaryEquilibrium2}
    \boundaryDatumVectorial{\alpha}{\indexSpace}^{\indexTime} = \dashint_{\timeGridPoint{\indexTime}}^{\timeGridPoint{\indexTime + 1}}\reduceSpaceDoubleInt\dashint_{\indexSpace\spaceStep}^{(\indexSpace+1)\spaceStep}\tilde{\vectorial{\conservedVariable}}_{\alpha}(\timeVariable, \spaceVariable)\differential\spaceVariable\differential\timeVariable,
\end{equation}
where the functions $\tilde{\vectorial{\conservedVariable}}_{\alpha}$ are those in \eqref{eq:BCW}, \eqref{eq:BCE}, \eqref{eq:BCS}, and \eqref{eq:BCN}.

\begin{remark}[Simpler velocity stencils]
    Still remaining in 2D, if one wants to consider a \lbmScheme{2}{4} scheme, it is enough to set $\vectorial{\distributionFunctionLetter}^{\atEquilibrium}_{\labZeroVel} \equiv \vectorial{0}$.
    Since we start at equilibrium, this guarantees that $\vectorial{\distributionFunction}_{\labZeroVel, \vectorial{\indexSpace}}^{\indexTime}\equiv \vectorial{0}$, which hence plays no role.

    In the same fashion, to go towards 1D problems, one obtains a \lbmScheme{1}{3} scheme taking $\vectorial{\distributionFunctionLetter}^{\atEquilibrium}_{\labPosY} = \vectorial{\distributionFunctionLetter}^{\atEquilibrium}_{\labNegY} \equiv \vectorial{0}$, and a \lbmScheme{1}{2} scheme by also adding $\vectorial{\distributionFunctionLetter}^{\atEquilibrium}_{\labZeroVel} \equiv \vectorial{0}$.
\end{remark}

\section{Convergence in the scalar case}\label{sec:convergenceScalar}

We now focus on the convergence of the numerical schemes in the case of $\numberConservationLaws = 1$.
The proof is based on showing, as in \cite{crandall1980monotone}, that the numerical solution is bounded in $\lebesgueSpace{1}\cap\boundedVariationSpace$, thus precompact in $L^1_{\text{loc}}$. Proving equicontinuity in time, this allows to extract a subsequence of the numerical solution converging in $\lebesgueSpace{1}$ in space, uniformly in time.
We finally prove that the limit is the solution of the PDE we are looking at.

\subsection{Weak entropy solution}

We start by defining a notion of solution for the PDE at hand. Hence, we adapt the definition of weak entropy solution given in \cite{bardos1979first}.
    \begin{definition}[Weak entropy solution]\label{def:weakSolution}
        A function $\conservedVariable\in\boundedVariationSpace((0, \finalTime)\times (0, 1)^2)$ is said to be a weak entropy solution to \eqref{eq:systemCons}, \eqref{eq:initialData}, \eqref{eq:BCW}, \eqref{eq:BCE}, \eqref{eq:BCS}, and \eqref{eq:BCN}, if for all $\krushkovParameter\in\reals$ and for all $\testFunction\in\smoothFunctionsSpace([0, \finalTime )\times[0, 1]^2)$ such that $\testFunction\geq 0$, we have 
        \begin{align*}
            \int_0^{\finalTime}\reduceSpaceDoubleInt \int_{(0, 1)^2} |\conservedVariable(\timeVariable, \vectorial{\spaceVariable})-\krushkovParameter| + &\sign(\conservedVariable(\timeVariable, \vectorial{\spaceVariable})-\krushkovParameter)((\flux_{\xLabel}(\conservedVariable(\timeVariable, \vectorial{\spaceVariable}))-\flux_{\xLabel}(\krushkovParameter))\partial_{\xLabel}\testFunction(\timeVariable, \vectorial{\spaceVariable})
            +(\flux_{\yLabel}(\conservedVariable(\timeVariable, \vectorial{\spaceVariable}))-\flux_{\yLabel}(\krushkovParameter))\partial_{\yLabel}\testFunction(\timeVariable, \vectorial{\spaceVariable}))\differential\vectorial{\spaceVariable}\differential\timeVariable \\
            +\int_0^{\finalTime}\reduceSpaceDoubleInt \int_{0}^1 &\sign(\boundaryFunction{\labelWest}(\timeVariable, \yLabel)-\krushkovParameter)(\flux_{\xLabel}(\traceOperator_{\labelWest}(\conservedVariable)(\timeVariable, \yLabel))-\flux_{\xLabel}(\krushkovParameter))\testFunction(\timeVariable, 0, \yLabel)\differential\yLabel\differential\timeVariable\\
            -\int_0^{\finalTime}\reduceSpaceDoubleInt \int_{0}^1 &\sign(\boundaryFunction{\labelEast}(\timeVariable, \yLabel)-\krushkovParameter)(\flux_{\xLabel}(\traceOperator_{\labelEast}(\conservedVariable)(\timeVariable, \yLabel))-\flux_{\xLabel}(\krushkovParameter))\testFunction(\timeVariable, 1, \yLabel)\differential\yLabel\differential\timeVariable\\
            +\int_0^{\finalTime}\reduceSpaceDoubleInt \int_{0}^1 &\sign(\boundaryFunction{\labelSouth}(\timeVariable, \xLabel)-\krushkovParameter)(\flux_{\yLabel}(\traceOperator_{\labelSouth}(\conservedVariable)(\timeVariable, \xLabel))-\flux_{\yLabel}(\krushkovParameter))\testFunction(\timeVariable, \xLabel, 0)\differential\xLabel\differential\timeVariable\\
            -\int_0^{\finalTime}\reduceSpaceDoubleInt \int_{0}^1 &\sign(\boundaryFunction{\labelNorth}(\timeVariable, \xLabel)-\krushkovParameter)(\flux_{\yLabel}(\traceOperator_{\labelNorth}(\conservedVariable)(\timeVariable, \xLabel))-\flux_{\yLabel}(\krushkovParameter))\testFunction(\timeVariable, \xLabel, 1)\differential\xLabel\differential\timeVariable + \int_{(0, 1)^2}|\conservedVariable^{\initial}(\vectorial{\spaceVariable})-\krushkovParameter|\testFunction(0, \vectorial{\spaceVariable})\differential\vectorial{\spaceVariable} \geq 0,
        \end{align*}
        where $\traceOperator_{\alpha}$ denotes the trace on one of the four boundaries, which is well-defined for functions with bounded variation.
    \end{definition}
Existence and uniqueness of the solution in \Cref{def:weakSolution} is shown under the assumptions that $\flux_{\xLabel}, \flux_{\yLabel}\in\continuousSpace{1}(\reals)$,\footnote{In \cite{bardos1979first}, the authors request $\flux_{\xLabel}, \flux_{\yLabel}\in\continuousSpace{2}$ due to the fact that dependence of the fluxes on time and space is also allowed.} $\conservedVariable^{\initial}\in\continuousSpace{2}([0, 1]^2)$, and $(\timeVariable, \vectorial{\spaceVariable}) \mapsto \boundaryFunction{\labelWest}(\timeVariable, \yLabel)\indicatorFunction{\xLabel=0} + \boundaryFunction{\labelEast}(\timeVariable, \yLabel)\indicatorFunction{\xLabel=1} + \boundaryFunction{\labelSouth}(\timeVariable, \xLabel)\indicatorFunction{\yLabel=0} + \boundaryFunction{\labelNorth}(\timeVariable, \xLabel)\indicatorFunction{\yLabel=1}$ be $\continuousSpace{2}$ over $(0, \finalTime)\times (\{0, 1\}\times [0, 1]\cup[0, 1]\times \{0, 1\})$.

For our proof of convergence, we assume the following on fluxes and initial-boundary data.
\begin{assumption}[Smoothness of the fluxes]
    We assume that $\flux_{\xLabel}, \flux_{\yLabel}\in\continuousSpace{1}(\reals)$, $\flux_{\xLabel}(0)= 0$, and $\flux_{\yLabel}(0)= 0$.
\end{assumption}
\begin{assumption}[Smoothness of the initial and boundary data]
    We assume the following regularity on the initial and boundary data:
    \begin{align}
        \conservedVariable^{\initial}&\in \lebesgueSpace{\infty}((0, 1)^2)\cap \boundedVariationSpace((0, 1)^2)\cap \lebesgueSpace{1}((0, 1)^2),\label{eq:smoothnessInitial} \\
        \boundaryFunction{\alpha} &\in \lebesgueSpace{\infty}((0, \finalTime)\times(0, 1))\cap \boundedVariationSpace((0, \finalTime)\times(0, 1))\cap \lebesgueSpace{1}((0, \finalTime)\times(0, 1)),\qquad \text{for}\quad \alpha=\labelWest, \labelEast, \labelSouth, \labelNorth.\label{eq:smoothnessBoundary}
    \end{align}
\end{assumption}
Notice that requesting $\lebesgueSpace{1}$-smoothness comes from the fact that we develop proofs on a quarter of plane in space (\idEst{} and unbounded domain). On bounded domains, $\lebesgueSpace{\infty}((0, 1)^2)$ and $\lebesgueSpace{\infty}((0, \finalTime)\times(0, 1))$-smoothness makes them pointless.

\subsection{Various useful results on bounded variation functions}\label{sec:usefulBV}

Before proceeding, we collect several results that are useful in what follows and which are proved in \Cref{app:proofsBV}.
They mainly connect the notion of total variation in 2D with estimates along the axes, namely parallel to the boundaries of the domain at hand, or the relation between the total variation of piecewise constant reconstuctions and that of the original functions.

Each time a norm or a semi-norm of a vectorial quantity is considered, we implicitly utilize the $\petitLebesgueSpace{1}$-norm for the finite dimensional vector space.
Moreover, let $\Omega\subset \reals^{\spatialDimensionality}$ be an open subset, and $\conservedVariable\in\lebesgueSpace{1}(\Omega)$, either scalar or vector valued.
We indicate the total variation of $\conservedVariable$ over $\Omega$, defined in the standard fashion \cite{ambrosio2000functions}, by $\totalVariation{\conservedVariable}{\Omega}$.

We first start by a result, see \cite[Theorem 1.3]{chambolle2010introduction} on the approximation of $\boundedVariationSpace$-functions by smooth functions.
\begin{theorem}[Generalized Meyers–Serrin's]\label{thm:approxThmBV}
    Let $\Omega\subset \reals^{\spatialDimensionality}$ be an open set and $\conservedVariable\in\lebesgueSpace{1}(\Omega)\cap \boundedVariationSpace(\Omega)$.
    Then, there exists a sequence $(\conservedVariable_k)_{k\in\naturals}$ in $\cKClass{\infty}(\Omega)\cap \sobolevSpace{1}{1}(\Omega)$  such that 
    \begin{equation*}
        \conservedVariable_k\to \conservedVariable \quad \text{in}\quad \lebesgueSpace{1}(\Omega), \qquad \text{and}\qquad \totalVariation{\conservedVariable_k}{\Omega} = \int_{\Omega}|\nabla \conservedVariable_k(\vectorial{\spaceVariable})|\differential\vectorial{\spaceVariable}\to\totalVariation{\conservedVariable}{\Omega}.
    \end{equation*}
\end{theorem}

We now provide the link between the total variation of piecewise constant reconstuctions and that of the original functions.
We state the result for a function of $\xLabel$ and $\yLabel$ on $(\realsPositive)^2$ where control volumes are of the kind given by \eqref{eq:paving}.
Similar results hold for functions of one variable, functions of two variable where one is the time, or on bounded domains, upon adapting the control volumes accordingly.
\begin{proposition}[Total variation of a piecewise-constant approximation]\label{prop:totalVariationPiecewiseApprox}
    Let $\conservedVariable\in\lebesgueSpace{1}((\realsPositive)^2)\cap \boundedVariationSpace((\realsPositive)^2)$. 
    Let  $\conservedVariable_{\discreteMark}$ be a piecewise constant approximation of $\conservedVariable$ given by
    \begin{equation*}
        \conservedVariable_{\discreteMark}(\xLabel, \yLabel) \definitionEquality \sum_{\vectorial{\indexSpace}\in\naturals^2} \Bigl ( \dashint_{\cell{\vectorial{\indexSpace}}}\conservedVariable(\tilde{\xLabel}, \tilde{\yLabel})\differential\tilde{\xLabel}\differential \tilde{\yLabel} \Bigr ) \indicatorFunction{\cell{\vectorial{\indexSpace}}}(\xLabel, \yLabel).
    \end{equation*}
    Then 
    \begin{equation*}
        \totalVariation{\conservedVariable_{\discreteMark}}{(\realsPositive)^2}\leq \totalVariation{\conservedVariable}{(\realsPositive)^2}.
    \end{equation*}
\end{proposition}

\begin{proposition}[Total variation of the absolute value]\label{prop:totalVariationObsoluteValue}
    Let $\Omega\subset \reals^{\spatialDimensionality}$ be an open set and $\conservedVariable\in\lebesgueSpace{\infty}(\Omega)\cap \lebesgueSpace{1}(\Omega)\cap \boundedVariationSpace(\Omega)$.
    Then $|\conservedVariable|\in\lebesgueSpace{\infty}(\Omega)\cap \lebesgueSpace{1}(\Omega)\cap \boundedVariationSpace(\Omega)$ and moreover
    \begin{equation*}
        \totalVariation{|\conservedVariable|}{\Omega}\leq \totalVariation{\conservedVariable}{\Omega}.
    \end{equation*}
\end{proposition}

The following results are given for functions of $\xLabel$ and $\yLabel$, but they also apply when the function is of the time $\timeVariable$ and either $\xLabel$ or $\yLabel$, which is the case when boundaries are present.
\begin{definition}[2D total variation along one axis]
    Let $\Omega\subset \reals^{2}$ be an open set and $\conservedVariable\in\lebesgueSpace{1}_{\text{loc}}(\Omega)$. We define:
    \begin{align*}
        \totalVariationAlongAxis{\conservedVariable}{\Omega}{\xLabel}&\definitionEquality \sup \Bigl \{ \int_{\Omega} \conservedVariable(\vectorial{\spaceVariable}) \partial_{\xLabel}\testFunction(\vectorial{\spaceVariable}) \differential\vectorial{\spaceVariable}, \qquad \testFunction\in\smoothFunctionsSpaceWithReg{1}(\Omega), \quad \lVert\testFunction\rVert_{\lebesgueSpace{\infty}(\Omega)}\leq 1\Bigr \}, \\
        \totalVariationAlongAxis{\conservedVariable}{\Omega}{\yLabel}&\definitionEquality \sup \Bigl \{ \int_{\Omega} \conservedVariable(\vectorial{\spaceVariable}) \partial_{\yLabel}\testFunction(\vectorial{\spaceVariable}) \differential\vectorial{\spaceVariable}, \qquad \testFunction\in\smoothFunctionsSpaceWithReg{1}(\Omega), \quad \lVert\testFunction\rVert_{\lebesgueSpace{\infty}(\Omega)}\leq 1\Bigr \}.
    \end{align*}
\end{definition}
\begin{remark}\label{rem:1Dvs2D}
    One easily sees that for $\conservedVariable\in\lebesgueSpace{1}(\Omega)\cap \boundedVariationSpace(\Omega)$, we have that $\totalVariationAlongAxis{\conservedVariable}{\Omega}{\xLabel}\leq \totalVariation{\conservedVariable}{\Omega}$ and $\totalVariationAlongAxis{\conservedVariable}{\Omega}{\yLabel}\leq \totalVariation{\conservedVariable}{\Omega}$.
\end{remark}
We have the result from \cite[Theorem 3.103]{ambrosio2000functions}, which links the 2D total variation along a given axis with the integral of the total variation of the restriction of the function to the other axis.
\begin{theorem}\label{thm:2DtotalAlongAxis}
    Let $\Omega = (a_{\xLabel}, b_{\xLabel})\times (a_{\yLabel}, b_{\yLabel})$ with $a_{\xLabel}, b_{\xLabel}, a_{\yLabel}, b_{\yLabel}\in\reals\cup \{\pm\infty\}$ and $a_{\xLabel}< b_{\xLabel}$, $a_{\yLabel}< b_{\yLabel}$, and $\conservedVariable\in\lebesgueSpace{1}_{\text{loc}}(\Omega)$.
    Then we have that 
    \begin{align*}
        \totalVariationAlongAxis{\conservedVariable}{\Omega}{\xLabel}&=\int_{a_{\yLabel}}^{b_{\yLabel}}\totalVariation{\conservedVariable(\cdot, \yLabel)}{(a_{\xLabel}, b_{\xLabel})}\differential\yLabel, \\
        \totalVariationAlongAxis{\conservedVariable}{\Omega}{\yLabel}&=\int_{a_{\xLabel}}^{b_{\xLabel}}\totalVariation{\conservedVariable(\xLabel, \cdot)}{(a_{\yLabel}, b_{\yLabel})}\differential\xLabel.
    \end{align*}
\end{theorem}

From now on in this section, we consider $\Omega = (a_{\xLabel}, b_{\xLabel})\times (a_{\yLabel}, b_{\yLabel})$ with $a_{\xLabel}, b_{\xLabel}, a_{\yLabel}, b_{\yLabel}\in\reals\cup \{\pm\infty\}$ and $a_{\xLabel}< b_{\xLabel}$, $a_{\yLabel}< b_{\yLabel}$, and $\conservedVariable\in\lebesgueSpace{1}(\Omega)\cap \boundedVariationSpace(\Omega)$.
We define, for almost every $\xLabel\in (a_{\xLabel}, b_{\xLabel})$, the function 
\begin{equation}\label{eq:definitionRestrictionIntegralOneAxis}
    v(\xLabel)\definitionEquality \int_{a_{\yLabel}}^{b_{\yLabel}} \conservedVariable(\xLabel, \yLabel)\differential\yLabel.
\end{equation}
We first would like to ensure that $v$ has finite total variation.
\begin{proposition}\label{prop:restrictionTotalVariation}
    For $v$ defined by \eqref{eq:definitionRestrictionIntegralOneAxis}, we have that $v\in\lebesgueSpace{1}(a_{\xLabel}, b_{\xLabel})\cap \boundedVariationSpace(a_{\xLabel}, b_{\xLabel})$.
    Moreover, we have 
    \begin{equation*}
        \totalVariation{v}{(a_{\xLabel}, b_{\xLabel})}\leq \int_{a_{\yLabel}}^{b_{\yLabel}}\totalVariation{\conservedVariable(\cdot, \yLabel)}{(a_{\xLabel}, b_{\xLabel})}\differential\yLabel = \totalVariationAlongAxis{\conservedVariable}{\Omega}{\xLabel}\leq \totalVariation{\conservedVariable}{\Omega}<+\infty.
    \end{equation*}
\end{proposition}

We would also like to ensure that $v$ is bounded.
This comes from the particular nature of the 1D case.
\begin{proposition}\label{prop:restrictionLinF}
    For $v$ defined by \eqref{eq:definitionRestrictionIntegralOneAxis}, we have that $v\in\lebesgueSpace{\infty}(a_{\xLabel}, b_{\xLabel})$.
\end{proposition}
In what follows, we denote $\lVert v \rVert_{\lebesgueSpace{\infty}(a_{\xLabel}, b_{\xLabel})} = \sup_{\xLabel\in(a_{\xLabel}, b_{\xLabel})}|v(\xLabel)|$, that is, a bounded representative of $v$ is chosen.

We finish on a result that we exploit several times in the sequel and which justifies all previous discussions.
It says that 2D integrals on (vertical or horizontal) strips, potentially unbounded in one direction, and of width $\eta$, are of order $\bigO{\eta}$\footnote{We shall in general consider $\eta = \spaceStep$ or $\eta=\timeStep$.}.
\begin{proposition}[Magnitude of integrals on strips]\label{prop:VerticalIntegrationBV}
    Let $v$ be defined by \eqref{eq:definitionRestrictionIntegralOneAxis}. 
    Let $c_{\xLabel}\in\reals$ and $\eta>0$ such that $a_{\xLabel}\leq c_{\xLabel}<c_{\xLabel}+\eta<b_{\xLabel}$.
    Then 
    \begin{align*}
        \Bigl  | \int_{c_{\xLabel}}^{c_{\xLabel}+\eta}v(\xLabel)\differential\xLabel\Bigr  | \leq \eta \lVert v \rVert_{\lebesgueSpace{\infty}(a_{\xLabel}, b_{\xLabel})}.
    \end{align*}
\end{proposition}

\subsection{Monotone relaxation}

As mentioned in the introduction, since our problem is non-linear, desirable properties cannot hold for any data regardless of their magnitude.
This magnitude, both for initial and boundary data, is measured by 
\begin{equation}\label{eq:maximumData}
    \maximumInitialDatum\definitionEquality\max(\lVert\conservedVariable^{\initial}\rVert_{\lebesgueSpace{\infty}((0, 1)^2)}, \max_{\alpha=\labelWest, \labelEast, \labelSouth, \labelNorth}\lVert\boundaryFunction{\alpha}\rVert_{\lebesgueSpace{\infty}((0, \finalTime)\times (0, 1))}).
\end{equation}
The crucial desirable property that we utilize is the \strong{monotonicity} of the relaxation operator, within the magnitude defined by $\maximumInitialDatum$.
In the sequel of \Cref{sec:convergenceScalar}, we assume that the equilibria have the following form, see \cite{aregba2025monotonicity}.
\begin{assumption}[On the form of the equilibria]
    We assume that
    \begin{align*}
        \distributionFunctionLetter_{\labZeroVel}^{\atEquilibrium}(\conservedVariableDiscrete) = (1-2 \equilibriumCoefficientLinear_{\xLabel} - 2\equilibriumCoefficientLinear_{\yLabel})\conservedVariableDiscrete, \qquad &\distributionFunctionLetter_{\labPosX}^{\atEquilibrium}(\conservedVariableDiscrete) = \equilibriumCoefficientLinear_{\xLabel}\conservedVariableDiscrete + \frac{1}{2\latticeVelocity}\flux_{\xLabel}(\conservedVariableDiscrete), \qquad \distributionFunctionLetter_{\labNegX}^{\atEquilibrium}(\conservedVariableDiscrete) = \equilibriumCoefficientLinear_{\xLabel}\conservedVariableDiscrete - \frac{1}{2\latticeVelocity}\flux_{\xLabel}(\conservedVariableDiscrete), \\
        &\distributionFunctionLetter_{\labPosY}^{\atEquilibrium}(\conservedVariableDiscrete) = \equilibriumCoefficientLinear_{\yLabel}\conservedVariableDiscrete + \frac{1}{2\latticeVelocity}\flux_{\yLabel}(\conservedVariableDiscrete), \qquad \distributionFunctionLetter_{\labNegY}^{\atEquilibrium}(\conservedVariableDiscrete) = \equilibriumCoefficientLinear_{\yLabel}\conservedVariableDiscrete - \frac{1}{2\latticeVelocity}\flux_{\yLabel}(\conservedVariableDiscrete),
    \end{align*}
    where $\equilibriumCoefficientLinear_{\xLabel}, \equilibriumCoefficientLinear_{\yLabel}\in\reals$ are free.
\end{assumption}

\begin{remark}[Simpler velocity stencils]
    Simpler schemes compared to the \lbmScheme{2}{5} scheme under scrutiny come from the following choices:
    \begin{align*}
        &\text{\lbmScheme{2}{4}}\qquad : \qquad 1-2 \equilibriumCoefficientLinear_{\xLabel} - 2\equilibriumCoefficientLinear_{\yLabel} = 0, \\
        &\text{\lbmScheme{1}{3}}\qquad : \qquad \equilibriumCoefficientLinear_{\yLabel} = 0, \quad \flux_{\yLabel} \equiv 0,\\
        &\text{\lbmScheme{1}{2}}\qquad : \qquad \equilibriumCoefficientLinear_{\xLabel} = \tfrac{1}{2}, \quad \equilibriumCoefficientLinear_{\yLabel} = 0, \quad \flux_{\yLabel} \equiv 0.
    \end{align*}
\end{remark}

Requesting that the relaxation operator is monotone non decreasing with respect to each of its $\numberVelocities = 5$ arguments gives the following constraints, \confer{} \cite[Proposition 3.3]{aregba2025monotonicity}:
\begin{equation}\label{eq:mononotonicityConditionsD2Q5}
    \text{\lbmScheme{2}{5}}\qquad :\qquad
    \begin{cases}
        &\relaxationParameterSymmetric (1-2 \equilibriumCoefficientLinear_{\xLabel} - 2\equilibriumCoefficientLinear_{\yLabel}) \geq \max(0, \relaxationParameterSymmetric-1), \\
        &\frac{\relaxationParameterAntiSymmetric}{2}\max_{\conservedVariableDiscrete\in[-\maximumInitialDatum, \maximumInitialDatum]}\frac{|\flux_{\xLabel}'(\conservedVariableDiscrete)|}{\latticeVelocity}\leq \relaxationParameterSymmetric\equilibriumCoefficientLinear_{\xLabel}+\frac{1}{2}\min(2-\relaxationParameterSymmetric-\relaxationParameterAntiSymmetric, 0, \relaxationParameterAntiSymmetric-\relaxationParameterSymmetric), \\
        &\frac{\relaxationParameterAntiSymmetric}{2}\max_{\conservedVariableDiscrete\in[-\maximumInitialDatum, \maximumInitialDatum]}\frac{|\flux_{\yLabel}'(\conservedVariableDiscrete)|}{\latticeVelocity}\leq \relaxationParameterSymmetric\equilibriumCoefficientLinear_{\yLabel}+\frac{1}{2}\min(2-\relaxationParameterSymmetric-\relaxationParameterAntiSymmetric, 0, \relaxationParameterAntiSymmetric-\relaxationParameterSymmetric).
    \end{cases}
\end{equation}
These conditions can be downgraded to other schemes taking into account the fact that some distribution functions remain identically zero.
\begin{equation}\label{eq:mononotonicityConditionsD2Q4}
    \text{\lbmScheme{2}{4}}\qquad :\qquad
    \begin{cases}
        &\frac{\relaxationParameterAntiSymmetric}{2}\max_{\conservedVariableDiscrete\in[-\maximumInitialDatum, \maximumInitialDatum]}\frac{|\flux_{\xLabel}'(\conservedVariableDiscrete)|}{\latticeVelocity}\leq \relaxationParameterSymmetric\equilibriumCoefficientLinear_{\xLabel}+\frac{1}{2}\min(2-\relaxationParameterSymmetric-\relaxationParameterAntiSymmetric, 0, \relaxationParameterAntiSymmetric-\relaxationParameterSymmetric), \\
        &\frac{\relaxationParameterAntiSymmetric}{2}\max_{\conservedVariableDiscrete\in[-\maximumInitialDatum, \maximumInitialDatum]}\frac{|\flux_{\yLabel}'(\conservedVariableDiscrete)|}{\latticeVelocity}\leq \relaxationParameterSymmetric\equilibriumCoefficientLinear_{\yLabel}+\frac{1}{2}\min(2-\relaxationParameterSymmetric-\relaxationParameterAntiSymmetric, 0, \relaxationParameterAntiSymmetric-\relaxationParameterSymmetric).
    \end{cases}
\end{equation}
\begin{equation}\label{eq:mononotonicityConditionsD1Q3}
    \text{\lbmScheme{1}{3}}\qquad :\qquad
    \begin{cases}
        &\relaxationParameterSymmetric (1-2 \equilibriumCoefficientLinear_{\xLabel}) \geq \max(0, \relaxationParameterSymmetric-1), \\
        &\frac{\relaxationParameterAntiSymmetric}{2}\max_{\conservedVariableDiscrete\in[-\maximumInitialDatum, \maximumInitialDatum]}\frac{|\flux_{\xLabel}'(\conservedVariableDiscrete)|}{\latticeVelocity}\leq \relaxationParameterSymmetric\equilibriumCoefficientLinear_{\xLabel}+\frac{1}{2}\min(2-\relaxationParameterSymmetric-\relaxationParameterAntiSymmetric, 0, \relaxationParameterAntiSymmetric-\relaxationParameterSymmetric).
    \end{cases}
\end{equation}
In the case of \lbmScheme{1}{2}, one can see that $\relaxationParameterSymmetric$ does not play any role, hence 
\begin{equation}\label{eq:mononotonicityConditionsD1Q2}
    \text{\lbmScheme{1}{2}}\qquad :\qquad
    \begin{cases}
        &\relaxationParameterAntiSymmetric\max_{\conservedVariableDiscrete\in[-\maximumInitialDatum, \maximumInitialDatum]}\frac{|\flux_{\xLabel}'(\conservedVariableDiscrete)|}{\latticeVelocity}\leq \relaxationParameterAntiSymmetric+2\min(1-\relaxationParameterAntiSymmetric, 0). 
    \end{cases}
\end{equation}
In \cite{aregba2025monotonicity}, we proved the following properties.
\begin{proposition}\label{prop:monotonicity}
    Except in the trivial case when all Courant numbers are zero (\idEst{} $\max_{\conservedVariableDiscrete\in[-\maximumInitialDatum, \maximumInitialDatum]}|\flux_{\xLabel}'(\conservedVariableDiscrete)|=0$ and $\max_{\conservedVariableDiscrete\in[-\maximumInitialDatum, \maximumInitialDatum]}|\flux_{\yLabel}'(\conservedVariableDiscrete)|=0$), which we exclude in what follows, under \eqref{eq:mononotonicityConditionsD2Q5} (or \eqref{eq:mononotonicityConditionsD2Q4}, \eqref{eq:mononotonicityConditionsD1Q3}, \eqref{eq:mononotonicityConditionsD1Q2} if they apply)
    \begin{equation}\label{eq:boundRelaxationParam}
        (\relaxationParameterSymmetric, \relaxationParameterAntiSymmetric)\in (0, 2)^2.
    \end{equation}
    Moreover, the relaxation operator $\collisionOperator_{\indexVelocity}$, is monotone non decreasing with respect to each of its $\numberVelocities$ arguments, when they belong to $\invariantCompactSetDistributions\definitionEquality \prod_{p\in\setVelIndexes}[\distributionFunctionLetter_{p}^{\atEquilibrium}(-\maximumInitialDatum), \distributionFunctionLetter_{p}^{\atEquilibrium}(\maximumInitialDatum)]$, for each $\indexVelocity\in\setVelIndexes$.
    Finally, the equilibria are monotone non decreasing functions of their unique argument: for $\indexVelocity\in\setVelIndexes$
    \begin{equation*}
        \frac{\differential \distributionFunctionLetter_{\indexVelocity}^{\atEquilibrium}(\conservedVariableDiscrete)}{\differential\conservedVariable}\geq 0, \qquad \conservedVariableDiscrete\in[-\maximumInitialDatum, \maximumInitialDatum].
    \end{equation*}
\end{proposition}

\subsection{Maximum principle}

From now on, we discuss only the case of \lbmScheme{2}{5} scheme and set all the steps needed to prove its convergence. 
The other previously mentioned schemes can be handled analoguously with easier computations.
Since we have monotone equilibria by \Cref{prop:monotonicity}, initialization \eqref{eq:initialization} and boundary conditions \eqref{eq:boundaryEquilibrium1}-\eqref{eq:boundaryEquilibrium2} are at equilibrium, and data are measured taking boundary data into account \eqref{eq:maximumData}, we can extend \cite[Proposition 4.2]{aregba2025monotonicity} in presence of boundaries, and have the following result.
\begin{proposition}[Maximum principle]
    Let \eqref{eq:mononotonicityConditionsD2Q5} hold.
    Then, for all $\indexTime\in\integerInterval{0}{\numberTimeSteps}$ and all $\vectorial{\indexSpace}\in\integerInterval{0}{\numberCells-1}^2$
    \begin{equation*}
        \distributionFunction_{\indexVelocity, \vectorial{\indexSpace}}^{\indexTime}\in[\distributionFunctionLetter_{\indexVelocity}^{\atEquilibrium}(-\maximumInitialDatum), \distributionFunctionLetter_{\indexVelocity}^{\atEquilibrium}(\maximumInitialDatum)]\quad \text{for}\quad \indexVelocity\in\setVelIndexes, \qquad \text{and}\qquad \conservedVariableDiscrete_{\vectorial{\indexSpace}}^{\indexTime}\in[-\maximumInitialDatum, \maximumInitialDatum].
    \end{equation*}
\end{proposition}

\subsection{$\petitLebesgueSpace{1}$-contractivity of the relaxation}

The essential property in all that follows is the fact that the relaxation phase is a $\petitLebesgueSpace{1}$ contraction.
Since relaxation is only performed inside the domain, this is \cite[Proposition 4.4]{aregba2025monotonicity}.
\begin{proposition}[$\petitLebesgueSpace{1}$-contractivity of the relaxation]\label{prop:contractivity}
    Let \eqref{eq:mononotonicityConditionsD2Q5} hold.
    Consider $(\discrete{g}_{\labZeroVel}, \dots, \discrete{g}_{\labNegY}) \in \invariantCompactSetDistributions$ and $(\distributionFunction_{\labZeroVel}, \dots, \distributionFunction_{\labNegY}) \in \invariantCompactSetDistributions$, then 
    \begin{equation*}
        \sum_{\indexVelocity\in\setVelIndexes} |\collisionOperator_{\indexVelocity}(\discrete{g}_{\labZeroVel}, \dots, \discrete{g}_{\labNegY}) -  \collisionOperator_{\indexVelocity}(\distributionFunction_{\labZeroVel}, \dots, \distributionFunction_{\labNegY}) | \leq \sum_{\indexVelocity\in\setVelIndexes}|\discrete{g}_{\indexVelocity}-\distributionFunction_{\indexVelocity}|.
    \end{equation*} 
\end{proposition}

\subsection{$\lebesgueSpace{1}$-boundedness}

In the sequel of the proof, just for the sake of simplifying the exposition, we consider the problem as it were set on a quarter of plane $(\reals_+^{*})^2$ instead of $(0, 1)^2$.
We therefore consider only the western and southern boundary conditions.
However, analogous properties and estimates hold for the numerical scheme on $(0, 1)^2$.

For sequences defined inside the domain (both scalar like $\conservedVariableDiscrete$ or vectors like $\vectorial{\distributionFunction}$), we define piecewise constant reconstructions as 
\begin{equation*}
    \conservedVariableDiscreteAsFunctionSpace{\indexTime} (\vectorial{\spaceVariable}) \definitionEquality\sum_{\vectorial{\indexSpace}\in\naturals^2} \conservedVariableDiscrete_{\vectorial{\indexSpace}}^{\indexTime}\indicatorFunction{\cell{\vectorial{\indexSpace}}}(\vectorial{\spaceVariable}), \qquad \conservedVariableDiscreteAsFunction(\timeVariable, \vectorial{\spaceVariable})\definitionEquality \sum_{\indexTime=0}^{\numberTimeSteps - 1} \conservedVariableDiscreteAsFunctionSpace{\indexTime} (\vectorial{\spaceVariable})\indicatorFunction{[\timeGridPoint{\indexTime}, \timeGridPoint{\indexTime + 1})}(\timeVariable).
\end{equation*}

\begin{remark}[Sharpness of the estimates]\label{rem:sharpness}
    In what follows, we aim at keeping expressions as simple as possible, in particular the right-hand sides in the inequalities---as long as they are bounded and controlled by the data of the problem.
    Therefore, we do not strive to make these inequalities very sharp.

    This is achieved by adding non-negative terms to the right-hand side of inequalities. In particular, we introduce what we refer to as $\flat$-procedure: let $\discrete{u}, \discrete{v}\in[-\maximumInitialDatum, \maximumInitialDatum]$ and $\varnothing\neq \tilde{\setVelIndexes}\subset \setVelIndexes$, we consistently write 
    \begin{equation*}
        \sum_{\indexVelocity\in\tilde{\setVelIndexes}}|\distributionFunctionLetter_{\indexVelocity}^{\atEquilibrium}(\discrete{u}) - \distributionFunctionLetter_{\indexVelocity}^{\atEquilibrium}(\discrete{v})|\leq \sum_{\indexVelocity\in{\setVelIndexes}}|\distributionFunctionLetter_{\indexVelocity}^{\atEquilibrium}(\discrete{u}) - \distributionFunctionLetter_{\indexVelocity}^{\atEquilibrium}(\discrete{v})| = |\discrete{u}-\discrete{v}|,
    \end{equation*}
    where the last equality comes from the monotonicity of equilibria, \confer{} \Cref{prop:monotonicity}, and \eqref{eq:consistency}.
    Another variant of the $\flat$-procedure that we employ boils down, with $\discrete{f}_{\indexVelocity}, \discrete{g}_{\indexVelocity}\in[\distributionFunctionLetter_{\indexVelocity}^{\atEquilibrium}(-\maximumInitialDatum), \distributionFunctionLetter_{\indexVelocity}^{\atEquilibrium}(\maximumInitialDatum)]$ for $\indexVelocity\in\setVelIndexes$, to employ the following loose bound:
    \begin{equation*}
        \sum_{\indexVelocity\in\tilde{\setVelIndexes}}|\discrete{f}_{\indexVelocity} - \discrete{g}_{\indexVelocity}|\leq\sum_{\indexVelocity\in{\setVelIndexes}}|\discrete{f}_{\indexVelocity} - \discrete{g}_{\indexVelocity}|.
    \end{equation*}
\end{remark}

\begin{proposition}[$\lebesgueSpace{1}$ bound]\label{prop:L1Bound}
    Consider two numerical solutions $\distributionFunction_{\indexVelocity, \vectorial{\indexSpace}}^{\indexTime}$ and $\discrete{g}_{\indexVelocity, \vectorial{\indexSpace}}^{\indexTime}$ obtained the initial data $\conservedVariable^{\initial}$ and $v^{\initial}$, and boundary data $\boundaryFunction{\alpha}$ and $\tilde{v}_{\alpha}$ for $\alpha=\labelWest, \labelSouth$, satisfying the same assumptions as with \eqref{eq:smoothnessInitial}-\eqref{eq:smoothnessBoundary} (with $(0, 1)$ replaced by $\reals_+^{*}$), such that 
    \begin{align*}
        &\max(\lVert\conservedVariable^{\initial}\rVert_{\lebesgueSpace{\infty}((\reals_+^{*})^2)}, \max_{\alpha=\labelWest, \labelSouth}\lVert\boundaryFunction{\alpha}\rVert_{\lebesgueSpace{\infty}((0, \finalTime)\times\reals_+^{*})})\leq \maximumInitialDatum, \\
        &\max(\lVert v^{\initial}\rVert_{\lebesgueSpace{\infty}((\reals_+^{*})^2)}, \max_{\alpha=\labelWest, \labelSouth}\lVert\tilde{v}_{\alpha}\rVert_{\lebesgueSpace{\infty}((0, \finalTime)\times\reals_+^{*})})\leq \maximumInitialDatum.
    \end{align*}
    Let \eqref{eq:mononotonicityConditionsD2Q5} hold with the previous $\maximumInitialDatum$.
    Then, for all $\indexTime\in\integerInterval{0}{\numberTimeSteps}$
    \begin{equation*}
        \lVert \vectorial{{g}}_{\discreteMark}^{\indexTime} - \distributionFunctionsAsFunction^{\indexTime}\rVert_{\lebesgueSpace{1}((\reals_+^{*})^2)} \leq \lVert v^{\initial}- \conservedVariable^{\initial}\rVert_{\lebesgueSpace{1}((\reals_+^{*})^2)} + \latticeVelocity\sum_{\alpha = \labelWest,  \labelSouth}\lVert \tilde{{v}}_{\alpha}- \boundaryFunction{\alpha}\rVert_{\lebesgueSpace{1}((0, \finalTime)\times\reals_+^{*})}.
    \end{equation*}
\end{proposition}
\begin{proof}
    We have, taking the transport into account:
    \begin{multline*}
        \lVert \vectorial{{g}}_{\discreteMark}^{\indexTime + 1} - \distributionFunctionsAsFunction^{\indexTime + 1}\rVert_{\lebesgueSpace{1}((\reals_+^{*})^2)} = \spaceStep^2\sum_{\indexVelocity\in\setVelIndexes}\sum_{\vectorial{\indexSpace}\in\naturals^2} |\discrete{{g}}_{\indexVelocity, \vectorial{\indexSpace}}^{\indexTime + 1} - \distributionFunction_{\indexVelocity, \vectorial{\indexSpace}}^{\indexTime + 1}|\\
        =\spaceStep^2 \Bigl (\sum_{\vectorial{\indexSpace}\in\naturals^2} |\discrete{{g}}_{\labZeroVel, \vectorial{\indexSpace}}^{\indexTime, \collided} - \distributionFunction_{\labZeroVel, \vectorial{\indexSpace}}^{\indexTime, \collided}| + \sum_{\indexSpace_{\yLabel}\in\naturals}|\distributionFunctionLetter_{\labPosX}^{\atEquilibrium}(\tilde{\discrete{v}}_{\labelWest, \indexSpace_{\yLabel}}^{\indexTime}) - \distributionFunctionLetter_{\labPosX}^{\atEquilibrium}(\boundaryDatumWest{\indexSpace_{\yLabel}}^{\indexTime})| + \sum_{\vectorial{\indexSpace}\in\naturals^2} |\discrete{{g}}_{\labPosX, \vectorial{\indexSpace}}^{\indexTime, \collided} - \distributionFunction_{\labPosX, \vectorial{\indexSpace}}^{\indexTime, \collided}|  + \sum_{\substack{\indexSpace_{\xLabel}\geq 1\\ \indexSpace_{\yLabel}\in\naturals }} |\discrete{{g}}_{\labNegX, \vectorial{\indexSpace}}^{\indexTime, \collided} - \distributionFunction_{\labNegX, \vectorial{\indexSpace}}^{\indexTime, \collided}| \\
        + \sum_{\indexSpace_{\xLabel}\in\naturals}|\distributionFunctionLetter_{\labPosY}^{\atEquilibrium}(\tilde{\discrete{v}}_{\labelSouth, \indexSpace_{\xLabel}}^{\indexTime}) - \distributionFunctionLetter_{\labPosY}^{\atEquilibrium}(\boundaryDatumSouth{\indexSpace_{\xLabel}}^{\indexTime})| + \sum_{\vectorial{\indexSpace}\in\naturals^2} |\discrete{{g}}_{\labPosY, \vectorial{\indexSpace}}^{\indexTime, \collided} - \distributionFunction_{\labPosY, \vectorial{\indexSpace}}^{\indexTime, \collided}|  + \sum_{\substack{\indexSpace_{\xLabel}\in\naturals\\ \indexSpace_{\yLabel}\geq 1 }} |\discrete{{g}}_{\labNegY, \vectorial{\indexSpace}}^{\indexTime, \collided} - \distributionFunction_{\labNegY, \vectorial{\indexSpace}}^{\indexTime, \collided}|
        \Bigr ).
    \end{multline*}
    Through the $\flat$-procedure, \confer{} \Cref{rem:sharpness}, we obtain 
    \begin{align*}
        &\lVert \vectorial{{g}}_{\discreteMark}^{\indexTime + 1} - \distributionFunctionsAsFunction^{\indexTime + 1}\rVert_{\lebesgueSpace{1}((\reals_+^{*})^2)}\\
        &\leq  \spaceStep^2\sum_{\indexVelocity\in\setVelIndexes}\sum_{\vectorial{\indexSpace}\in\naturals^2} |\discrete{{g}}_{\indexVelocity, \vectorial{\indexSpace}}^{\indexTime, \collided} - \distributionFunction_{\indexVelocity, \vectorial{\indexSpace}}^{\indexTime, \collided}| + \spaceStep \dashint_{\timeGridPoint{\indexTime}}^{\timeGridPoint{\indexTime+1}} \lVert \tilde{{v}}_{\labelWest}(\timeVariable, \cdot) - \boundaryFunction{\labelWest}(\timeVariable, \cdot) \rVert_{\lebesgueSpace{1}(\reals_+^{*})} \differential\timeVariable   + \spaceStep\dashint_{\timeGridPoint{\indexTime}}^{\timeGridPoint{\indexTime+1}} \lVert \tilde{{v}}_{\labelSouth}(\timeVariable, \cdot) - \boundaryFunction{\labelSouth}(\timeVariable, \cdot) \rVert_{\lebesgueSpace{1}(\reals_+^{*})}\differential\timeVariable  \\
        &\leq \lVert \vectorial{{g}}_{\discreteMark}^{\indexTime} - \distributionFunctionsAsFunction^{\indexTime}\rVert_{\lebesgueSpace{1}((\reals_+^{*})^2)} +\latticeVelocity \int_{\timeGridPoint{\indexTime}}^{\timeGridPoint{\indexTime+1}} \lVert \tilde{{v}}_{\labelWest}(\timeVariable, \cdot) - \boundaryFunction{\labelWest}(\timeVariable, \cdot) \rVert_{\lebesgueSpace{1}(\reals_+^{*})} \differential\timeVariable   + \latticeVelocity\int_{\timeGridPoint{\indexTime}}^{\timeGridPoint{\indexTime+1}} \lVert \tilde{{v}}_{\labelSouth}(\timeVariable, \cdot) - \boundaryFunction{\labelSouth}(\timeVariable, \cdot) \rVert_{\lebesgueSpace{1}(\reals_+^{*})}\differential\timeVariable ,
    \end{align*}
    where the second inequality comes from \Cref{prop:contractivity}.
    Iterating in time until reaching the initial one, we obtain 
    \begin{equation*}
        \lVert \vectorial{{g}}_{\discreteMark}^{\indexTime} - \distributionFunctionsAsFunction^{\indexTime}\rVert_{\lebesgueSpace{1}((\reals_+^{*})^2)} \leq \lVert v^{\initial}- \conservedVariable^{\initial}\rVert_{\lebesgueSpace{1}((\reals_+^{*})^2)} + \latticeVelocity\sum_{\alpha = \labelWest, \labelSouth}\lVert \tilde{{v}}_{\alpha}- \boundaryFunction{\alpha}\rVert_{\lebesgueSpace{1}((0, \finalTime)\times\reals_+^{*})}.
    \end{equation*}
\end{proof}

\subsection{Equicontinuity}

\begin{proposition}[Equicontinuity]\label{prop:equicontinuity}
    Let \eqref{eq:mononotonicityConditionsD2Q5} hold.
    Then, for all $\indexTime\in\integerInterval{0}{\numberTimeSteps - 1}$
    \begin{equation*}
        \lVert \distributionFunctionsAsFunction^{\indexTime + 1} - \distributionFunctionsAsFunction^{\indexTime} \rVert_{\lebesgueSpace{1}((\reals_+^{*})^2)}\leq  C_{EC}(\conservedVariable^{\initial}, \boundaryFunction{\labelWest}, \boundaryFunction{\labelSouth}) \spaceStep,
    \end{equation*}
    with 
    \begin{multline*}
        C_{EC}(\conservedVariable^{\initial}, \boundaryFunction{\labelWest}, \boundaryFunction{\labelSouth}) \definitionEquality 2 \totalVariation{\conservedVariable^{\initial}}{(\reals_+^*)^2} +\sup_{\xLabel>0} \lVert \conservedVariable^{\initial}(\spaceVariable, \cdot) \rVert_{\lebesgueSpace{1}(\reals_+^*)} + \sup_{\yLabel>0} \lVert \conservedVariable^{\initial}(\cdot, \yLabel) \rVert_{\lebesgueSpace{1}(\reals_+^*)} \\
        + \sum_{\alpha = \labelWest, \labelSouth}\Bigl ( \sup_{\timeVariable\in(0, \finalTime)}\lVert \boundaryFunction{\alpha}(\timeVariable, \cdot)\rVert_{\lebesgueSpace{1}(\reals_+^*)} + \totalVariation{\boundaryFunction{\alpha}}{(0, \finalTime)\times \reals_+^*} \Bigr ) , 
    \end{multline*}
    with $C_{EC}(\conservedVariable^{\initial}, \boundaryFunction{\labelWest}, \boundaryFunction{\labelSouth})$, and in particular the terms featuring suprema of $\lebesgueSpace{1}$ norms in one space direction, are bounded by virtue of \Cref{prop:restrictionLinF}.
\end{proposition}
\begin{proof}
    From \Cref{prop:contractivity} and through a $\flat$-procedure, we obtain
    \begin{equation*}
        \lVert \distributionFunctionsAsFunction^{\indexTime + 1} - \distributionFunctionsAsFunction^{\indexTime} \rVert_{\lebesgueSpace{1}((\reals_+^{*})^2)} \leq\lVert \distributionFunctionsAsFunction^{\indexTime} - \distributionFunctionsAsFunction^{\indexTime-1} \rVert_{\lebesgueSpace{1}((\reals_+^{*})^2)}  + \spaceStep^2\sum_{\indexSpace_{\yLabel}\in\naturals}|\boundaryDatumWest{\indexSpace_{\yLabel}}^{\indexTime}-\boundaryDatumWest{\indexSpace_{\yLabel}}^{\indexTime-1}| + \spaceStep^2\sum_{\indexSpace_{\xLabel}\in\naturals}|\boundaryDatumSouth{\indexSpace_{\xLabel}}^{\indexTime}-\boundaryDatumSouth{\indexSpace_{\xLabel}}^{\indexTime-1}|.
    \end{equation*}
    Iterating, we obtain:
    \begin{equation*}
        \lVert \distributionFunctionsAsFunction^{\indexTime + 1} - \distributionFunctionsAsFunction^{\indexTime} \rVert_{\lebesgueSpace{1}((\reals_+^{*})^2)}\leq  \lVert \distributionFunctionsAsFunction^{1} - \distributionFunctionsAsFunction^{0} \rVert_{\lebesgueSpace{1}((\reals_+^{*})^2)} + \spaceStep^2\sum_{p = 0}^{\numberTimeSteps - 1}\sum_{\indexSpace_{\yLabel}\in\naturals}|\boundaryDatumWest{\indexSpace_{\yLabel}}^{p+1}-\boundaryDatumWest{\indexSpace_{\yLabel}}^{p}| + \spaceStep^2\sum_{p = 0}^{\numberTimeSteps - 1} \sum_{\indexSpace_{\xLabel}\in\naturals}|\boundaryDatumSouth{\indexSpace_{\xLabel}}^{p+1}-\boundaryDatumSouth{\indexSpace_{\xLabel}}^{p}|.
    \end{equation*}
    For the last two terms on the right-hand side (we give the example of the first one), we have:
    \begin{multline*}
        \spaceStep^2\sum_{p = 0}^{\numberTimeSteps - 1}\sum_{\indexSpace_{\yLabel}\in\naturals}|\boundaryDatumWest{\indexSpace_{\yLabel}}^{p+1}-\boundaryDatumWest{\indexSpace_{\yLabel}}^{p}| = \latticeVelocity \sum_{\indexSpace_{\yLabel}\in\naturals}\sum_{p = 0}^{\numberTimeSteps - 1} \Bigl | \int_{\timeGridPoint{p}}^{\timeGridPoint{p + 1}}\reduceSpaceDoubleInt\int_{\indexSpace_{\yLabel}\spaceStep}^{(\indexSpace_{\yLabel} + 1)\spaceStep}(\boundaryFunction{\labelWest}(\timeVariable+\timeStep, \yLabel)-\boundaryFunction{\labelWest}(\timeVariable, \yLabel))\differential\yLabel\differential\timeVariable \Bigr |   \\
        \leq\latticeVelocity\sum_{\indexSpace_{\yLabel}\in\naturals}  \int_{\indexSpace_{\yLabel}\spaceStep}^{(\indexSpace_{\yLabel} + 1)\spaceStep} \sum_{p = 0}^{\numberTimeSteps - 1} \Bigl |  \int_{\timeGridPoint{p}}^{\timeGridPoint{p + 1}} (\boundaryFunction{\labelWest}(\timeVariable+\timeStep, \yLabel)-\boundaryFunction{\labelWest}(\timeVariable, \yLabel)) \differential\timeVariable \Bigr |  \differential\yLabel 
        = \spaceStep\int_0^{+\infty} \totalVariation{\tilde{\conservedVariable}_{\labelWest, \discreteMark_{\timeVariable}}(\cdot, \yLabel)}{(0, \finalTime)}\differential\yLabel,
    \end{multline*}
    where we have defined 
    \begin{equation*}
        \tilde{\conservedVariable}_{\labelWest, \discreteMark_{\timeVariable}}(\timeVariable, \yLabel)\definitionEquality \sum_{p=0}^{\numberTimeSteps - 1}\Bigl ( \dashint_{\timeGridPoint{p}}^{\timeGridPoint{p + 1}} \boundaryFunction{\labelWest}(\tilde{\timeVariable}, \yLabel)\differential\tilde{\timeVariable} \Bigr )\indicatorFunction{[\timeGridPoint{p}, \timeGridPoint{p+1})}(\timeVariable).
    \end{equation*}
    From \Cref{prop:totalVariationPiecewiseApprox}, we have that $\totalVariation{\tilde{\conservedVariable}_{\labelWest, \discreteMark_{\timeVariable}}(\cdot, \yLabel)}{(0, \finalTime)}\leq \totalVariation{\boundaryFunction{\labelWest}(\cdot, \yLabel)}{(0, \finalTime)}$ in the almost-everywhere sense in $\yLabel$, therefore:
    \begin{equation}\label{eq:tmp9}
        \spaceStep^2\sum_{p = 0}^{\numberTimeSteps - 1}\sum_{\indexSpace_{\yLabel}\in\naturals}|\boundaryDatumWest{\indexSpace_{\yLabel}}^{p+1}-\boundaryDatumWest{\indexSpace_{\yLabel}}^{p}| \leq \spaceStep\totalVariationAlongAxis{\boundaryFunction{\labelWest}}{(0, \finalTime)}{\timeVariable} \leq \spaceStep\totalVariation{\boundaryFunction{\labelWest}}{(0, \finalTime)\times \realsPositive},
    \end{equation}
    where the first inequality comes from \Cref{thm:2DtotalAlongAxis}, and the second one from the discussion of \Cref{rem:1Dvs2D}.
    This entails
    \begin{equation*}
        \lVert \distributionFunctionsAsFunction^{\indexTime + 1} - \distributionFunctionsAsFunction^{\indexTime} \rVert_{\lebesgueSpace{1}((\reals_+^{*})^2)}\leq  \lVert \distributionFunctionsAsFunction^{1} - \distributionFunctionsAsFunction^{0} \rVert_{\lebesgueSpace{1}((\reals_+^{*})^2)} +\spaceStep( \totalVariation{\boundaryFunction{\labelWest}}{(0, \finalTime)\times \realsPositive} + \totalVariation{\boundaryFunction{\labelSouth}}{(0, \finalTime)\times \realsPositive}).
    \end{equation*}

    For the remaining term, since we start at equilibrium
    \begin{align*}
        \lVert \distributionFunctionsAsFunction^{1} - \distributionFunctionsAsFunction^{0} \rVert_{\lebesgueSpace{1}((\reals_+^{*})^2)} = \spaceStep^2\Bigl ( &\sum_{\indexSpace_{\yLabel}\in\naturals}|\distributionFunction_{\labPosX}^{\atEquilibrium}(\boundaryDatumWest{\indexSpace_{\yLabel}}^0) - \distributionFunction_{\labPosX}^{\atEquilibrium}(\conservedVariableDiscrete_{0, \indexSpace_{\yLabel}}^0)| + \sum_{\substack{\indexSpace_{\xLabel}\geq 1\\ \indexSpace_{\yLabel}\in\naturals}}|\distributionFunction_{\labPosX}^{\atEquilibrium}(\conservedVariableDiscrete_{\vectorial{\indexSpace} - \canonicalBasisVector{\xLabel}}^0) - \distributionFunction_{\labPosX}^{\atEquilibrium}(\conservedVariableDiscrete_{\vectorial{\indexSpace}}^0)| + \sum_{\vectorial{\indexSpace}\in\naturals^2}|\distributionFunction_{\labNegX}^{\atEquilibrium}(\conservedVariableDiscrete_{\vectorial{\indexSpace} + \canonicalBasisVector{\xLabel}}^0) - \distributionFunction_{\labNegX}^{\atEquilibrium}(\conservedVariableDiscrete_{\vectorial{\indexSpace}}^0)| \\
        + &\sum_{\indexSpace_{\xLabel}\in\naturals}|\distributionFunction_{\labPosY}^{\atEquilibrium}(\boundaryDatumSouth{\indexSpace_{\xLabel}}^0) - \distributionFunction_{\labPosY}^{\atEquilibrium}(\conservedVariableDiscrete_{\indexSpace_{\xLabel}, 0}^0)| + \sum_{\substack{\indexSpace_{\xLabel}\in\naturals\\ \indexSpace_{\yLabel}\geq 1}}|\distributionFunction_{\labPosY}^{\atEquilibrium}(\conservedVariableDiscrete_{\vectorial{\indexSpace} - \canonicalBasisVector{\yLabel}}^0) - \distributionFunction_{\labPosY}^{\atEquilibrium}(\conservedVariableDiscrete_{\vectorial{\indexSpace}}^0)| + \sum_{\vectorial{\indexSpace}\in\naturals^2}|\distributionFunction_{\labNegY}^{\atEquilibrium}(\conservedVariableDiscrete_{\vectorial{\indexSpace} + \canonicalBasisVector{\yLabel}}^0) - \distributionFunction_{\labNegY}^{\atEquilibrium}(\conservedVariableDiscrete_{\vectorial{\indexSpace}}^0)| \Bigr ).
    \end{align*}
    Through the $\flat$-procedure, the last two terms of each line can be bounded by twice the total variation of the initial datum.
    On the other hand, another $\flat$-procedure transforms the first terms on the right-hand side of each line into the discrepancy between initial datum close to the boundary and the boundary data at the beginning.
    This reads
    \begin{equation*}
        \lVert \distributionFunctionsAsFunction^{1} - \distributionFunctionsAsFunction^{0} \rVert_{\lebesgueSpace{1}((\reals_+^{*})^2)} \leq \spaceStep \Bigl ( 2 \totalVariation{\conservedVariable^{\initial}}{(\reals_+^*)^2} +  \spaceStep \sum_{\indexSpace_{\yLabel}\in\naturals} |\conservedVariableDiscrete_{0, \indexSpace_{\yLabel}}^{0} - \boundaryDatumWest{\indexSpace_{\yLabel}}^{0}| + \spaceStep \sum_{\indexSpace_{\xLabel}\in\naturals} |\conservedVariableDiscrete_{\indexSpace_{\xLabel}, 0}^{0} - \boundaryDatumSouth{\indexSpace_{\xLabel}}^{0}| \Bigl ) .
    \end{equation*}
    The last two terms on the right-hand side can be controlled considering that they are bounded by integrals on narrow vertical bands.
    Let us show for the first one how the procedure goes:
    \begin{align}
        \spaceStep \sum_{\indexSpace_{\yLabel}\in\naturals} |\conservedVariableDiscrete_{0, \indexSpace_{\yLabel}}^{0} - \boundaryDatumWest{\indexSpace_{\yLabel}}^{0}|&\leq \spaceStep \sum_{\indexSpace_{\yLabel}\in\naturals} |\conservedVariableDiscrete_{0, \indexSpace_{\yLabel}}^{0}| + \spaceStep \sum_{\indexSpace_{\yLabel}\in\naturals} |\boundaryDatumWest{\indexSpace_{\yLabel}}^{0}|\nonumber \\
        &\leq \frac{1}{\spaceStep}\int_0^{\spaceStep}\reduceSpaceDoubleInt\int_0^{+\infty}|\conservedVariable^{\initial}(\xLabel, \yLabel)|\differential\yLabel\differential\xLabel + \frac{1}{\timeStep}\int_0^{\timeStep}\reduceSpaceDoubleInt\int_0^{+\infty}|\boundaryFunction{\labelWest}(\timeVariable, \yLabel)|\differential\yLabel\differential\timeVariable\nonumber \\
        &\leq\sup_{\xLabel>0}\lVert\conservedVariable^{\initial}(\xLabel, \cdot)\rVert_{\lebesgueSpace{1}(\realsPositive)} + \sup_{\timeVariable\in(0, \finalTime)}\lVert\boundaryFunction{\labelWest}(\timeVariable, \cdot)\rVert_{\lebesgueSpace{1}(\realsPositive)} \label{eq:tmp11}.
    \end{align}
    where the inequality on \eqref{eq:tmp11} comes from \Cref{prop:VerticalIntegrationBV} together with \Cref{prop:restrictionTotalVariation}.
    Notice that the suprema could be on narrow bands $\xLabel\in(0, \spaceStep)$ and on $\timeVariable\in(0, \timeStep)$: in accordance with \Cref{rem:sharpness}, we provide loose bounds for $\xLabel\in(0, +\infty)$ and on $\timeVariable\in(0, \finalTime)$.
    
    This provides 
    \begin{multline*}
        \lVert \distributionFunctionsAsFunction^{\indexTime + 1} - \distributionFunctionsAsFunction^{\indexTime} \rVert_{\lebesgueSpace{1}((\reals_+^{*})^2)}\leq 
        \spaceStep \Bigl ( 2\totalVariation{\conservedVariable^{\initial}}{(\reals_+^*)^2}  +\sup_{\xLabel>0}\lVert\conservedVariable^{\initial}(\xLabel, \cdot)\rVert_{\lebesgueSpace{1}(\realsPositive)} +\sup_{\yLabel>0}\lVert\conservedVariable^{\initial}(\cdot, \yLabel)\rVert_{\lebesgueSpace{1}(\realsPositive)} \\
        + \totalVariation{\boundaryFunction{\labelWest}}{(0, \finalTime)\times \realsPositive} + \totalVariation{\boundaryFunction{\labelSouth}}{(0, \finalTime)\times \realsPositive}) + \sup_{\timeVariable\in(0, \finalTime)}\lVert\boundaryFunction{\labelWest}(\timeVariable, \cdot)\rVert_{\lebesgueSpace{1}(\realsPositive)}
        + \sup_{\timeVariable\in(0, \finalTime)}\lVert\boundaryFunction{\labelSouth}(\timeVariable, \cdot)\rVert_{\lebesgueSpace{1}(\realsPositive)}\Bigl ),
    \end{multline*}
    which is the claim.
\end{proof}

\subsection{Total variation boundedness}

The following estimate is the most technical to be obtained.
\begin{proposition}\label{prop:totalVariationEstimates}
    Let \eqref{eq:mononotonicityConditionsD2Q5} hold. Then, for $\indexTime\in\integerInterval{0}{\numberTimeSteps}$:
    \begin{equation*}
        \totalVariation{\distributionFunctionsAsFunction^{\indexTime}}{(\reals_+^{*})^2} 
        \leq C_{\mathscr{V}}(\conservedVariable^{\initial}, \boundaryFunction{\labelWest},  \boundaryFunction{\labelSouth}, \finalTime),
    \end{equation*}
    where the constant $C_{\mathscr{V}}(\conservedVariable^{\initial}, \boundaryFunction{\labelWest},  \boundaryFunction{\labelSouth}, \finalTime)$ is given by 
    \begin{multline*}
        C_{\mathscr{V}}(\conservedVariable^{\initial}, \boundaryFunction{\labelWest},  \boundaryFunction{\labelSouth}, \finalTime)\definitionEquality 2 \Bigl ( \totalVariation{\conservedVariable^{\initial}}{(\reals_+^{*})^2} 
        +2\sup_{\xLabel>0}\lVert\conservedVariable^{\initial}(\xLabel, \cdot)\rVert_{\lebesgueSpace{1}(\realsPositive)} +2\sup_{\yLabel>0}\lVert\conservedVariable^{\initial}(\cdot, \yLabel)\rVert_{\lebesgueSpace{1}(\realsPositive)} +2\latticeVelocity \finalTime  \maximumInitialDatum\\
        +\sum_{\alpha=\labelWest, \labelSouth}\Bigl ( \totalVariation{\boundaryFunction{\alpha}}{(0, \finalTime)\times \realsPositive} 
        + \sup_{\timeVariable\in(0, \finalTime)}\lVert\boundaryFunction{\alpha}(\timeVariable, \cdot)\rVert_{\lebesgueSpace{1}(\realsPositive)} \Bigr ) \Bigr ),
    \end{multline*}  
    and is bounded by virtue of \Cref{prop:restrictionLinF}.
\end{proposition}
\begin{proof}
    We have, using \Cref{prop:contractivity}, that
    \begin{align*}
        \totalVariation{\distributionFunctionsAsFunction^{\indexTime, \collided}}{(\reals_+^{*})^2} &= \spaceStep \sum_{\indexVelocity\in\setVelIndexes}\sum_{\vectorial{\indexSpace}\in\naturals^2} \bigl (|\distributionFunction_{\indexVelocity, \vectorial{\indexSpace}}^{\indexTime, \collided} - \distributionFunction_{\indexVelocity, \vectorial{\indexSpace}+\canonicalBasisVector{\xLabel}}^{\indexTime, \collided}| + |\distributionFunction_{\indexVelocity, \vectorial{\indexSpace}}^{\indexTime, \collided} - \distributionFunction_{\indexVelocity, \vectorial{\indexSpace}+\canonicalBasisVector{\yLabel}}^{\indexTime, \collided}| \bigr )\\
        &=\spaceStep \sum_{\vectorial{\indexSpace}\in\naturals^2}\sum_{\indexVelocity\in\setVelIndexes} \bigl (|\collisionOperator_{\indexVelocity}(\vectorial{\distributionFunction}_{\vectorial{\indexSpace}}^{\indexTime}) - \collisionOperator_{\indexVelocity}(\vectorial{\distributionFunction}_{\vectorial{\indexSpace}+\canonicalBasisVector{\xLabel}}^{\indexTime})| + |\collisionOperator_{\indexVelocity}(\vectorial{\distributionFunction}_{\vectorial{\indexSpace}}^{\indexTime}) - \collisionOperator_{\indexVelocity}(\vectorial{\distributionFunction}_{\vectorial{\indexSpace}+\canonicalBasisVector{\yLabel}}^{\indexTime})| \bigr )\\
        &\leq\spaceStep \sum_{\indexVelocity\in\setVelIndexes}\sum_{\vectorial{\indexSpace}\in\naturals^2} \bigl (|\distributionFunction_{\indexVelocity, \vectorial{\indexSpace}}^{\indexTime} - \distributionFunction_{\indexVelocity, \vectorial{\indexSpace}+\canonicalBasisVector{\xLabel}}^{\indexTime}| + |\distributionFunction_{\indexVelocity, \vectorial{\indexSpace}}^{\indexTime} - \distributionFunction_{\indexVelocity, \vectorial{\indexSpace}+\canonicalBasisVector{\yLabel}}^{\indexTime}| \bigr ) = \totalVariation{\distributionFunctionsAsFunction^{\indexTime}}{(\reals_+^{*})^2}.
    \end{align*}
    We hence see that the relaxation is total variation diminishing: $\totalVariation{\distributionFunctionsAsFunction^{\indexTime, \collided}}{(\reals_+^{*})^2}\leq \totalVariation{\distributionFunctionsAsFunction^{\indexTime}}{(\reals_+^{*})^2}$.
    We now take the transport into account.
    \begin{align*}
        \totalVariation{\distributionFunctionsAsFunction^{\indexTime+1}}{(\reals_+^{*})^2} = \totalVariation{\distributionFunctionsAsFunctionComponent{\labZeroVel}^{\indexTime, \collided}}{(\reals_+^{*})^2} +\spaceStep\sum_{\vectorial{\indexSpace}\in\naturals^2} \bigl (&|\distributionFunction_{\labPosX, \vectorial{\indexSpace}-\canonicalBasisVector{\xLabel}}^{\indexTime, \collided} - \distributionFunction_{\labPosX, \vectorial{\indexSpace}}^{\indexTime, \collided}| + |\distributionFunction_{\labPosX, \vectorial{\indexSpace}-\canonicalBasisVector{\xLabel}}^{\indexTime, \collided} - \distributionFunction_{\labPosX, \vectorial{\indexSpace}+\canonicalBasisVector{\yLabel}-\canonicalBasisVector{\xLabel}}^{\indexTime, \collided}| \\
        +&|\distributionFunction_{\labNegX, \vectorial{\indexSpace} + \canonicalBasisVector{\xLabel}}^{\indexTime, \collided} - \distributionFunction_{\labNegX, \vectorial{\indexSpace}+2\canonicalBasisVector{\xLabel}}^{\indexTime, \collided}| + |\distributionFunction_{\labNegX, \vectorial{\indexSpace} + \canonicalBasisVector{\xLabel}}^{\indexTime, \collided} - \distributionFunction_{\labNegX, \vectorial{\indexSpace}+\canonicalBasisVector{\yLabel} + \canonicalBasisVector{\xLabel}}^{\indexTime, \collided}| \\
        +&|\distributionFunction_{\labPosY, \vectorial{\indexSpace} - \canonicalBasisVector{\yLabel}}^{\indexTime, \collided} - \distributionFunction_{\labPosY, \vectorial{\indexSpace}+\canonicalBasisVector{\xLabel}-\canonicalBasisVector{\yLabel}}^{\indexTime, \collided}| + |\distributionFunction_{\labPosY, \vectorial{\indexSpace}-\canonicalBasisVector{\yLabel}}^{\indexTime, \collided} - \distributionFunction_{\labPosY, \vectorial{\indexSpace}}^{\indexTime, \collided}| \\
        +&|\distributionFunction_{\labNegY, \vectorial{\indexSpace}+\canonicalBasisVector{\yLabel}}^{\indexTime, \collided} - \distributionFunction_{\labNegY, \vectorial{\indexSpace}+\canonicalBasisVector{\xLabel}+\canonicalBasisVector{\yLabel}}^{\indexTime, \collided}| + |\distributionFunction_{\labNegY, \vectorial{\indexSpace}+\canonicalBasisVector{\yLabel}}^{\indexTime, \collided} - \distributionFunction_{\labNegY, \vectorial{\indexSpace}+2\canonicalBasisVector{\yLabel}}^{\indexTime, \collided}| \bigr ).
    \end{align*}
    Changes in the indices yield
    \begin{multline*}
        \totalVariation{\distributionFunctionsAsFunction^{\indexTime+1}}{(\reals_+^{*})^2} = \totalVariation{\distributionFunctionsAsFunction^{\indexTime, \collided}}{(\reals_+^{*})^2} \\
        + \spaceStep \Bigl ( \sum_{\indexSpace_{\yLabel}\in\naturals} \bigl ( |\distributionFunction_{\labPosX, -1, \indexSpace_{\yLabel}}^{\indexTime, \collided} - \distributionFunction_{\labPosX, 0, \indexSpace_{\yLabel}}^{\indexTime, \collided}| + |\distributionFunction_{\labPosX, -1, \indexSpace_{\yLabel}}^{\indexTime, \collided} - \distributionFunction_{\labPosX, -1, \indexSpace_{\yLabel}+1}^{\indexTime, \collided}| \bigr ) 
        -\sum_{\indexSpace_{\yLabel}\in\naturals}\bigl ( |\distributionFunction_{\labNegX, 0, \indexSpace_{\yLabel}}^{\indexTime, \collided} - \distributionFunction_{\labNegX, 1, \indexSpace_{\yLabel}}^{\indexTime, \collided}| + |\distributionFunction_{\labNegX, 0, \indexSpace_{\yLabel}}^{\indexTime, \collided} - \distributionFunction_{\labNegX, 0, \indexSpace_{\yLabel}+1}^{\indexTime, \collided}| \bigr )\\
        +\sum_{\indexSpace_{\xLabel}\in\naturals}\bigl ( |\distributionFunction_{\labPosY, \indexSpace_{\xLabel}, -1}^{\indexTime, \collided} - \distributionFunction_{\labPosY, \indexSpace_{\xLabel}+1, -1}^{\indexTime, \collided}| + |\distributionFunction_{\labPosY, \indexSpace_{\xLabel}, -1}^{\indexTime, \collided} - \distributionFunction_{\labPosY, \indexSpace_{\xLabel}, 0}^{\indexTime, \collided}| \bigr )
        -\sum_{\indexSpace_{\xLabel}\in\naturals}\bigl ( |\distributionFunction_{\labNegY, \indexSpace_{\xLabel}, 0}^{\indexTime, \collided} - \distributionFunction_{\labNegY, \indexSpace_{\xLabel}+1, 0}^{\indexTime, \collided}| + |\distributionFunction_{\labNegY, \indexSpace_{\xLabel}, 0}^{\indexTime, \collided} - \distributionFunction_{\labNegY, \indexSpace_{\xLabel}, 1}^{\indexTime, \collided}| \bigr )  \Bigr )
    \end{multline*}
    Replacing the boundary conditions and using the total variation diminishing character of the relaxation
    \begin{multline*}
        \totalVariation{\distributionFunctionsAsFunction^{\indexTime+1}}{(\reals_+^{*})^2} \leq \totalVariation{\distributionFunctionsAsFunction^{\indexTime}}{(\reals_+^{*})^2} \\
        + \spaceStep \Bigl ( \sum_{\indexSpace_{\yLabel}\in\naturals} \bigl ( \overbrace{|\distributionFunctionLetter_{\labPosX}^{\atEquilibrium}(\boundaryDatumWest{\indexSpace_{\yLabel}}^{\indexTime}) - \distributionFunction_{\labPosX, 0, \indexSpace_{\yLabel}}^{\indexTime, \collided}|}^{\text{1D term}} + \overbrace{|\distributionFunctionLetter_{\labPosX}^{\atEquilibrium}(\boundaryDatumWest{\indexSpace_{\yLabel}}^{\indexTime})  - \distributionFunctionLetter_{\labPosX}^{\atEquilibrium}(\boundaryDatumWest{\indexSpace_{\yLabel} + 1}^{\indexTime}) |}^{\text{1D TV of the boundary datum}} \bigr ) 
        -\sum_{\indexSpace_{\yLabel}\in\naturals}\bigl ( \overbrace{|\distributionFunction_{\labNegX, 0, \indexSpace_{\yLabel}}^{\indexTime, \collided} - \distributionFunction_{\labNegX, 1, \indexSpace_{\yLabel}}^{\indexTime, \collided}|}^{\text{1D term}} + \overbrace{|\distributionFunction_{\labNegX, 0, \indexSpace_{\yLabel}}^{\indexTime, \collided} - \distributionFunction_{\labNegX, 0, \indexSpace_{\yLabel}+1}^{\indexTime, \collided}|}^{\text{multi-dimensional term}} \bigr )\\
        +\sum_{\indexSpace_{\xLabel}\in\naturals}\bigl ( |\distributionFunctionLetter_{\labPosY}^{\atEquilibrium}(\boundaryDatumSouth{\indexSpace_{\xLabel}}^{\indexTime}) - \distributionFunctionLetter_{\labPosY}^{\atEquilibrium}(\boundaryDatumSouth{\indexSpace_{\xLabel} + 1}^{\indexTime})| + |\distributionFunctionLetter_{\labPosY}^{\atEquilibrium}(\boundaryDatumSouth{\indexSpace_{\xLabel}}^{\indexTime})- \distributionFunction_{\labPosY, \indexSpace_{\xLabel}, 0}^{\indexTime, \collided}| \bigr )
        -\sum_{\indexSpace_{\xLabel}\in\naturals}\bigl ( |\distributionFunction_{\labNegY, \indexSpace_{\xLabel}, 0}^{\indexTime, \collided} - \distributionFunction_{\labNegY, \indexSpace_{\xLabel}+1, 0}^{\indexTime, \collided}| + |\distributionFunction_{\labNegY, \indexSpace_{\xLabel}, 0}^{\indexTime, \collided} - \distributionFunction_{\labNegY, \indexSpace_{\xLabel}, 1}^{\indexTime, \collided}| \bigr )  \Bigr ).
    \end{multline*}
    We bound the last term of the second line and the penultimate term in the last line by zero, yielding
    \begin{multline}\label{eq:estimateTVwithBoundary}
        \totalVariation{\distributionFunctionsAsFunction^{\indexTime + 1}}{(\reals_+^{*})^2} \leq \totalVariation{\distributionFunctionsAsFunction^{\indexTime}}{(\reals_+^{*})^2} \\
        + \spaceStep \Bigl ( \sum_{\indexSpace_{\yLabel}\in\naturals} \bigl ( |\distributionFunctionLetter_{\labPosX}^{\atEquilibrium}(\boundaryDatumWest{\indexSpace_{\yLabel}}^{\indexTime}) - \distributionFunction_{\labPosX, 0, \indexSpace_{\yLabel}}^{\indexTime, \collided}| + |\distributionFunctionLetter_{\labPosX}^{\atEquilibrium}(\boundaryDatumWest{\indexSpace_{\yLabel}}^{\indexTime})  - \distributionFunctionLetter_{\labPosX}^{\atEquilibrium}(\boundaryDatumWest{\indexSpace_{\yLabel} + 1}^{\indexTime}) | \bigr ) 
        -\sum_{\indexSpace_{\yLabel}\in\naturals} |\distributionFunction_{\labNegX, 0, \indexSpace_{\yLabel}}^{\indexTime, \collided} - \distributionFunction_{\labNegX, 1, \indexSpace_{\yLabel}}^{\indexTime, \collided}| \\
        +\sum_{\indexSpace_{\xLabel}\in\naturals}\bigl ( |\distributionFunctionLetter_{\labPosY}^{\atEquilibrium}(\boundaryDatumSouth{\indexSpace_{\xLabel}}^{\indexTime}) - \distributionFunctionLetter_{\labPosY}^{\atEquilibrium}(\boundaryDatumSouth{\indexSpace_{\xLabel} + 1}^{\indexTime})| + |\distributionFunctionLetter_{\labPosY}^{\atEquilibrium}(\boundaryDatumSouth{\indexSpace_{\xLabel}}^{\indexTime})- \distributionFunction_{\labPosY, \indexSpace_{\xLabel}, 0}^{\indexTime, \collided}| \bigr )
        -\sum_{\indexSpace_{\xLabel}\in\naturals} |\distributionFunction_{\labNegY, \indexSpace_{\xLabel}, 0}^{\indexTime, \collided} - \distributionFunction_{\labNegY, \indexSpace_{\xLabel}, 1}^{\indexTime, \collided}|  \Bigr ).
    \end{multline}
    The usual $\flat$-procedure yields the estimate
    \begin{multline}\label{eq:estimateLackingLastTermNew}
        \totalVariation{\distributionFunctionsAsFunction^{\indexTime + 1}}{(\reals_+^{*})^2} \leq \totalVariation{\distributionFunctionsAsFunction^{\indexTime}}{(\reals_+^{*})^2} + \spaceStep \sum_{\indexSpace_{\yLabel}\in\naturals} |\boundaryDatumWest{\indexSpace_{\yLabel}}^{\indexTime}  - \boundaryDatumWest{\indexSpace_{\yLabel} + 1}^{\indexTime} | + \spaceStep \sum_{\indexSpace_{\xLabel}\in\naturals} |\boundaryDatumSouth{\indexSpace_{\xLabel}}^{\indexTime}  - \boundaryDatumSouth{\indexSpace_{\xLabel} + 1}^{\indexTime} | \\
        + \spaceStep \Bigl ( \sum_{\indexSpace_{\yLabel}\in\naturals} |\distributionFunctionLetter_{\labPosX}^{\atEquilibrium}(\boundaryDatumWest{\indexSpace_{\yLabel}}^{\indexTime}) - \distributionFunction_{\labPosX, 0, \indexSpace_{\yLabel}}^{\indexTime, \collided}|  
        -\sum_{\indexSpace_{\yLabel}\in\naturals} |\distributionFunction_{\labNegX, 0, \indexSpace_{\yLabel}}^{\indexTime, \collided} - \distributionFunction_{\labNegX, 1, \indexSpace_{\yLabel}}^{\indexTime, \collided}| 
        +\sum_{\indexSpace_{\xLabel}\in\naturals} |\distributionFunctionLetter_{\labPosY}^{\atEquilibrium}(\boundaryDatumSouth{\indexSpace_{\xLabel}}^{\indexTime})- \distributionFunction_{\labPosY, \indexSpace_{\xLabel}, 0}^{\indexTime, \collided}| 
        -\sum_{\indexSpace_{\xLabel}\in\naturals} |\distributionFunction_{\labNegY, \indexSpace_{\xLabel}, 0}^{\indexTime, \collided} - \distributionFunction_{\labNegY, \indexSpace_{\xLabel}, 1}^{\indexTime, \collided}|  \Bigr ).
    \end{multline}
    Let us bound the last two terms on the first line of \eqref{eq:estimateLackingLastTermNew}.
    \begin{align*}
        \spaceStep \sum_{\indexSpace_{\yLabel}\in\naturals} |\boundaryDatumWest{\indexSpace_{\yLabel}}^{\indexTime}  - \boundaryDatumWest{\indexSpace_{\yLabel} + 1}^{\indexTime} | &= \frac{1}{\timeStep} \sum_{\indexSpace_{\yLabel}\in\naturals} \Bigl |\int_{\timeGridPoint{\indexTime}}^{\timeGridPoint{\indexTime + 1}}\reduceSpaceDoubleInt\int_{\indexSpace_{\yLabel}\spaceStep}^{(\indexSpace_{\yLabel}+1)\spaceStep} (\boundaryFunction{\labelWest}(\timeVariable, \yLabel)-\boundaryFunction{\labelWest}(\timeVariable, \yLabel + \spaceStep))\differential\yLabel\differential\timeVariable\Bigr |\\
        &\leq \frac{1}{\timeStep} \sum_{\indexSpace_{\yLabel}\in\naturals} \int_{\timeGridPoint{\indexTime}}^{\timeGridPoint{\indexTime + 1}}\Bigl | \int_{\indexSpace_{\yLabel}\spaceStep}^{(\indexSpace_{\yLabel}+1)\spaceStep} (\boundaryFunction{\labelWest}(\timeVariable, \yLabel)-\boundaryFunction{\labelWest}(\timeVariable, \yLabel + \spaceStep))\differential\yLabel \Bigr | \differential\timeVariable\\
        &= \frac{1}{\timeStep}  \int_{\timeGridPoint{\indexTime}}^{\timeGridPoint{\indexTime + 1}}\reduceSpaceDoubleInt\sum_{\indexSpace_{\yLabel}\in\naturals} \Bigl | \int_{\indexSpace_{\yLabel}\spaceStep}^{(\indexSpace_{\yLabel}+1)\spaceStep} (\boundaryFunction{\labelWest}(\timeVariable, \yLabel)-\boundaryFunction{\labelWest}(\timeVariable, \yLabel + \spaceStep))\differential\yLabel \Bigr | \differential\timeVariable =\latticeVelocity  \int_{\timeGridPoint{\indexTime}}^{\timeGridPoint{\indexTime + 1}} \reduceSpaceDoubleInt\totalVariation{\tilde{\conservedVariable}_{\labelWest, \discreteMark_{\yLabel}}(\timeVariable, \cdot)}{\realsPositive}\differential\timeVariable,
    \end{align*}
    where we have defined 
    \begin{equation*}
        \tilde{\conservedVariable}_{\labelWest, \discreteMark_{\yLabel}}(\timeVariable, \yLabel)\definitionEquality \sum_{\indexSpace_{\yLabel}\in\naturals}\Bigl ( \dashint_{\indexSpace_{\yLabel}\spaceStep}^{(\indexSpace_{\yLabel}+1)\spaceStep}\boundaryFunction{\labelWest}(\timeVariable, \tilde{\yLabel})\differential\tilde{\yLabel} \Bigr )\indicatorFunction{[\indexSpace_{\yLabel}\spaceStep, (\indexSpace_{\yLabel}+1)\spaceStep)}(\yLabel).
    \end{equation*}
    Using \Cref{prop:totalVariationPiecewiseApprox}, we obtain that $\totalVariation{\tilde{\conservedVariable}_{\labelWest, \discreteMark_{\yLabel}}(\timeVariable, \cdot)}{\realsPositive}\leq \totalVariation{\boundaryFunction{\labelWest}(\timeVariable, \cdot)}{\realsPositive}$, in the almost-everywhere sense in $\timeVariable$, hence 
    \begin{equation}\label{eq:tmp7}
        \spaceStep \sum_{\indexSpace_{\yLabel}\in\naturals} |\boundaryDatumWest{\indexSpace_{\yLabel}}^{\indexTime}  - \boundaryDatumWest{\indexSpace_{\yLabel} + 1}^{\indexTime} |\leq \latticeVelocity  \int_{\timeGridPoint{\indexTime}}^{\timeGridPoint{\indexTime + 1}} \reduceSpaceDoubleInt\totalVariation{\boundaryFunction{\labelWest}(\timeVariable, \cdot)}{\realsPositive}\differential\timeVariable.
    \end{equation}
    Performing the same kind of procedure for the term concerning the southern boundary, we bound the right-hand side of \eqref{eq:estimateLackingLastTermNew} to gain 
    \begin{multline}\label{eq:estimateLackingLastTermNewNew}
        \totalVariation{\distributionFunctionsAsFunction^{\indexTime + 1}}{(\reals_+^{*})^2} \leq \totalVariation{\distributionFunctionsAsFunction^{\indexTime}}{(\reals_+^{*})^2} + \latticeVelocity  \int_{\timeGridPoint{\indexTime}}^{\timeGridPoint{\indexTime + 1}} \reduceSpaceDoubleInt\totalVariation{\boundaryFunction{\labelWest}(\timeVariable, \cdot)}{\realsPositive}\differential\timeVariable + \latticeVelocity  \int_{\timeGridPoint{\indexTime}}^{\timeGridPoint{\indexTime + 1}} \reduceSpaceDoubleInt\totalVariation{\boundaryFunction{\labelSouth}(\timeVariable, \cdot)}{\realsPositive}\differential\timeVariable \\
        + \spaceStep \Bigl ( \sum_{\indexSpace_{\yLabel}\in\naturals} |\distributionFunctionLetter_{\labPosX}^{\atEquilibrium}(\boundaryDatumWest{\indexSpace_{\yLabel}}^{\indexTime}) - \distributionFunction_{\labPosX, 0, \indexSpace_{\yLabel}}^{\indexTime, \collided}|  
        -\sum_{\indexSpace_{\yLabel}\in\naturals} |\distributionFunction_{\labNegX, 0, \indexSpace_{\yLabel}}^{\indexTime, \collided} - \distributionFunction_{\labNegX, 1, \indexSpace_{\yLabel}}^{\indexTime, \collided}| 
        +\sum_{\indexSpace_{\xLabel}\in\naturals} |\distributionFunctionLetter_{\labPosY}^{\atEquilibrium}(\boundaryDatumSouth{\indexSpace_{\xLabel}}^{\indexTime})- \distributionFunction_{\labPosY, \indexSpace_{\xLabel}, 0}^{\indexTime, \collided}| 
        -\sum_{\indexSpace_{\xLabel}\in\naturals} |\distributionFunction_{\labNegY, \indexSpace_{\xLabel}, 0}^{\indexTime, \collided} - \distributionFunction_{\labNegY, \indexSpace_{\xLabel}, 1}^{\indexTime, \collided}|  \Bigr ).
    \end{multline}
    We now have to handle the last four terms.
    We follow the idea of \cite[Lemma 3.2]{aregba2004kinetic}, extended in 2D (\emph{i.e.} summing on $\indexSpace_{\yLabel}\in\naturals$).
    Using \Cref{prop:contractivity}, we can write:
    \begin{multline}\label{eq:tmp8}
        \sum_{\indexVelocity\in\setVelIndexes} \sum_{\indexSpace_{\yLabel}\in\naturals}|\distributionFunction_{\indexVelocity, 0, \indexSpace_{\yLabel}}^{\indexTime + 1, \collided}  - \distributionFunctionLetter_{\indexVelocity}^{\atEquilibrium}(\boundaryDatumWest{\indexSpace_{\yLabel}}^{\indexTime})| = \sum_{\indexVelocity\in\setVelIndexes} \sum_{\indexSpace_{\yLabel}\in\naturals}|\collisionOperator_{\indexVelocity}(\vectorial{\distributionFunction}_{ 0, \indexSpace_{\yLabel}}^{\indexTime+1})  - \collisionOperator_{\indexVelocity}(\vectorial{\distributionFunctionLetter}^{\atEquilibrium}(\boundaryDatumWest{\indexSpace_{\yLabel}}^{\indexTime}))|
        \leq \sum_{\indexVelocity\in\setVelIndexes} \sum_{\indexSpace_{\yLabel}\in\naturals}|\distributionFunction_{\indexVelocity, 0, \indexSpace_{\yLabel}}^{\indexTime+1}  - \distributionFunctionLetter_{\indexVelocity}^{\atEquilibrium}(\boundaryDatumWest{\indexSpace_{\yLabel}}^{\indexTime})| \\
        = \sum_{\indexSpace_{\yLabel}\in\naturals}\Bigl ( |\distributionFunction_{\labZeroVel, 0, \indexSpace_{\yLabel}}^{\indexTime, \collided}  - \distributionFunctionLetter_{\labZeroVel}^{\atEquilibrium}(\boundaryDatumWest{\indexSpace_{\yLabel}}^{\indexTime})| + |\distributionFunction_{\labNegX, 1, \indexSpace_{\yLabel}}^{\indexTime, \collided}  - \distributionFunctionLetter_{\indexVelocity}^{\atEquilibrium}(\boundaryDatumWest{\indexSpace_{\yLabel}}^{\indexTime})| + |\distributionFunction_{\labPosY, 0, \indexSpace_{\yLabel}-1}^{\indexTime, \collided}  - \distributionFunctionLetter_{\labPosY}^{\atEquilibrium}(\boundaryDatumWest{\indexSpace_{\yLabel}}^{\indexTime})| + |\distributionFunction_{\labNegY, 0, \indexSpace_{\yLabel}+1}^{\indexTime, \collided}  - \distributionFunctionLetter_{\labNegY}^{\atEquilibrium}(\boundaryDatumWest{\indexSpace_{\yLabel}}^{\indexTime})|\Bigr ).
    \end{multline}
    Notice that in this expression, terms in $\distributionFunction_{\labPosX}$ lack thanks to the boundary condition.
    Using the boundary condition on the southern wall---indicated by $\labelSouth$---gives 
    \begin{multline*}
        \sum_{\indexVelocity\in\setVelIndexes} \sum_{\indexSpace_{\yLabel}\in\naturals}|\distributionFunction_{\indexVelocity, 0, \indexSpace_{\yLabel}}^{\indexTime+1, \collided}  - \distributionFunctionLetter_{\indexVelocity}^{\atEquilibrium}(\boundaryDatumWest{\indexSpace_{\yLabel}}^{\indexTime})|\leq | \distributionFunctionLetter_{\labPosY}^{\atEquilibrium}(\boundaryDatumWest{0}^{\indexTime}) - \distributionFunctionLetter_{\labPosY}^{\atEquilibrium}(\boundaryDatumSouth{0}^{\indexTime})| \\
        +\sum_{\indexSpace_{\yLabel}\in\naturals}\Bigl ( |\distributionFunction_{\labZeroVel, 0, \indexSpace_{\yLabel}}^{\indexTime, \collided}  - \distributionFunctionLetter_{\labZeroVel}^{\atEquilibrium}(\boundaryDatumWest{\indexSpace_{\yLabel}}^{\indexTime})| + |\distributionFunction_{\labNegX, 1, \indexSpace_{\yLabel}}^{\indexTime, \collided}  - \distributionFunctionLetter_{\indexVelocity}^{\atEquilibrium}(\boundaryDatumWest{\indexSpace_{\yLabel}}^{\indexTime})| + |\distributionFunction_{\labPosY, 0, \indexSpace_{\yLabel}}^{\indexTime, \collided}  - \distributionFunctionLetter_{\labPosY}^{\atEquilibrium}(\boundaryDatumWest{\indexSpace_{\yLabel} + 1}^{\indexTime})| + |\distributionFunction_{\labNegY, 0, \indexSpace_{\yLabel} + 1}^{\indexTime, \collided}  - \distributionFunctionLetter_{\labNegY}^{\atEquilibrium}(\boundaryDatumWest{\indexSpace_{\yLabel}}^{\indexTime})|\Bigr ).
    \end{multline*}
    We add and substract terms in $\distributionFunction_{\labNegX}$, $\distributionFunction_{\labPosY}$, and $\distributionFunction_{\labNegY}$, and use the triangle inequality, to yield
    \begin{multline*}
        \sum_{\indexVelocity\in\setVelIndexes} \sum_{\indexSpace_{\yLabel}\in\naturals}|\distributionFunction_{\indexVelocity, 0, \indexSpace_{\yLabel}}^{\indexTime + 1, \collided}  - \distributionFunctionLetter_{\indexVelocity}^{\atEquilibrium}(\boundaryDatumWest{\indexSpace_{\yLabel}}^{\indexTime})|\leq | \distributionFunctionLetter_{\labPosY}^{\atEquilibrium}(\boundaryDatumWest{0}^{\indexTime}) - \distributionFunctionLetter_{\labPosY}^{\atEquilibrium}(\boundaryDatumSouth{0}^{\indexTime})| +\sum_{\indexSpace_{\yLabel}\in\naturals} |\distributionFunction_{\labZeroVel, 0, \indexSpace_{\yLabel}}^{\indexTime, \collided}  - \distributionFunctionLetter_{\labZeroVel}^{\atEquilibrium}(\boundaryDatumWest{\indexSpace_{\yLabel}}^{\indexTime})|\\
        +\sum_{\indexSpace_{\yLabel}\in\naturals} (|\distributionFunction_{\labNegX, 0, \indexSpace_{\yLabel}}^{\indexTime, \collided} - \distributionFunction_{\labNegX, 1, \indexSpace_{\yLabel}}^{\indexTime, \collided}  | + |\distributionFunction_{\labNegX, 0, \indexSpace_{\yLabel}}^{\indexTime, \collided}  - \distributionFunctionLetter_{\indexVelocity}^{\atEquilibrium}(\boundaryDatumWest{\indexSpace_{\yLabel}}^{\indexTime}) | )
        +\sum_{\indexSpace_{\yLabel}\in\naturals}\Bigl ( |\distributionFunction_{\labPosY, 0, \indexSpace_{\yLabel}}^{\indexTime, \collided}  - \distributionFunctionLetter_{\labPosY}^{\atEquilibrium}(\boundaryDatumWest{\indexSpace_{\yLabel}}^{\indexTime})| + |\distributionFunctionLetter_{\labPosY}^{\atEquilibrium}(\boundaryDatumWest{\indexSpace_{\yLabel}}^{\indexTime}) - \distributionFunctionLetter_{\labPosY}^{\atEquilibrium}(\boundaryDatumWest{\indexSpace_{\yLabel} + 1}^{\indexTime})|  \Bigr ) \\
        +\sum_{\indexSpace_{\yLabel}\in\naturals}\Bigl (|\distributionFunction_{\labNegY, 0, \indexSpace_{\yLabel} + 1}^{\indexTime, \collided}  - \distributionFunctionLetter_{\labNegY}^{\atEquilibrium}(\boundaryDatumWest{\indexSpace_{\yLabel} + 1}^{\indexTime})| + | \distributionFunctionLetter_{\labNegY}^{\atEquilibrium}(\boundaryDatumWest{\indexSpace_{\yLabel}}^{\indexTime}) - \distributionFunctionLetter_{\labNegY}^{\atEquilibrium}(\boundaryDatumWest{\indexSpace_{\yLabel}+1}^{\indexTime})|\Bigr ).
    \end{multline*}
    We add $|\distributionFunction_{\labNegY, 0, 0}^{\indexTime, \collided}  - \distributionFunctionLetter_{\labNegY}^{\atEquilibrium}(\boundaryDatumWest{0}^{\indexTime})|$ to the right-hand side of the previous inequality, thus:
    \begin{multline*}
        \sum_{\indexVelocity\in\setVelIndexes} \sum_{\indexSpace_{\yLabel}\in\naturals}|\distributionFunction_{\indexVelocity, 0, \indexSpace_{\yLabel}}^{\indexTime + 1, \collided}  - \distributionFunctionLetter_{\indexVelocity}^{\atEquilibrium}(\boundaryDatumWest{\indexSpace_{\yLabel}}^{\indexTime})|\leq | \distributionFunctionLetter_{\labPosY}^{\atEquilibrium}(\boundaryDatumWest{0}^{\indexTime}) - \distributionFunctionLetter_{\labPosY}^{\atEquilibrium}(\boundaryDatumSouth{0}^{\indexTime})| +\sum_{\indexSpace_{\yLabel}\in\naturals} |\distributionFunction_{\labZeroVel, 0, \indexSpace_{\yLabel}}^{\indexTime, \collided}  - \distributionFunctionLetter_{\labZeroVel}^{\atEquilibrium}(\boundaryDatumWest{\indexSpace_{\yLabel}}^{\indexTime})|\\
        +\sum_{\indexSpace_{\yLabel}\in\naturals} (|\distributionFunction_{\labNegX, 0, \indexSpace_{\yLabel}}^{\indexTime, \collided} - \distributionFunction_{\labNegX, 1, \indexSpace_{\yLabel}}^{\indexTime, \collided}  | + |\distributionFunction_{\labNegX, 0, \indexSpace_{\yLabel}}^{\indexTime, \collided}  - \distributionFunctionLetter_{\indexVelocity}^{\atEquilibrium}(\boundaryDatumWest{\indexSpace_{\yLabel}}^{\indexTime}) | )
        +\sum_{\indexSpace_{\yLabel}\in\naturals}\Bigl ( |\distributionFunction_{\labPosY, 0, \indexSpace_{\yLabel}}^{\indexTime, \collided}  - \distributionFunctionLetter_{\labPosY}^{\atEquilibrium}(\boundaryDatumWest{\indexSpace_{\yLabel}}^{\indexTime})| + |\distributionFunctionLetter_{\labPosY}^{\atEquilibrium}(\boundaryDatumWest{\indexSpace_{\yLabel}}^{\indexTime}) - \distributionFunctionLetter_{\labPosY}^{\atEquilibrium}(\boundaryDatumWest{\indexSpace_{\yLabel} + 1}^{\indexTime})|  \Bigr ) \\
        +\sum_{\indexSpace_{\yLabel}\in\naturals}|\distributionFunction_{\labNegY, 0, \indexSpace_{\yLabel}}^{\indexTime, \collided}  - \distributionFunctionLetter_{\labNegY}^{\atEquilibrium}(\boundaryDatumWest{\indexSpace_{\yLabel}}^{\indexTime})| +\sum_{\indexSpace_{\yLabel}\in\naturals} | \distributionFunctionLetter_{\labNegY}^{\atEquilibrium}(\boundaryDatumWest{\indexSpace_{\yLabel}}^{\indexTime}) - \distributionFunctionLetter_{\labNegY}^{\atEquilibrium}(\boundaryDatumWest{\indexSpace_{\yLabel}+1}^{\indexTime})|.
    \end{multline*}
    Writing things differently, we obtain
    \begin{multline*}
        \sum_{\indexVelocity\in\setVelIndexes} \sum_{\indexSpace_{\yLabel}\in\naturals}|\distributionFunction_{\indexVelocity, 0, \indexSpace_{\yLabel}}^{\indexTime + 1, \collided}  - \distributionFunctionLetter_{\indexVelocity}^{\atEquilibrium}(\boundaryDatumWest{\indexSpace_{\yLabel}}^{\indexTime})|\leq | \distributionFunctionLetter_{\labPosY}^{\atEquilibrium}(\boundaryDatumWest{0}^{\indexTime}) - \distributionFunctionLetter_{\labPosY}^{\atEquilibrium}(\boundaryDatumSouth{0}^{\indexTime})| +\sum_{\indexVelocity\in\setVelIndexes}\sum_{\indexSpace_{\yLabel}\in\naturals} |\distributionFunction_{\indexVelocity, 0, \indexSpace_{\yLabel}}^{\indexTime, \collided}  - \distributionFunctionLetter_{\indexVelocity}^{\atEquilibrium}(\boundaryDatumWest{\indexSpace_{\yLabel}}^{\indexTime})| + \sum_{\indexSpace_{\yLabel}\in\naturals} |\distributionFunction_{\labNegX, 0, \indexSpace_{\yLabel}}^{\indexTime, \collided} - \distributionFunction_{\labNegX, 1, \indexSpace_{\yLabel}}^{\indexTime, \collided}  | \\
        - \sum_{\indexSpace_{\yLabel}\in\naturals} |\distributionFunction_{\labPosX, 0, \indexSpace_{\yLabel}}^{\indexTime, \collided}  - \distributionFunctionLetter_{\labPosX}^{\atEquilibrium}(\boundaryDatumWest{\indexSpace_{\yLabel}}^{\indexTime})|
        +\sum_{\indexSpace_{\yLabel}\in\naturals} |\distributionFunctionLetter_{\labPosY}^{\atEquilibrium}(\boundaryDatumWest{\indexSpace_{\yLabel}}^{\indexTime}) - \distributionFunctionLetter_{\labPosY}^{\atEquilibrium}(\boundaryDatumWest{\indexSpace_{\yLabel} + 1}^{\indexTime})|  
        +\sum_{\indexSpace_{\yLabel}\in\naturals} | \distributionFunctionLetter_{\labNegY}^{\atEquilibrium}(\boundaryDatumWest{\indexSpace_{\yLabel}}^{\indexTime}) - \distributionFunctionLetter_{\labNegY}^{\atEquilibrium}(\boundaryDatumWest{\indexSpace_{\yLabel}+1}^{\indexTime})|.
    \end{multline*}
    By the usual $\flat$-procedure, we have 
    \begin{multline}\label{eq:tmp1}
        \sum_{\indexVelocity\in\setVelIndexes} \sum_{\indexSpace_{\yLabel}\in\naturals}|\distributionFunction_{\indexVelocity, 0, \indexSpace_{\yLabel}}^{\indexTime + 1, \collided}  - \distributionFunctionLetter_{\indexVelocity}^{\atEquilibrium}(\boundaryDatumWest{\indexSpace_{\yLabel}}^{\indexTime})|\leq | \boundaryDatumWest{0}^{\indexTime} - \boundaryDatumSouth{0}^{\indexTime}|  + \sum_{\indexVelocity\in\setVelIndexes}\sum_{\indexSpace_{\yLabel}\in\naturals} |\distributionFunction_{\indexVelocity, 0, \indexSpace_{\yLabel}}^{\indexTime, \collided}  - \distributionFunctionLetter_{\indexVelocity}^{\atEquilibrium}(\boundaryDatumWest{\indexSpace_{\yLabel}}^{\indexTime})| + \sum_{\indexSpace_{\yLabel}\in\naturals} |\distributionFunction_{\labNegX, 0, \indexSpace_{\yLabel}}^{\indexTime, \collided} - \distributionFunction_{\labNegX, 1, \indexSpace_{\yLabel}}^{\indexTime, \collided}  | \\
        - \sum_{\indexSpace_{\yLabel}\in\naturals} |\distributionFunction_{\labPosX, 0, \indexSpace_{\yLabel}}^{\indexTime, \collided}  - \distributionFunctionLetter_{\labPosX}^{\atEquilibrium}(\boundaryDatumWest{\indexSpace_{\yLabel}}^{\indexTime})|
        +
        \frac{1}{\timeStep}  \int_{\timeGridPoint{\indexTime}}^{\timeGridPoint{\indexTime + 1}} \reduceSpaceDoubleInt\totalVariation{\boundaryFunction{\labelWest}(\timeVariable, \cdot)}{\realsPositive}\differential\timeVariable.
    \end{multline}
    Adding and subtracting $\distributionFunctionLetter_{\indexVelocity}^{\atEquilibrium}(\boundaryDatumWest{\indexSpace_{\yLabel}}^{\indexTime-1})$ in the second term of the right-hand side, using a triangle inequality, the monotonicity of equilibria, and \Cref{prop:contractivity}, we gain:
    \begin{multline}\label{eq:keyIneqX}
        \sum_{\indexVelocity\in\setVelIndexes} \sum_{\indexSpace_{\yLabel}\in\naturals}|\distributionFunction_{\indexVelocity, 0, \indexSpace_{\yLabel}}^{\indexTime+1, \collided}  - \distributionFunctionLetter_{\indexVelocity}^{\atEquilibrium}(\boundaryDatumWest{\indexSpace_{\yLabel}}^{\indexTime})|\leq | \boundaryDatumWest{0}^{\indexTime} - \boundaryDatumSouth{0}^{\indexTime}|  + \sum_{\indexVelocity\in\setVelIndexes}\sum_{\indexSpace_{\yLabel}\in\naturals} |\distributionFunction_{\indexVelocity, 0, \indexSpace_{\yLabel}}^{\indexTime, \collided}  - \distributionFunctionLetter_{\indexVelocity}^{\atEquilibrium}(\boundaryDatumWest{\indexSpace_{\yLabel}}^{\indexTime-1})| + \sum_{\indexSpace_{\yLabel}\in\naturals} |\boundaryDatumWest{\indexSpace_{\yLabel}}^{\indexTime} - \boundaryDatumWest{\indexSpace_{\yLabel}}^{\indexTime-1}|\\
        + \sum_{\indexSpace_{\yLabel}\in\naturals} |\distributionFunction_{\labNegX, 0, \indexSpace_{\yLabel}}^{\indexTime, \collided} - \distributionFunction_{\labNegX, 1, \indexSpace_{\yLabel}}^{\indexTime, \collided}  | 
        - \sum_{\indexSpace_{\yLabel}\in\naturals} |\distributionFunction_{\labPosX, 0, \indexSpace_{\yLabel}}^{\indexTime, \collided}  - \distributionFunctionLetter_{\labPosX}^{\atEquilibrium}(\boundaryDatumWest{\indexSpace_{\yLabel}}^{\indexTime})|
        +
        \frac{1}{\timeStep}  \int_{\timeGridPoint{\indexTime}}^{\timeGridPoint{\indexTime + 1}} \reduceSpaceDoubleInt\totalVariation{\boundaryFunction{\labelWest}(\timeVariable, \cdot)}{\realsPositive}\differential\timeVariable.
    \end{multline}
    Using the same techniques (i.e. switching the role of $x$ and $y$), we obtain the inequality on the southern boundary:
    \begin{multline}\label{eq:keyIneqY}
        \sum_{\indexVelocity\in\setVelIndexes} \sum_{\indexSpace_{\xLabel}\in\naturals}|\distributionFunction_{\indexVelocity, \indexSpace_{\xLabel}, 0}^{\indexTime +1, \collided}  - \distributionFunctionLetter_{\indexVelocity}^{\atEquilibrium}(\boundaryDatumSouth{\indexSpace_{\xLabel}}^{\indexTime})|\leq | \boundaryDatumWest{0}^{\indexTime} - \boundaryDatumSouth{0}^{\indexTime}| + \sum_{\indexVelocity\in\setVelIndexes}\sum_{\indexSpace_{\xLabel}\in\naturals} |\distributionFunction_{\indexVelocity, \indexSpace_{\xLabel}, 0}^{\indexTime, \collided}  - \distributionFunctionLetter_{\indexVelocity}^{\atEquilibrium}(\boundaryDatumSouth{\indexSpace_{\xLabel}}^{\indexTime-1})| + \sum_{\indexSpace_{\xLabel}\in\naturals} |\boundaryDatumSouth{\indexSpace_{\xLabel}}^{\indexTime} - \boundaryDatumSouth{\indexSpace_{\xLabel}}^{\indexTime-1}|\\
        + \sum_{\indexSpace_{\xLabel}\in\naturals} |\distributionFunction_{\labNegY, \indexSpace_{\xLabel}, 0}^{\indexTime, \collided} - \distributionFunction_{\labNegY, \indexSpace_{\xLabel}, 1}^{\indexTime, \collided}  | - \sum_{\indexSpace_{\xLabel}\in\naturals} |\distributionFunction_{\labPosY, \indexSpace_{\xLabel}, 0}^{\indexTime, \collided}  - \distributionFunctionLetter_{\labPosY}^{\atEquilibrium}(\boundaryDatumSouth{\indexSpace_{\xLabel}}^{\indexTime})| 
        +
        \frac{1}{\timeStep}  \int_{\timeGridPoint{\indexTime}}^{\timeGridPoint{\indexTime + 1}} \reduceSpaceDoubleInt\totalVariation{\boundaryFunction{\labelSouth}(\timeVariable, \cdot)}{\realsPositive}\differential\timeVariable.
    \end{multline}
    Summing \eqref{eq:keyIneqX} and \eqref{eq:keyIneqY} gives an upper bound for the second line of \eqref{eq:estimateLackingLastTermNewNew}:
    \begin{align*}
        \sum_{\indexSpace_{\yLabel}\in\naturals} |\distributionFunctionLetter_{\labPosX}^{\atEquilibrium}(\boundaryDatumWest{\indexSpace_{\yLabel}}^{\indexTime}) &- \distributionFunction_{\labPosX, 0, \indexSpace_{\yLabel}}^{\indexTime, \collided}|  
        -\sum_{\indexSpace_{\yLabel}\in\naturals} |\distributionFunction_{\labNegX, 0, \indexSpace_{\yLabel}}^{\indexTime, \collided} - \distributionFunction_{\labNegX, 1, \indexSpace_{\yLabel}}^{\indexTime, \collided}| 
        +\sum_{\indexSpace_{\xLabel}\in\naturals} |\distributionFunctionLetter_{\labPosY}^{\atEquilibrium}(\boundaryDatumSouth{\indexSpace_{\xLabel}}^{\indexTime})- \distributionFunction_{\labPosY, \indexSpace_{\xLabel}, 0}^{\indexTime, \collided}| 
        -\sum_{\indexSpace_{\xLabel}\in\naturals} |\distributionFunction_{\labNegY, \indexSpace_{\xLabel}, 0}^{\indexTime, \collided} - \distributionFunction_{\labNegY, \indexSpace_{\xLabel}, 1}^{\indexTime, \collided}|\\
        \leq &-\sum_{\indexVelocity\in\setVelIndexes} \sum_{\indexSpace_{\yLabel}\in\naturals}|\distributionFunction_{\indexVelocity, 0, \indexSpace_{\yLabel}}^{\indexTime + 1, \collided}  - \distributionFunctionLetter_{\indexVelocity}^{\atEquilibrium}(\boundaryDatumWest{\indexSpace_{\yLabel}}^{\indexTime})| - \sum_{\indexVelocity\in\setVelIndexes} \sum_{\indexSpace_{\xLabel}\in\naturals}|\distributionFunction_{\indexVelocity, \indexSpace_{\xLabel}, 0}^{\indexTime + 1, \collided}  - \distributionFunctionLetter_{\indexVelocity}^{\atEquilibrium}(\boundaryDatumSouth{\indexSpace_{\xLabel}}^{\indexTime})|\\
        &+\sum_{\indexVelocity\in\setVelIndexes}\sum_{\indexSpace_{\yLabel}\in\naturals} |\distributionFunction_{\indexVelocity, 0, \indexSpace_{\yLabel}}^{\indexTime, \collided}  - \distributionFunctionLetter_{\indexVelocity}^{\atEquilibrium}(\boundaryDatumWest{\indexSpace_{\yLabel}}^{\indexTime-1})| +\sum_{\indexVelocity\in\setVelIndexes}\sum_{\indexSpace_{\xLabel}\in\naturals} |\distributionFunction_{\indexVelocity, \indexSpace_{\xLabel}, 0}^{\indexTime, \collided}  - \distributionFunctionLetter_{\indexVelocity}^{\atEquilibrium}(\boundaryDatumSouth{\indexSpace_{\xLabel}}^{\indexTime-1})|\\
        &+ \sum_{\indexSpace_{\yLabel}\in\naturals} |\boundaryDatumWest{\indexSpace_{\yLabel}}^{\indexTime} - \boundaryDatumWest{\indexSpace_{\yLabel}}^{\indexTime-1}| + \sum_{\indexSpace_{\xLabel}\in\naturals} |\boundaryDatumSouth{\indexSpace_{\xLabel}}^{\indexTime} - \boundaryDatumSouth{\indexSpace_{\xLabel}}^{\indexTime-1}|\\
        &+2| \boundaryDatumWest{0}^{\indexTime} - \boundaryDatumSouth{0}^{\indexTime}| 
        +
        \frac{1}{\timeStep}  \int_{\timeGridPoint{\indexTime}}^{\timeGridPoint{\indexTime + 1}} \reduceSpaceDoubleInt\totalVariation{\boundaryFunction{\labelWest}(\timeVariable, \cdot)}{\realsPositive}\differential\timeVariable
        +
        \frac{1}{\timeStep}  \int_{\timeGridPoint{\indexTime}}^{\timeGridPoint{\indexTime + 1}} \reduceSpaceDoubleInt\totalVariation{\boundaryFunction{\labelSouth}(\timeVariable, \cdot)}{\realsPositive}\differential\timeVariable.
    \end{align*}
    Into \eqref{eq:estimateLackingLastTermNewNew}, upon rearranging
    \begin{multline*}
        \overbrace{\totalVariation{\distributionFunctionsAsFunction^{\indexTime+1}}{(\reals_+^{*})^2} +\spaceStep\Bigl( \sum_{\indexVelocity\in\setVelIndexes} \sum_{\indexSpace_{\yLabel}\in\naturals}|\distributionFunction_{\indexVelocity, 0, \indexSpace_{\yLabel}}^{\indexTime + 1, \collided}  - \distributionFunctionLetter_{\indexVelocity}^{\atEquilibrium}(\boundaryDatumWest{\indexSpace_{\yLabel}}^{\indexTime})| + \sum_{\indexVelocity\in\setVelIndexes} \sum_{\indexSpace_{\xLabel}\in\naturals}|\distributionFunction_{\indexVelocity, \indexSpace_{\xLabel}, 0}^{\indexTime + 1, \collided}  - \distributionFunctionLetter_{\indexVelocity}^{\atEquilibrium}(\boundaryDatumSouth{\indexSpace_{\xLabel}}^{\indexTime})| \Bigr )}^{=:a^{\indexTime}} \\
        \leq \underbrace{\totalVariation{\distributionFunctionsAsFunction^{\indexTime}}{(\reals_+^{*})^2} +\spaceStep \Bigl ( \sum_{\indexVelocity\in\setVelIndexes}\sum_{\indexSpace_{\yLabel}\in\naturals} |\distributionFunction_{\indexVelocity, 0, \indexSpace_{\yLabel}}^{\indexTime, \collided}  - \distributionFunctionLetter_{\indexVelocity}^{\atEquilibrium}(\boundaryDatumWest{\indexSpace_{\yLabel}}^{\indexTime-1})| +\sum_{\indexVelocity\in\setVelIndexes}\sum_{\indexSpace_{\xLabel}\in\naturals} |\distributionFunction_{\indexVelocity, \indexSpace_{\xLabel}, 0}^{\indexTime, \collided}  - \distributionFunctionLetter_{\indexVelocity}^{\atEquilibrium}(\boundaryDatumSouth{\indexSpace_{\xLabel}}^{\indexTime-1})| \Bigr )}_{=:a^{\indexTime-1}} \\
        +\underbrace{\spaceStep\Bigl ( 
        \sum_{\indexSpace_{\yLabel}\in\naturals} |\boundaryDatumWest{\indexSpace_{\yLabel}}^{\indexTime} - \boundaryDatumWest{\indexSpace_{\yLabel}}^{\indexTime-1}| + \sum_{\indexSpace_{\xLabel}\in\naturals} |\boundaryDatumSouth{\indexSpace_{\xLabel}}^{\indexTime} - \boundaryDatumSouth{\indexSpace_{\xLabel}}^{\indexTime-1}| 
        +2 | \boundaryDatumWest{0}^{\indexTime} - \boundaryDatumSouth{0}^{\indexTime}| \Bigr )
        +
        2\latticeVelocity\int_{\timeGridPoint{\indexTime}}^{\timeGridPoint{\indexTime + 1}} \reduceSpaceDoubleInt\totalVariation{\boundaryFunction{\labelWest}(\timeVariable, \cdot)}{\realsPositive}\differential\timeVariable
        +
        2\latticeVelocity\int_{\timeGridPoint{\indexTime}}^{\timeGridPoint{\indexTime + 1}} \reduceSpaceDoubleInt\totalVariation{\boundaryFunction{\labelSouth}(\timeVariable, \cdot)}{\realsPositive}\differential\timeVariable
        }_{=:b^{\indexTime-1}}.
    \end{multline*}
    We thus have the recurrence $a^{\indexTime} \leq a^{\indexTime-1} + b^{\indexTime-1}$ for $\indexTime\geq 1$, hence by iteration $a^{\indexTime} \leq a^0 + \sum_{p=0}^{\indexTime-1}b^p$.
    We have the bound  
    \begin{multline*}
        \sum_{p=0}^{\indexTime-1}b^p\leq \spaceStep\sum_{p = 1}^{\numberTimeSteps}\sum_{\indexSpace_{\yLabel}\in\naturals} |\boundaryDatumWest{\indexSpace_{\yLabel}}^{p} - \boundaryDatumWest{\indexSpace_{\yLabel}}^{p-1}| + \spaceStep\sum_{p = 1}^{\numberTimeSteps}\sum_{\indexSpace_{\xLabel}\in\naturals} |\boundaryDatumSouth{\indexSpace_{\xLabel}}^{p} - \boundaryDatumSouth{\indexSpace_{\xLabel}}^{p-1}|\\
        +2\latticeVelocity (\finalTime-\timeStep)  \maximumInitialDatum
        +2\latticeVelocity\int_{\timeStep}^{\finalTime} \reduceSpaceDoubleInt\totalVariation{\boundaryFunction{\labelWest}(\timeVariable, \cdot)}{\realsPositive}\differential\timeVariable
        +2\latticeVelocity
        \int_{\timeStep}^{\finalTime} \reduceSpaceDoubleInt\totalVariation{\boundaryFunction{\labelSouth}(\timeVariable, \cdot)}{\realsPositive}\differential\timeVariable.
    \end{multline*}
    The first and the second term on the right-hand side are treated using \eqref{eq:tmp9}, so that we obtain
    \begin{multline*}
        \sum_{p=0}^{\indexTime-1}b^p\leq \totalVariation{\boundaryFunction{\labelWest}}{(0, \finalTime)\times \realsPositive} + \totalVariation{\boundaryFunction{\labelSouth}}{(0, \finalTime)\times \realsPositive} 
        +2\latticeVelocity (\finalTime-\timeStep)  \maximumInitialDatum\\
        +2\latticeVelocity\int_{\timeStep}^{\finalTime} \reduceSpaceDoubleInt\totalVariation{\boundaryFunction{\labelWest}(\timeVariable, \cdot)}{\realsPositive}\differential\timeVariable
        +2\latticeVelocity
        \int_{\timeStep}^{\finalTime} \reduceSpaceDoubleInt\totalVariation{\boundaryFunction{\labelSouth}(\timeVariable, \cdot)}{\realsPositive}\differential\timeVariable.
    \end{multline*}

    We now work on a bound for $a^0$:
    \begin{equation}\label{eq:tmp6}
        a^0 = \totalVariation{\distributionFunctionsAsFunction^{1}}{(\reals_+^{*})^2} +\spaceStep\Bigl( \sum_{\indexVelocity\in\setVelIndexes} \sum_{\indexSpace_{\yLabel}\in\naturals}|\distributionFunction_{\indexVelocity, 0, \indexSpace_{\yLabel}}^{1, \collided}  - \distributionFunctionLetter_{\indexVelocity}^{\atEquilibrium}(\boundaryDatumWest{\indexSpace_{\yLabel}}^{0})| + \sum_{\indexVelocity\in\setVelIndexes} \sum_{\indexSpace_{\xLabel}\in\naturals}|\distributionFunction_{\indexVelocity, \indexSpace_{\xLabel}, 0}^{1, \collided}  - \distributionFunctionLetter_{\indexVelocity}^{\atEquilibrium}(\boundaryDatumSouth{\indexSpace_{\xLabel}}^{0})| \Bigr ).
    \end{equation}
    Using \eqref{eq:estimateLackingLastTermNewNew} at the initial time, bounding negative terms on the right-hand side by zero, and the fact of initializing at equilibrium:
    \begin{multline}
        \totalVariation{\distributionFunctionsAsFunction^{1}}{(\reals_+^{*})^2} \leq \totalVariation{\distributionFunctionsAsFunction^{0}}{(\reals_+^{*})^2} +\latticeVelocity  \int_{0}^{\timeStep} \reduceSpaceDoubleInt\totalVariation{\boundaryFunction{\labelWest}(\timeVariable, \cdot)}{\realsPositive}\differential\timeVariable + \latticeVelocity  \int_{0}^{\timeStep} \reduceSpaceDoubleInt\totalVariation{\boundaryFunction{\labelSouth}(\timeVariable, \cdot)}{\realsPositive}\differential\timeVariable \\
        + \spaceStep \Bigl ( \sum_{\indexSpace_{\yLabel}\in\naturals} |\distributionFunctionLetter_{\labPosX}^{\atEquilibrium}(\boundaryDatumWest{\indexSpace_{\yLabel}}^{0}) - \distributionFunctionLetter_{\labPosX}^{\atEquilibrium}(\conservedVariableDiscrete^0_{0, \indexSpace_{\yLabel}})|   
        +\sum_{\indexSpace_{\xLabel}\in\naturals} |\distributionFunctionLetter_{\labPosY}^{\atEquilibrium}(\boundaryDatumSouth{\indexSpace_{\xLabel}}^{0})- \distributionFunctionLetter_{\labPosY}^{\atEquilibrium}(\conservedVariableDiscrete_{\indexSpace_{\xLabel}, 0}^{0})| 
        \Bigr )\\
        \leq \totalVariation{\conservedVariable^{\initial}}{(\reals_+^{*})^2} + \latticeVelocity  \int_{0}^{\timeStep} \reduceSpaceDoubleInt\totalVariation{\boundaryFunction{\labelWest}(\timeVariable, \cdot)}{\realsPositive}\differential\timeVariable + \latticeVelocity  \int_{0}^{\timeStep} \reduceSpaceDoubleInt\totalVariation{\boundaryFunction{\labelSouth}(\timeVariable, \cdot)}{\realsPositive}\differential\timeVariable  + \spaceStep\sum_{\indexSpace_{\yLabel}\in\naturals} |\boundaryDatumWest{\indexSpace_{\yLabel}}^{0} - \conservedVariableDiscrete^0_{0, \indexSpace_{\yLabel}}| + \spaceStep\sum_{\indexSpace_{\xLabel}\in\naturals} |\boundaryDatumSouth{\indexSpace_{\xLabel}}^{0}-\conservedVariableDiscrete_{\indexSpace_{\xLabel}, 0}^{0}|,
    \end{multline}
    where the last inequality has been obtained through the $\flat$-procedure.

    For the last two terms in \eqref{eq:tmp6} (we take the first one for the sake of presentation), we use \eqref{eq:tmp1} with $\indexTime = 0$, and bound negative terms by zero, to yield
    \begin{multline*}
        \spaceStep\sum_{\indexVelocity\in\setVelIndexes} \sum_{\indexSpace_{\yLabel}\in\naturals}|\distributionFunction_{\indexVelocity, 0, \indexSpace_{\yLabel}}^{1, \collided}  - \distributionFunctionLetter_{\indexVelocity}^{\atEquilibrium}(\boundaryDatumWest{\indexSpace_{\yLabel}}^{0})|\leq \spaceStep| \boundaryDatumWest{0}^{0} - \boundaryDatumSouth{0}^{0}|  + \spaceStep\sum_{\indexVelocity\in\setVelIndexes}\sum_{\indexSpace_{\yLabel}\in\naturals} |\distributionFunction_{\indexVelocity, 0, \indexSpace_{\yLabel}}^{0, \collided}  - \distributionFunctionLetter_{\indexVelocity}^{\atEquilibrium}(\boundaryDatumWest{\indexSpace_{\yLabel}}^{0})| + \spaceStep \sum_{\indexSpace_{\yLabel}\in\naturals} |\distributionFunction_{\labNegX, 0, \indexSpace_{\yLabel}}^{0, \collided} - \distributionFunction_{\labNegX, 1, \indexSpace_{\yLabel}}^{0, \collided}  | \\
        +
       \latticeVelocity \int_{0}^{\timeStep} \reduceSpaceDoubleInt\totalVariation{\boundaryFunction{\labelWest}(\timeVariable, \cdot)}{\realsPositive}\differential\timeVariable.
    \end{multline*}
    Using the fact that we start at equilibrium, using the $\flat$-procedure on the penultimate term gives
    \begin{align*}
        \spaceStep \sum_{\indexVelocity\in\setVelIndexes} \sum_{\indexSpace_{\yLabel}\in\naturals}|\distributionFunction_{\indexVelocity, 0, \indexSpace_{\yLabel}}^{1, \collided}  - \distributionFunctionLetter_{\indexVelocity}^{\atEquilibrium}(\boundaryDatumWest{\indexSpace_{\yLabel}}^{0})| \leq \spaceStep| \boundaryDatumWest{0}^{0} - \boundaryDatumSouth{0}^{0}|  + \spaceStep \sum_{\indexSpace_{\yLabel}\in\naturals} |\conservedVariableDiscrete_{0, \indexSpace_{\yLabel}}^{0} - \boundaryDatumWest{\indexSpace_{\yLabel}}^{0}| + \spaceStep \sum_{\indexSpace_{\yLabel}\in\naturals} |\conservedVariableDiscrete_{0, \indexSpace_{\yLabel}}^{0} - \conservedVariableDiscrete_{1, \indexSpace_{\yLabel}}^{0}|  \\
        + \latticeVelocity \int_{0}^{\timeStep} \reduceSpaceDoubleInt\totalVariation{\boundaryFunction{\labelWest}(\timeVariable, \cdot)}{\realsPositive}\differential\timeVariable.
    \end{align*}
    Bounding corner terms by $\maximumInitialDatum$, we have 
    \begin{equation*}
        \spaceStep \sum_{\indexVelocity\in\setVelIndexes} \sum_{\indexSpace_{\yLabel}\in\naturals}|\distributionFunction_{\indexVelocity, 0, \indexSpace_{\yLabel}}^{1, \collided}  - \distributionFunctionLetter_{\indexVelocity}^{\atEquilibrium}(\boundaryDatumWest{\indexSpace_{\yLabel}}^{0})| \leq \spaceStep \sum_{\indexSpace_{\yLabel}\in\naturals} |\conservedVariableDiscrete_{0, \indexSpace_{\yLabel}}^{0} - \boundaryDatumWest{\indexSpace_{\yLabel}}^{0}| + \spaceStep \sum_{\indexSpace_{\yLabel}\in\naturals} |\conservedVariableDiscrete_{0, \indexSpace_{\yLabel}}^{0} - \conservedVariableDiscrete_{1, \indexSpace_{\yLabel}}^{0}|  + \latticeVelocity  \int_{0}^{\timeStep} \reduceSpaceDoubleInt\totalVariation{\boundaryFunction{\labelWest}(\timeVariable, \cdot)}{\realsPositive}\differential\timeVariable 
        + 2\spaceStep\maximumInitialDatum.
    \end{equation*}
    The first two terms on the right-hand side can be controlled by \eqref{eq:tmp11}, thus we obtain 
    \begin{multline*}
        \spaceStep \sum_{\indexVelocity\in\setVelIndexes} \sum_{\indexSpace_{\yLabel}\in\naturals}|\distributionFunction_{\indexVelocity, 0, \indexSpace_{\yLabel}}^{1, \collided}  - \distributionFunctionLetter_{\indexVelocity}^{\atEquilibrium}(\boundaryDatumWest{\indexSpace_{\yLabel}}^{0})| \leq 3\sup_{\xLabel>0}\lVert\conservedVariable^{\initial}(\xLabel, \cdot)\rVert_{\lebesgueSpace{1}(\realsPositive)}  
        + \sup_{\timeVariable\in(0, \finalTime)}\lVert\boundaryFunction{\labelWest}(\timeVariable, \cdot)\rVert_{\lebesgueSpace{1}(\realsPositive)} \\
        + \latticeVelocity  \int_{0}^{\timeStep} \reduceSpaceDoubleInt\totalVariation{\boundaryFunction{\labelWest}(\timeVariable, \cdot)}{\realsPositive}\differential\timeVariable 
        + 2\spaceStep\maximumInitialDatum.
    \end{multline*}

    Putting all these facts into \eqref{eq:tmp6}, we have 
    \begin{multline*}
        a^0 \leq \totalVariation{\conservedVariable^{\initial}}{(\reals_+^{*})^2} + 2\latticeVelocity  \int_{0}^{\timeStep} \reduceSpaceDoubleInt\totalVariation{\boundaryFunction{\labelWest}(\timeVariable, \cdot)}{\realsPositive}\differential\timeVariable + 2\latticeVelocity  \int_{0}^{\timeStep} \reduceSpaceDoubleInt\totalVariation{\boundaryFunction{\labelSouth}(\timeVariable, \cdot)}{\realsPositive}\differential\timeVariable  \\
        +3\sup_{\xLabel>0}\lVert\conservedVariable^{\initial}(\xLabel, \cdot)\rVert_{\lebesgueSpace{1}(\realsPositive)} +3\sup_{\yLabel>0}\lVert\conservedVariable^{\initial}(\cdot, \yLabel)\rVert_{\lebesgueSpace{1}(\realsPositive)} 
        + \sup_{\timeVariable\in(0, \finalTime)}\lVert\boundaryFunction{\labelWest}(\timeVariable, \cdot)\rVert_{\lebesgueSpace{1}(\realsPositive)} + \sup_{\timeVariable\in(0, \finalTime)}\lVert\boundaryFunction{\labelSouth}(\timeVariable, \cdot)\rVert_{\lebesgueSpace{1}(\realsPositive)} 
        + 4\spaceStep\maximumInitialDatum\\
        + \spaceStep\sum_{\indexSpace_{\yLabel}\in\naturals} |\boundaryDatumWest{\indexSpace_{\yLabel}}^{0} - \conservedVariableDiscrete^0_{0, \indexSpace_{\yLabel}}| + \spaceStep\sum_{\indexSpace_{\xLabel}\in\naturals} |\boundaryDatumSouth{\indexSpace_{\xLabel}}^{0}-\conservedVariableDiscrete_{\indexSpace_{\xLabel}, 0}^{0}|.
    \end{multline*}
    The last two terms on the right-hand side by \eqref{eq:tmp11}, we end up with
    \begin{multline*}
        a^0 \leq 2\totalVariation{\conservedVariable^{\initial}}{(\reals_+^{*})^2} + 2\latticeVelocity  \int_{0}^{\timeStep} \reduceSpaceDoubleInt\totalVariation{\boundaryFunction{\labelWest}(\timeVariable, \cdot)}{\realsPositive}\differential\timeVariable + 2\latticeVelocity  \int_{0}^{\timeStep} \reduceSpaceDoubleInt\totalVariation{\boundaryFunction{\labelSouth}(\timeVariable, \cdot)}{\realsPositive}\differential\timeVariable  \\
        +4\sup_{\xLabel>0}\lVert\conservedVariable^{\initial}(\xLabel, \cdot)\rVert_{\lebesgueSpace{1}(\realsPositive)} +4\sup_{\yLabel>0}\lVert\conservedVariable^{\initial}(\cdot, \yLabel)\rVert_{\lebesgueSpace{1}(\realsPositive)}  
        + 2\sup_{\timeVariable\in(0, \finalTime)}\lVert\boundaryFunction{\labelWest}(\timeVariable, \cdot)\rVert_{\lebesgueSpace{1}(\realsPositive)} + 2\sup_{\timeVariable\in(0, \finalTime)}\lVert\boundaryFunction{\labelSouth}(\timeVariable, \cdot)\rVert_{\lebesgueSpace{1}(\realsPositive)} 
        + 4\spaceStep\maximumInitialDatum.
    \end{multline*}

    Overall, the upper bound now reads
    \begin{multline*}
        \totalVariation{\distributionFunctionsAsFunction^{\indexTime}}{(\reals_+^{*})^2} 
        \leq 2\totalVariation{\conservedVariable^{\initial}}{(\reals_+^{*})^2} 
        +4\sup_{\xLabel>0}\lVert\conservedVariable^{\initial}(\xLabel, \cdot)\rVert_{\lebesgueSpace{1}(\realsPositive)} +4\sup_{\yLabel>0}\lVert\conservedVariable^{\initial}(\cdot, \yLabel)\rVert_{\lebesgueSpace{1}(\realsPositive)} \\
        + \totalVariation{\boundaryFunction{\labelWest}}{(0, \finalTime)\times \realsPositive} +  \totalVariation{\boundaryFunction{\labelSouth}}{(0, \finalTime)\times \realsPositive} 
        +4\latticeVelocity \finalTime  \maximumInitialDatum
        +2\latticeVelocity\int_{0}^{\finalTime} \reduceSpaceDoubleInt\totalVariation{\boundaryFunction{\labelWest}(\timeVariable, \cdot)}{\realsPositive}\differential\timeVariable
        +2\latticeVelocity
        \int_{0}^{\finalTime} \reduceSpaceDoubleInt\totalVariation{\boundaryFunction{\labelSouth}(\timeVariable, \cdot)}{\realsPositive}\differential\timeVariable \\
        + 2\sup_{\timeVariable\in(0, \finalTime)}\lVert\boundaryFunction{\labelWest}(\timeVariable, \cdot)\rVert_{\lebesgueSpace{1}(\realsPositive)} + 2\sup_{\timeVariable\in(0, \finalTime)}\lVert\boundaryFunction{\labelSouth}(\timeVariable, \cdot)\rVert_{\lebesgueSpace{1}(\realsPositive)}.
    \end{multline*}    
    Finally, using \Cref{thm:approxThmBV} and \Cref{rem:1Dvs2D}, we have that 
    \begin{equation*}
        \int_{0}^{\finalTime} \reduceSpaceDoubleInt\totalVariation{\boundaryFunction{\labelWest}(\timeVariable, \cdot)}{\realsPositive}\differential\timeVariable \leq \totalVariation{\boundaryFunction{\labelWest}}{(0, \finalTime)\times \realsPositive}, \qquad \int_{0}^{\finalTime} \reduceSpaceDoubleInt\totalVariation{\boundaryFunction{\labelSouth}(\timeVariable, \cdot)}{\realsPositive}\differential\timeVariable \leq \totalVariation{\boundaryFunction{\labelSouth}}{(0, \finalTime)\times \realsPositive},
    \end{equation*}
    thus
    \begin{multline*}
        \totalVariation{\distributionFunctionsAsFunction^{\indexTime}}{(\reals_+^{*})^2} 
        \leq 2\totalVariation{\conservedVariable^{\initial}}{(\reals_+^{*})^2} 
        +4\sup_{\xLabel>0}\lVert\conservedVariable^{\initial}(\xLabel, \cdot)\rVert_{\lebesgueSpace{1}(\realsPositive)} +4\sup_{\yLabel>0}\lVert\conservedVariable^{\initial}(\cdot, \yLabel)\rVert_{\lebesgueSpace{1}(\realsPositive)} +4\latticeVelocity \finalTime  \maximumInitialDatum\\
        +2 \totalVariation{\boundaryFunction{\labelWest}}{(0, \finalTime)\times \realsPositive} + 2 \totalVariation{\boundaryFunction{\labelSouth}}{(0, \finalTime)\times \realsPositive} 
        + 2\sup_{\timeVariable\in(0, \finalTime)}\lVert\boundaryFunction{\labelWest}(\timeVariable, \cdot)\rVert_{\lebesgueSpace{1}(\realsPositive)} + 2\sup_{\timeVariable\in(0, \finalTime)}\lVert\boundaryFunction{\labelSouth}(\timeVariable, \cdot)\rVert_{\lebesgueSpace{1}(\realsPositive)}.
    \end{multline*}  
\end{proof}

\subsection{Closeness to equilibrium}

The last ingredient that we need is a result on the fact that the solution dwells within a $\bigO{\spaceStep}$-neighborhood of the equilibrium associated with the conserved moment $\conservedVariableDiscrete$.
This eventually fosters the ``elimination'' of the distribution functions to remain with $\conservedVariableDiscrete$ only.
\begin{proposition}[Closeness to equilibrium]\label{prop:convergenceEquilibrium}
    Let \eqref{eq:mononotonicityConditionsD2Q5} be fulfilled. Then, there exists $C_{\atEquilibrium}(\conservedVariable^{\initial}, \boundaryFunction{\labelWest}, \boundaryFunction{\labelSouth}, \finalTime)>0$ such that, for $\indexTime\in\integerInterval{0}{\numberTimeSteps}$
    \begin{equation*}
        \lVert\distributionFunctionsAsFunction^{\indexTime} - \vectorial{\distributionFunctionLetter}^{\atEquilibrium}(\conservedVariableDiscreteAsFunction^{\indexTime})\rVert_{\lebesgueSpace{1}((\reals_+^{*})^2)} \leq C_{\atEquilibrium}(\conservedVariable^{\initial}, \boundaryFunction{\labelWest}, \boundaryFunction{\labelSouth})\spaceStep\frac{\max(|1-\relaxationParameterSymmetric|, |1-\relaxationParameterAntiSymmetric|)^{\indexTime}-1}{\max(|1-\relaxationParameterSymmetric|, |1-\relaxationParameterAntiSymmetric|)-1}\leq \frac{C_{\atEquilibrium}(\conservedVariable^{\initial}, \boundaryFunction{\labelWest}, \boundaryFunction{\labelSouth})}{1-\max(|1-\relaxationParameterSymmetric|, |1-\relaxationParameterAntiSymmetric|)}\spaceStep,
    \end{equation*}
    where the denominators do not vanish thanks to \eqref{eq:boundRelaxationParam}.
\end{proposition}
\begin{proof}
    Let us denote $\delta^{\indexTime} \definitionEquality \lVert\distributionFunctionsAsFunction^{\indexTime} - \vectorial{\distributionFunctionLetter}^{\atEquilibrium}(\conservedVariableDiscreteAsFunction^{\indexTime})\rVert_{\lebesgueSpace{1}((\reals_+^{*})^2)}$.
    We have 
    \begin{align*}
        \delta^{\indexTime+1} = \spaceStep^2 \Bigl (\sum_{\vectorial{\indexSpace}\in\naturals^2} |\distributionFunction_{\labZeroVel, \vectorial{\indexSpace}}^{\indexTime, \collided} - \distributionFunctionLetter_{\labZeroVel}^{\atEquilibrium}(\conservedVariableDiscrete_{\vectorial{\indexSpace}}^{\indexTime + 1})| &+ \sum_{\indexSpace_{\yLabel}\in\naturals} |\distributionFunctionLetter_{\labPosX}^{\atEquilibrium}(\boundaryDatumWest{\indexSpace_{\yLabel}}^{\indexTime}) - \distributionFunctionLetter_{\labPosX}^{\atEquilibrium}(\conservedVariableDiscrete_{0, \indexSpace_{\yLabel}}^{\indexTime + 1})| + \sum_{\vectorial{\indexSpace}\in\naturals^2} |\distributionFunction_{\labPosX, \vectorial{\indexSpace}}^{\indexTime, \collided} - \distributionFunctionLetter_{\labPosX}^{\atEquilibrium}(\conservedVariableDiscrete_{\vectorial{\indexSpace} + \canonicalBasisVector{\xLabel}}^{\indexTime + 1})| + \sum_{\substack{\indexSpace_{\xLabel}\geq 1\\\indexSpace_{\yLabel}\in\naturals}} |\distributionFunction_{\labNegX, \vectorial{\indexSpace}}^{\indexTime, \collided} - \distributionFunctionLetter_{\labNegX}^{\atEquilibrium}(\conservedVariableDiscrete_{\vectorial{\indexSpace} - \canonicalBasisVector{\xLabel}}^{\indexTime + 1})| \\
        &+ \sum_{\indexSpace_{\xLabel}\in\naturals} |\distributionFunctionLetter_{\labPosY}^{\atEquilibrium}(\boundaryDatumSouth{\indexSpace_{\xLabel}}^{\indexTime}) - \distributionFunctionLetter_{\labPosY}^{\atEquilibrium}(\conservedVariableDiscrete_{\indexSpace_{\xLabel}, 0}^{\indexTime + 1}) | + \sum_{\vectorial{\indexSpace}\in\naturals^2} |\distributionFunction_{\labPosY, \vectorial{\indexSpace}}^{\indexTime, \collided} - \distributionFunctionLetter_{\labPosY}^{\atEquilibrium}(\conservedVariableDiscrete_{\vectorial{\indexSpace} + \canonicalBasisVector{\yLabel}}^{\indexTime + 1})| + \sum_{\substack{\indexSpace_{\xLabel}\in\naturals\\\indexSpace_{\yLabel}\geq 1}} |\distributionFunction_{\labNegY, \vectorial{\indexSpace}}^{\indexTime, \collided} - \distributionFunctionLetter_{\labNegY}^{\atEquilibrium}(\conservedVariableDiscrete_{\vectorial{\indexSpace} - \canonicalBasisVector{\yLabel}}^{\indexTime + 1})| \Bigr ).
    \end{align*}
    Using a $\flat$-procedure, we obtain---as with \eqref{eq:tmp11}:
    \begin{multline*}
        \spaceStep^2 \sum_{\indexSpace_{\yLabel}\in\naturals} |\distributionFunctionLetter_{\labPosX}^{\atEquilibrium}(\boundaryDatumWest{\indexSpace_{\yLabel}}^{\indexTime}) - \distributionFunctionLetter_{\labPosX}^{\atEquilibrium}(\conservedVariableDiscrete_{0, \indexSpace_{\yLabel}}^{\indexTime + 1})| \leq \spaceStep^2 \sum_{\indexSpace_{\yLabel}\in\naturals} |\boundaryDatumWest{\indexSpace_{\yLabel}}^{\indexTime} - \conservedVariableDiscrete_{0, \indexSpace_{\yLabel}}^{\indexTime + 1}| \leq \spaceStep^2 \sum_{\indexSpace_{\yLabel}\in\naturals} (|\boundaryDatumWest{\indexSpace_{\yLabel}}^{\indexTime}| + | \conservedVariableDiscrete_{0, \indexSpace_{\yLabel}}^{\indexTime + 1}|)\\
        \leq\latticeVelocity \int_{\timeGridPoint{\indexTime}}^{\timeGridPoint{\indexTime + 1}} \reduceSpaceDoubleInt\int_{0}^{+\infty}|\boundaryFunction{\labelWest}(\timeVariable, \yLabel)|\differential\yLabel\differential\timeVariable + \frac{1}{\timeStep} \int_{\timeGridPoint{\indexTime + 1}}^{\timeGridPoint{\indexTime + 2}} \reduceSpaceDoubleInt\int_{0}^{\spaceStep} \reduceSpaceDoubleInt\int_{0}^{+\infty}|\conservedVariableDiscreteAsFunction(\timeVariable, \xLabel, \yLabel)|\differential\yLabel\differential\xLabel\differential\timeVariable\\
        \leq \spaceStep \sup_{\timeVariable\in(0, \finalTime)} \lVert \boundaryFunction{\labelWest}(\timeVariable, \cdot) \rVert_{\lebesgueSpace{1}(\realsPositive)} 
        + \latticeVelocity \int_{\timeGridPoint{\indexTime + 1}}^{\timeGridPoint{\indexTime + 2}} \sup_{\xLabel>0} \lVert \conservedVariableDiscreteAsFunction(\timeVariable, \xLabel, \cdot) \rVert_{\lebesgueSpace{1}(\realsPositive)}\differential\timeVariable.
    \end{multline*}
    Thanks to \Cref{prop:L1Bound} and \Cref{prop:totalVariationEstimates}, the last integrand is uniformly bounded in time $\timeVariable$, therefore there exists a constant $C>0$ depending on the the data (and possibly on $\finalTime$) such that 
    \begin{align*}
        &\spaceStep^2 \sum_{\indexSpace_{\yLabel}\in\naturals} |\distributionFunctionLetter_{\labPosX}^{\atEquilibrium}(\boundaryDatumWest{\indexSpace_{\yLabel}}^{\indexTime}) - \distributionFunctionLetter_{\labPosX}^{\atEquilibrium}(\conservedVariableDiscrete_{0, \indexSpace_{\yLabel}}^{\indexTime + 1})|\leq C\spaceStep, \qquad \text{and analoguously, along the other axis}\\
        &\spaceStep^2\sum_{\indexSpace_{\xLabel}\in\naturals} |\distributionFunctionLetter_{\labPosY}^{\atEquilibrium}(\boundaryDatumSouth{\indexSpace_{\xLabel}}^{\indexTime}) - \distributionFunctionLetter_{\labPosY}^{\atEquilibrium}(\conservedVariableDiscrete_{\indexSpace_{\xLabel}, 0}^{\indexTime + 1}) | \leq C\spaceStep.
    \end{align*}

    For the other terms, adding and subtracting terms:
    \begin{align*}
        \circledast \definitionEquality \spaceStep^2 \Bigl (\sum_{\vectorial{\indexSpace}\in\naturals^2} |\distributionFunction_{\labZeroVel, \vectorial{\indexSpace}}^{\indexTime, \collided} - \distributionFunctionLetter_{\labZeroVel}^{\atEquilibrium}(\conservedVariableDiscrete_{\vectorial{\indexSpace}}^{\indexTime + 1})| &+ \sum_{\vectorial{\indexSpace}\in\naturals^2} |\distributionFunction_{\labPosX, \vectorial{\indexSpace}}^{\indexTime, \collided} - \distributionFunctionLetter_{\labPosX}^{\atEquilibrium}(\conservedVariableDiscrete_{\vectorial{\indexSpace} + \canonicalBasisVector{\xLabel}}^{\indexTime + 1})| + \sum_{\substack{\indexSpace_{\xLabel}\geq 1\\\indexSpace_{\yLabel}\in\naturals}} |\distributionFunction_{\labNegX, \vectorial{\indexSpace}}^{\indexTime, \collided} - \distributionFunctionLetter_{\labNegX}^{\atEquilibrium}(\conservedVariableDiscrete_{\vectorial{\indexSpace} - \canonicalBasisVector{\xLabel}}^{\indexTime + 1})| \\
        &+ \sum_{\vectorial{\indexSpace}\in\naturals^2} |\distributionFunction_{\labPosY, \vectorial{\indexSpace}}^{\indexTime, \collided} - \distributionFunctionLetter_{\labPosY}^{\atEquilibrium}(\conservedVariableDiscrete_{\vectorial{\indexSpace} + \canonicalBasisVector{\yLabel}}^{\indexTime + 1})| + \sum_{\substack{\indexSpace_{\xLabel}\in\naturals\\\indexSpace_{\yLabel}\geq 1}} |\distributionFunction_{\labNegY, \vectorial{\indexSpace}}^{\indexTime, \collided} - \distributionFunctionLetter_{\labNegY}^{\atEquilibrium}(\conservedVariableDiscrete_{\vectorial{\indexSpace} - \canonicalBasisVector{\yLabel}}^{\indexTime + 1})| \Bigr )\\
        = \spaceStep^2 \Bigl (&\sum_{\vectorial{\indexSpace}\in\naturals^2} |\distributionFunction_{\labZeroVel, \vectorial{\indexSpace}}^{\indexTime, \collided} - \distributionFunctionLetter_{\labZeroVel}^{\atEquilibrium}(\conservedVariableDiscrete_{\vectorial{\indexSpace}}^{\indexTime}) + \distributionFunctionLetter_{\labZeroVel}^{\atEquilibrium}(\conservedVariableDiscrete_{\vectorial{\indexSpace}}^{\indexTime})  - \distributionFunctionLetter_{\labZeroVel}^{\atEquilibrium}(\conservedVariableDiscrete_{\vectorial{\indexSpace}}^{\indexTime + 1})| \\
        + &\sum_{\vectorial{\indexSpace}\in\naturals^2} |\distributionFunction_{\labPosX, \vectorial{\indexSpace}}^{\indexTime, \collided} - \distributionFunctionLetter_{\labPosX}^{\atEquilibrium}(\conservedVariableDiscrete_{\vectorial{\indexSpace}}^{\indexTime}) + \distributionFunctionLetter_{\labPosX}^{\atEquilibrium}(\conservedVariableDiscrete_{\vectorial{\indexSpace}}^{\indexTime}) - \distributionFunctionLetter_{\labPosX}^{\atEquilibrium}(\conservedVariableDiscrete_{\vectorial{\indexSpace} + \canonicalBasisVector{\xLabel}}^{\indexTime}) + \distributionFunctionLetter_{\labPosX}^{\atEquilibrium}(\conservedVariableDiscrete_{\vectorial{\indexSpace} + \canonicalBasisVector{\xLabel}}^{\indexTime}) - \distributionFunctionLetter_{\labPosX}^{\atEquilibrium}(\conservedVariableDiscrete_{\vectorial{\indexSpace} + \canonicalBasisVector{\xLabel}}^{\indexTime + 1})| \\
        + &\sum_{\substack{\indexSpace_{\xLabel}\geq 1\\\indexSpace_{\yLabel}\in\naturals}} |\distributionFunction_{\labNegX, \vectorial{\indexSpace}}^{\indexTime, \collided} - \distributionFunctionLetter_{\labNegX}^{\atEquilibrium}(\conservedVariableDiscrete_{\vectorial{\indexSpace}}^{\indexTime}) + \distributionFunctionLetter_{\labNegX}^{\atEquilibrium}(\conservedVariableDiscrete_{\vectorial{\indexSpace}}^{\indexTime}) - \distributionFunctionLetter_{\labNegX}^{\atEquilibrium}(\conservedVariableDiscrete_{\vectorial{\indexSpace} - \canonicalBasisVector{\xLabel}}^{\indexTime}) + \distributionFunctionLetter_{\labNegX}^{\atEquilibrium}(\conservedVariableDiscrete_{\vectorial{\indexSpace} - \canonicalBasisVector{\xLabel}}^{\indexTime}) - \distributionFunctionLetter_{\labNegX}^{\atEquilibrium}(\conservedVariableDiscrete_{\vectorial{\indexSpace} - \canonicalBasisVector{\xLabel}}^{\indexTime + 1})| \\
        + &\sum_{\vectorial{\indexSpace}\in\naturals^2} |\distributionFunction_{\labPosY, \vectorial{\indexSpace}}^{\indexTime, \collided} - \distributionFunctionLetter_{\labPosY}^{\atEquilibrium}(\conservedVariableDiscrete_{\vectorial{\indexSpace}}^{\indexTime }) + \distributionFunctionLetter_{\labPosY}^{\atEquilibrium}(\conservedVariableDiscrete_{\vectorial{\indexSpace}}^{\indexTime }) - \distributionFunctionLetter_{\labPosY}^{\atEquilibrium}(\conservedVariableDiscrete_{\vectorial{\indexSpace} + \canonicalBasisVector{\yLabel}}^{\indexTime}) + \distributionFunctionLetter_{\labPosY}^{\atEquilibrium}(\conservedVariableDiscrete_{\vectorial{\indexSpace} + \canonicalBasisVector{\yLabel}}^{\indexTime}) - \distributionFunctionLetter_{\labPosY}^{\atEquilibrium}(\conservedVariableDiscrete_{\vectorial{\indexSpace} + \canonicalBasisVector{\yLabel}}^{\indexTime + 1})| \\
        + &\sum_{\substack{\indexSpace_{\xLabel}\in\naturals\\\indexSpace_{\yLabel}\geq 1}} |\distributionFunction_{\labNegY, \vectorial{\indexSpace}}^{\indexTime, \collided} - \distributionFunctionLetter_{\labNegY}^{\atEquilibrium}(\conservedVariableDiscrete_{\vectorial{\indexSpace}}^{\indexTime}) + \distributionFunctionLetter_{\labNegY}^{\atEquilibrium}(\conservedVariableDiscrete_{\vectorial{\indexSpace}}^{\indexTime}) - \distributionFunctionLetter_{\labNegY}^{\atEquilibrium}(\conservedVariableDiscrete_{\vectorial{\indexSpace} - \canonicalBasisVector{\yLabel}}^{\indexTime}) + \distributionFunctionLetter_{\labNegY}^{\atEquilibrium}(\conservedVariableDiscrete_{\vectorial{\indexSpace} - \canonicalBasisVector{\yLabel}}^{\indexTime}) - \distributionFunctionLetter_{\labNegY}^{\atEquilibrium}(\conservedVariableDiscrete_{\vectorial{\indexSpace} - \canonicalBasisVector{\yLabel}}^{\indexTime + 1})| \Bigr ).
    \end{align*}
    Using the triangle inequality and adding positive quantities on the right-hand side:
    \begin{align*}
        \circledast &\leq \spaceStep^2 \sum_{\indexVelocity\in\setVelIndexes} \sum_{\vectorial{\indexSpace}\in\naturals^2} |\collisionOperator_{\indexVelocity}(\vectorial{\distributionFunction_{\vectorial{\indexSpace}}}^{\indexTime}) - \distributionFunctionLetter_{\indexVelocity}^{\atEquilibrium}(\conservedVariableDiscrete_{\vectorial{\indexSpace}}^{\indexTime})| + 2\spaceStep\totalVariation{\conservedVariableDiscreteAsFunction^{\indexTime}}{(\reals_+^{*})^2} + \lVert\conservedVariableDiscreteAsFunction^{\indexTime + 1} - \conservedVariableDiscreteAsFunction^{\indexTime} \rVert_{\lebesgueSpace{1}((\reals_+^{*})^2)}\\
        &\leq \max(|1-\relaxationParameterSymmetric|, |1-\relaxationParameterAntiSymmetric|)\delta^{\indexTime}  + 2\spaceStep\totalVariation{\distributionFunctionsAsFunction^{\indexTime}}{(\reals_+^{*})^2} +  \lVert\distributionFunctionsAsFunction^{\indexTime + 1} - \distributionFunctionsAsFunction^{\indexTime} \rVert_{\lebesgueSpace{1}((\reals_+^{*})^2)}, 
    \end{align*}
    where the first term on the right-hand side of the second inequality is obtained as in the proof of \cite[Proposition 4.8]{aregba2025monotonicity}, and the others come from triangle inequalities.
    Using \Cref{prop:totalVariationEstimates} and \Cref{prop:equicontinuity}
    \begin{equation*}
        \delta^{\indexTime+1}\leq \max(|1-\relaxationParameterSymmetric|, |1-\relaxationParameterAntiSymmetric|)\delta^{\indexTime}  +  ( 2C + 2 C_{\mathscr{V}}(\conservedVariable^{\initial}, \boundaryFunction{\labelWest},  \boundaryFunction{\labelSouth}, \finalTime)  + C_{EC}(\conservedVariable^{\initial}, \boundaryFunction{\labelWest}, \boundaryFunction{\labelSouth}) ) \spaceStep .
    \end{equation*}
    Iterating and using that $\delta^0 = 0$, we obtain that there exists $C_{\atEquilibrium}(\conservedVariable^{\initial}, \boundaryFunction{\labelWest}, \boundaryFunction{\labelSouth}, \finalTime) \definitionEquality 2C + 2 C_{\mathscr{V}}(\conservedVariable^{\initial}, \boundaryFunction{\labelWest},  \boundaryFunction{\labelSouth}, \finalTime)  + C_{EC}(\conservedVariable^{\initial}, \boundaryFunction{\labelWest}, \boundaryFunction{\labelSouth}) >0$ such that
    \begin{equation*}
        \delta^{\indexTime}\leq C_{\atEquilibrium}(\conservedVariable^{\initial}, \boundaryFunction{\labelWest}, \boundaryFunction{\labelSouth}, \finalTime)\spaceStep\frac{\max(|1-\relaxationParameterSymmetric|, |1-\relaxationParameterAntiSymmetric|)^{\indexTime}-1}{\max(|1-\relaxationParameterSymmetric|, |1-\relaxationParameterAntiSymmetric|)-1}\leq \frac{C_{\atEquilibrium}(\conservedVariable^{\initial}, \boundaryFunction{\labelWest}, \boundaryFunction{\labelSouth}, \finalTime)}{1-\max(|1-\relaxationParameterSymmetric|, |1-\relaxationParameterAntiSymmetric|)}\spaceStep.
    \end{equation*}
\end{proof}

\subsection{Convergence of the discrete solution}

Now, combining \Cref{prop:L1Bound}, \ref{prop:equicontinuity}, \ref{prop:totalVariationEstimates}, and \ref{prop:convergenceEquilibrium} exactly in the way by \cite{aregba2024convergence, aregba2025monotonicity}, which follows the idea of proof by \cite{crandall1980monotone}, we prove what follows.
\begin{theorem}\label{thm:convergence}
    Let \eqref{eq:mononotonicityConditionsD2Q5} be satisfied and $\finalTime>0$ be fixed.
    Let $(\spaceStep_{p})_{p\in\naturals} $ be a sequence of non-negative space-steps such that $\lim_{p\to+\infty}\spaceStep_p = 0$.
    Then, there exists a subsequence of space-steps, also denoted $(\spaceStep_{p})_{p\in\naturals} $ for simplicity, and a function $\vectorial{\limitDistributionFunction}$ such that $\vectorial{\limitDistributionFunction}(\timeVariable, \cdot)\in\lebesgueSpace{1}((\reals_+^{*})^2)$ and $\vectorial{\limitDistributionFunction}(\timeVariable, \vectorial{\spaceVariable})\in\invariantCompactSetDistributions$ a.e. in $\vectorial{\spaceVariable}$, for $\timeVariable\in(0, \finalTime)$, such that
    \begin{equation}\label{eq:limit}
        \lim_{p\to +\infty} \lVert\distributionFunctionsAsFunctionWithStep{\discreteMark_p} - \vectorial{\limitDistributionFunction} \rVert_{\lebesgueTime{\infty}(0, \finalTime)\lebesgueInSpace{1}((\reals_+^{*})^2)} = 0,
    \end{equation}
    and, setting $\limitConservedMoment \definitionEquality \sum_{\indexVelocity=1}^{\indexVelocity=\numberVelocities}\limitDistributionFunction_{\indexVelocity}$, such that $\limitConservedMoment(\timeVariable, \vectorial{\spaceVariable})\in[-\maximumInitialDatum, \maximumInitialDatum]$ a.e. in $\vectorial{\spaceVariable}$, with $\lim_{p\to +\infty} \lVert\conservedVariableDiscreteAsFunctionWithStep{\discreteMark_p} -\limitConservedMoment \rVert_{\lebesgueTime{\infty}(0, \finalTime)\lebesgueInSpace{1}((\reals_+^{*})^2)} = 0$.
    Moreover, the limit distribution functions $\vectorial{\limitDistributionFunction}$ are at equilibrium, namely $\vectorial{\limitDistributionFunction}(\timeVariable, \vectorial{\spaceVariable}) = \vectorial{\distributionFunctionLetter}^{\atEquilibrium}(\limitConservedMoment(\timeVariable, \vectorial{\spaceVariable}))$ a.e. in $\vectorial{\spaceVariable}$.
\end{theorem}
This result ensures that the numerical solution converges---up to an extraction---to limit objects $\vectorial{\limitDistributionFunction}$  and $\limitConservedMoment$ in the $\lebesgueSpace{1}$-topology, uniformly in time. Moreover, the former equals the equilibrium corresponding to the latter.

\subsection{Entropy inequality}

What now remains to do is to show that $\limitConservedMoment$ is indeed the entropy solution of the problem at hand.

\begin{theorem}
    Under the same assumptions as \Cref{thm:convergence}, the limit $\limitConservedMoment$ satisfies \Cref{def:weakSolution}.
\end{theorem}
\begin{proof}
    We consider Krushkov kinetic entropies $\kineticEntropy_{\indexVelocity}(\distributionFunction_{\indexVelocity})\definitionEquality |\distributionFunction_{\indexVelocity}-\distributionFunctionLetter_{\indexVelocity}^{\atEquilibrium}(\krushkovParameter)|$ for $\krushkovParameter\in\reals$.
    We have, for $\vectorial{\indexSpace}\in\naturals^2$
    \begin{equation*}
        \sum_{\indexVelocity\in\setVelIndexes} \kineticEntropy_{\indexVelocity}(\distributionFunction_{\indexVelocity, \vectorial{\indexSpace}}^{\indexTime+1, \collided}) = \sum_{\indexVelocity\in\setVelIndexes} |\collisionOperator_{\indexVelocity}(\vectorial{\distributionFunction}_{\vectorial{\indexSpace}}^{\indexTime+1})-\collisionOperator_{\indexVelocity}(\vectorial{\distributionFunctionLetter}^{\atEquilibrium}(\krushkovParameter))|\leq \sum_{\indexVelocity\in\setVelIndexes} \kineticEntropy_{\indexVelocity}(\distributionFunction_{\indexVelocity, \vectorial{\indexSpace}}^{\indexTime+1}),
    \end{equation*}
    using \Cref{prop:contractivity}.
    We specifically study what happens close to the boundary.
    For the corner:
    \begin{equation}\label{eq:corner}
        \sum_{\indexVelocity\in\setVelIndexes} \kineticEntropy_{\indexVelocity}(\distributionFunction_{\indexVelocity, 0, 0}^{\indexTime+1, \collided}) \leq \sum_{\indexVelocity\in\setVelIndexes} \kineticEntropy_{\indexVelocity}(\distributionFunction_{\indexVelocity, 0, 0}^{\indexTime+1}) = \kineticEntropy_{\labZeroVel}(\distributionFunction_{\labZeroVel, 0, 0}^{\indexTime, \collided}) + \kineticEntropy_{\labPosX}(\distributionFunctionLetter_{\labPosX}^{\atEquilibrium}(\boundaryDatumWest{0}^{\indexTime})) + \kineticEntropy_{\labNegX}(\distributionFunction_{\labNegX, 1, 0}^{\indexTime, \collided}) + \kineticEntropy_{\labPosY}(\distributionFunctionLetter_{\labPosY}^{\atEquilibrium}(\boundaryDatumSouth{0}^{\indexTime})) + \kineticEntropy_{\labNegY}(\distributionFunction_{\labNegY, 0, 1}^{\indexTime, \collided}).
    \end{equation}
    For the western boundary, let $\indexSpace_{\yLabel}\geq 1$:
    \begin{equation}\label{eq:west}
        \sum_{\indexVelocity\in\setVelIndexes} \kineticEntropy_{\indexVelocity}(\distributionFunction_{\indexVelocity, 0, \indexSpace_{\yLabel}}^{\indexTime+1, \collided}) \leq \sum_{\indexVelocity\in\setVelIndexes} \kineticEntropy_{\indexVelocity}(\distributionFunction_{\indexVelocity, 0, \indexSpace_{\yLabel}}^{\indexTime+1}) = \kineticEntropy_{\labZeroVel}(\distributionFunction_{\labZeroVel, 0, \indexSpace_{\yLabel}}^{\indexTime, \collided}) + \kineticEntropy_{\labPosX}(\distributionFunctionLetter_{\labPosX}^{\atEquilibrium}(\boundaryDatumWest{\indexSpace_{\yLabel}}^{\indexTime})) + \kineticEntropy_{\labNegX}(\distributionFunction_{\labNegX, 1, \indexSpace_{\yLabel}}^{\indexTime, \collided}) + \kineticEntropy_{\labPosY}(\distributionFunction_{\labPosY, 0, \indexSpace_{\yLabel} - 1}^{\indexTime, \collided}) + \kineticEntropy_{\labNegY}(\distributionFunction_{\labNegY, 0, \indexSpace_{\yLabel} + 1}^{\indexTime, \collided}).
    \end{equation}
    For the southern boundary, let $\indexSpace_{\xLabel}\geq 1$:
    \begin{equation}\label{eq:south}
        \sum_{\indexVelocity\in\setVelIndexes} \kineticEntropy_{\indexVelocity}(\distributionFunction_{\indexVelocity, \indexSpace_{\xLabel}, 0}^{\indexTime+1, \collided}) \leq \sum_{\indexVelocity\in\setVelIndexes} \kineticEntropy_{\indexVelocity}(\distributionFunction_{\indexVelocity, \indexSpace_{\xLabel}, 0}^{\indexTime+1}) = \kineticEntropy_{\labZeroVel}(\distributionFunction_{\labZeroVel, \indexSpace_{\xLabel}, 0}^{\indexTime, \collided}) + \kineticEntropy_{\labPosX}(\distributionFunction_{\labPosX, \indexSpace_{\xLabel}-1, 0}^{\indexTime, \collided})+ \kineticEntropy_{\labNegX}(\distributionFunction_{\labNegX, \indexSpace_{\xLabel}+1, 0}^{\indexTime, \collided}) + \kineticEntropy_{\labPosY}(\distributionFunctionLetter_{\labPosY}^{\atEquilibrium}(\boundaryDatumSouth{\indexSpace_{\xLabel}}^{\indexTime})) + \kineticEntropy_{\labNegY}(\distributionFunction_{\labNegY, \indexSpace_{\xLabel}, 1}^{\indexTime, \collided}).
    \end{equation}
    Away from the boundary, that is for $\indexSpace_{\xLabel}, \indexSpace_{\yLabel}\geq 1$:
    \begin{equation}\label{eq:inner2}
        \sum_{\indexVelocity\in\setVelIndexes} \kineticEntropy_{\indexVelocity}(\distributionFunction_{\indexVelocity, \vectorial{\indexSpace}}^{\indexTime+1, \collided}) \leq \sum_{\indexVelocity\in\setVelIndexes} \kineticEntropy_{\indexVelocity}(\distributionFunction_{\indexVelocity, \vectorial{\indexSpace}}^{\indexTime+1}) = \sum_{\indexVelocity\in\setVelIndexes} \kineticEntropy_{\indexVelocity}(\distributionFunction_{\indexVelocity, \vectorial{\indexSpace} - \vectorial{\discreteVelocityLetter}_{\indexVelocity}/\latticeVelocity}^{\indexTime, \collided}).
    \end{equation}
    Consider a test function $\testFunction\in\smoothFunctionsSpace([0, \finalTime)\times [0, +\infty)^2)$ such that $\testFunction\geq 0$, and its discretization
    \begin{equation*}
        \testFunctionDiscrete_{\vectorial{\indexSpace}}^{\indexTime}\definitionEquality \dashint_{\timeGridPoint{\indexTime}}^{\timeGridPoint{\indexTime + 1}}\reduceSpaceDoubleInt\dashint_{\cell{\vectorial{\indexSpace}}} \testFunction(\timeVariable, \vectorial{\spaceVariable})\differential\vectorial{\spaceVariable}\differential\timeVariable.
    \end{equation*}
    We obtain, using \eqref{eq:corner}, \eqref{eq:west}, \eqref{eq:south}, and \eqref{eq:inner2}:
    \begin{multline*}
        \timeStep\spaceStep^2\sum_{\indexTime=0}^{\numberTimeSteps-1}\sum_{\vectorial{\indexSpace}\in\naturals^2}\frac{\sum_{\indexVelocity\in\setVelIndexes} \kineticEntropy_{\indexVelocity}(\distributionFunction_{\indexVelocity, \vectorial{\indexSpace}}^{\indexTime+1, \collided}) - \sum_{\indexVelocity\in\setVelIndexes} \kineticEntropy_{\indexVelocity}(\distributionFunction_{\indexVelocity, \vectorial{\indexSpace}}^{\indexTime, \collided})}{\timeStep} \testFunctionDiscrete_{\vectorial{\indexSpace}}^{\indexTime}\leq \spaceStep^2\sum_{\indexTime=0}^{\numberTimeSteps-1}\sum_{\substack{\indexSpace_{\xLabel}\geq 1\\\indexSpace_{\yLabel}\geq 1}} \sum_{\indexVelocity\in\setVelIndexes\smallsetminus\{\labZeroVel\}} \Bigl ( \kineticEntropy_{\indexVelocity}(\distributionFunction_{\indexVelocity, \vectorial{\indexSpace} - \vectorial{\discreteVelocityLetter}_{\indexVelocity}/\latticeVelocity}^{\indexTime, \collided}) - \kineticEntropy_{\indexVelocity}(\distributionFunction_{\indexVelocity, \vectorial{\indexSpace}}^{\indexTime, \collided})  \Bigr ) \testFunctionDiscrete_{\vectorial{\indexSpace}}^{\indexTime}\\
        +\spaceStep^2\sum_{\indexTime=0}^{\numberTimeSteps-1} \Bigl ( \kineticEntropy_{\labPosX}(\distributionFunctionLetter_{\labPosX}^{\atEquilibrium}(\boundaryDatumWest{0}^{\indexTime})) + \kineticEntropy_{\labNegX}(\distributionFunction_{\labNegX, 1, 0}^{\indexTime, \collided}) + \kineticEntropy_{\labPosY}(\distributionFunctionLetter_{\labPosY}^{\atEquilibrium}(\boundaryDatumSouth{0}^{\indexTime})) + \kineticEntropy_{\labNegY}(\distributionFunction_{\labNegY, 0, 1}^{\indexTime, \collided})- \sum_{\indexVelocity\in\setVelIndexes\smallsetminus\{\labZeroVel\}} \kineticEntropy_{\indexVelocity}(\distributionFunction_{\indexVelocity, 0, 0}^{\indexTime, \collided})\Bigr )\testFunctionDiscrete_{0, 0}^{\indexTime}\\
        +\spaceStep^2\sum_{\indexTime=0}^{\numberTimeSteps-1}\sum_{\indexSpace_{\yLabel}\geq 1}  \Bigl (\kineticEntropy_{\labPosX}(\distributionFunctionLetter_{\labPosX}^{\atEquilibrium}(\boundaryDatumWest{\indexSpace_{\yLabel}}^{\indexTime})) + \kineticEntropy_{\labNegX}(\distributionFunction_{\labNegX, 1, \indexSpace_{\yLabel}}^{\indexTime, \collided}) + \kineticEntropy_{\labPosY}(\distributionFunction_{\labPosY, 0, \indexSpace_{\yLabel} - 1}^{\indexTime, \collided}) + \kineticEntropy_{\labNegY}(\distributionFunction_{\labNegY, 0, \indexSpace_{\yLabel} + 1}^{\indexTime, \collided})  - \sum_{\indexVelocity\in\setVelIndexes\smallsetminus\{\labZeroVel\}} \kineticEntropy_{\indexVelocity}(\distributionFunction_{\indexVelocity, 0, \indexSpace_{\yLabel}}^{\indexTime, \collided})\Bigr )\testFunctionDiscrete_{0, \indexSpace_{\yLabel}}^{\indexTime}\\
        +\spaceStep^2\sum_{\indexTime=0}^{\numberTimeSteps-1}\sum_{\indexSpace_{\xLabel}\geq 1}  \Bigl (\kineticEntropy_{\labPosX}(\distributionFunction_{\labPosX, \indexSpace_{\xLabel}-1, 0}^{\indexTime, \collided})+ \kineticEntropy_{\labNegX}(\distributionFunction_{\labNegX, \indexSpace_{\xLabel}+1, 0}^{\indexTime, \collided}) + \kineticEntropy_{\labPosY}(\distributionFunctionLetter_{\labPosY}^{\atEquilibrium}(\boundaryDatumSouth{\indexSpace_{\xLabel}}^{\indexTime})) + \kineticEntropy_{\labNegY}(\distributionFunction_{\labNegY, \indexSpace_{\xLabel}, 1}^{\indexTime, \collided}) - \sum_{\indexVelocity\in\setVelIndexes\smallsetminus\{\labZeroVel\}} \kineticEntropy_{\indexVelocity}(\distributionFunction_{\indexVelocity, \indexSpace_{\xLabel}, 0}^{\indexTime, \collided})\Bigr )\testFunctionDiscrete_{\indexSpace_{\xLabel}, 0}^{\indexTime} =: \circledast.
    \end{multline*}
    The left-hand side is the same as in \cite{aregba2025monotonicity}, so it undergoes the same treatment.
    The term on the second row of the previous inequality (the corner term) is a $\bigO{\spaceStep}$ term, hence goes to zero with $\spaceStep$ and shall therefore be neglected.
    The right-hand side therefore reads, performing changes in the indexes
    \begin{multline*}
        \circledast = \bigO{\spaceStep} \\
        + \spaceStep^2\sum_{\indexTime=0}^{\numberTimeSteps-1}\Bigl ( \sum_{\substack{\indexSpace_{\xLabel}\geq 0\\\indexSpace_{\yLabel}\geq 1}} \kineticEntropy_{\labPosX}(\distributionFunction_{\labPosX, \vectorial{\indexSpace}}^{\indexTime, \collided}) \testFunctionDiscrete_{\vectorial{\indexSpace} + \canonicalBasisVector{\xLabel}}^{\indexTime} + \sum_{\substack{\indexSpace_{\xLabel}\geq 2\\\indexSpace_{\yLabel}\geq 1}} \kineticEntropy_{\labNegX}(\distributionFunction_{\labNegX, \vectorial{\indexSpace}}^{\indexTime, \collided}) \testFunctionDiscrete_{\vectorial{\indexSpace}-\canonicalBasisVector{\xLabel}}^{\indexTime} + \sum_{\substack{\indexSpace_{\xLabel}\geq 1\\\indexSpace_{\yLabel}\geq 0}} \kineticEntropy_{\labPosY}(\distributionFunction_{\labPosY, \vectorial{\indexSpace}}^{\indexTime, \collided}) \testFunctionDiscrete_{\vectorial{\indexSpace}+\canonicalBasisVector{\yLabel}}^{\indexTime} +  \sum_{\substack{\indexSpace_{\xLabel}\geq 1\\\indexSpace_{\yLabel}\geq 2}} \kineticEntropy_{\labNegY}(\distributionFunction_{\labNegY, \vectorial{\indexSpace}}^{\indexTime, \collided}) \testFunctionDiscrete_{\vectorial{\indexSpace}-\canonicalBasisVector{\yLabel}}^{\indexTime} 
        - \sum_{\substack{\indexSpace_{\xLabel}\geq 1\\\indexSpace_{\yLabel}\geq 1}} \sum_{\indexVelocity\in\setVelIndexes\smallsetminus\{\labZeroVel\}} \kineticEntropy_{\indexVelocity}(\distributionFunction_{\indexVelocity, \vectorial{\indexSpace}}^{\indexTime, \collided})  \testFunctionDiscrete_{\vectorial{\indexSpace}}^{\indexTime} \\
        +\sum_{\indexSpace_{\yLabel}\geq 1}  \Bigl (\kineticEntropy_{\labPosX}(\distributionFunctionLetter_{\labPosX}^{\atEquilibrium}(\boundaryDatumWest{\indexSpace_{\yLabel}}^{\indexTime})) + \kineticEntropy_{\labNegX}(\distributionFunction_{\labNegX, 1, \indexSpace_{\yLabel}}^{\indexTime, \collided}) + \kineticEntropy_{\labPosY}(\distributionFunction_{\labPosY, 0, \indexSpace_{\yLabel} - 1}^{\indexTime, \collided}) + \kineticEntropy_{\labNegY}(\distributionFunction_{\labNegY, 0, \indexSpace_{\yLabel} + 1}^{\indexTime, \collided})  - \sum_{\indexVelocity\in\setVelIndexes\smallsetminus\{\labZeroVel\}} \kineticEntropy_{\indexVelocity}(\distributionFunction_{\indexVelocity, 0, \indexSpace_{\yLabel}}^{\indexTime, \collided})\Bigr )\testFunctionDiscrete_{0, \indexSpace_{\yLabel}}^{\indexTime}\\
        +\sum_{\indexSpace_{\xLabel}\geq 1}  \Bigl (\kineticEntropy_{\labPosX}(\distributionFunction_{\labPosX, \indexSpace_{\xLabel}-1, 0}^{\indexTime, \collided})+ \kineticEntropy_{\labNegX}(\distributionFunction_{\labNegX, \indexSpace_{\xLabel}+1, 0}^{\indexTime, \collided}) + \kineticEntropy_{\labPosY}(\distributionFunctionLetter_{\labPosY}^{\atEquilibrium}(\boundaryDatumSouth{\indexSpace_{\xLabel}}^{\indexTime})) + \kineticEntropy_{\labNegY}(\distributionFunction_{\labNegY, \indexSpace_{\xLabel}, 1}^{\indexTime, \collided}) - \sum_{\indexVelocity\in\setVelIndexes\smallsetminus\{\labZeroVel\}} \kineticEntropy_{\indexVelocity}(\distributionFunction_{\indexVelocity, \indexSpace_{\xLabel}, 0}^{\indexTime, \collided})\Bigr )\testFunctionDiscrete_{\indexSpace_{\xLabel}, 0}^{\indexTime}
        \Bigr ).
    \end{multline*}
    This simplifies into
    \begin{multline*}
        \circledast = \bigO{\spaceStep}
        + \spaceStep^2\sum_{\indexTime=0}^{\numberTimeSteps-1}\Bigl (\sum_{\substack{\indexSpace_{\xLabel}\geq 1\\\indexSpace_{\yLabel}\geq 1}} \sum_{\indexVelocity\in\setVelIndexes\smallsetminus\{\labZeroVel\}} \kineticEntropy_{\indexVelocity}(\distributionFunction_{\indexVelocity, \vectorial{\indexSpace}}^{\indexTime, \collided})  (\testFunctionDiscrete_{\vectorial{\indexSpace} + \vectorial{\discreteVelocityLetter}_{\indexVelocity}/\latticeVelocity}^{\indexTime} - \testFunctionDiscrete_{\vectorial{\indexSpace}}^{\indexTime}) \\
        +\sum_{\indexSpace_{\yLabel}\geq 1}  \Bigl ( \Bigl (\kineticEntropy_{\labPosX}(\distributionFunctionLetter_{\labPosX}^{\atEquilibrium}(\boundaryDatumWest{\indexSpace_{\yLabel}}^{\indexTime})) + \kineticEntropy_{\labPosY}(\distributionFunction_{\labPosY, 0, \indexSpace_{\yLabel} - 1}^{\indexTime, \collided}) + \kineticEntropy_{\labNegY}(\distributionFunction_{\labNegY, 0, \indexSpace_{\yLabel} + 1}^{\indexTime, \collided})  - \sum_{\indexVelocity\in\setVelIndexes\smallsetminus\{\labZeroVel\}} \kineticEntropy_{\indexVelocity}(\distributionFunction_{\indexVelocity, 0, \indexSpace_{\yLabel}}^{\indexTime, \collided})\Bigr )\testFunctionDiscrete_{0, \indexSpace_{\yLabel}}^{\indexTime} + \kineticEntropy_{\labPosX}(\distributionFunction_{\labPosX, 0, \indexSpace_{\yLabel}}^{\indexTime, \collided}) \testFunctionDiscrete_{1, \indexSpace_{\yLabel}}^{\indexTime} \Bigr )\\
        +\sum_{\indexSpace_{\xLabel}\geq 1}  \Bigl (\Bigl (\kineticEntropy_{\labPosX}(\distributionFunction_{\labPosX, \indexSpace_{\xLabel}-1, 0}^{\indexTime, \collided})+ \kineticEntropy_{\labNegX}(\distributionFunction_{\labNegX, \indexSpace_{\xLabel}+1, 0}^{\indexTime, \collided}) + \kineticEntropy_{\labPosY}(\distributionFunctionLetter_{\labPosY}^{\atEquilibrium}(\boundaryDatumSouth{\indexSpace_{\xLabel}}^{\indexTime}))  - \sum_{\indexVelocity\in\setVelIndexes\smallsetminus\{\labZeroVel\}} \kineticEntropy_{\indexVelocity}(\distributionFunction_{\indexVelocity, \indexSpace_{\xLabel}, 0}^{\indexTime, \collided})\Bigr )\testFunctionDiscrete_{\indexSpace_{\xLabel}, 0}^{\indexTime} + \kineticEntropy_{\labPosY}(\distributionFunction_{\labPosY, \indexSpace_{\xLabel}, 0}^{\indexTime, \collided}) \testFunctionDiscrete_{\indexSpace_{\xLabel}, 1}^{\indexTime}\Bigr )
        \Bigr ).
    \end{multline*}
    Let us discuss term-by-term. From now on, limits are taken for the space step going to zero, assuming that we have already extracted the needed subsequences.
    \begin{multline*}
        \spaceStep^2\sum_{\indexTime=0}^{\numberTimeSteps-1} \sum_{\substack{\indexSpace_{\xLabel}\geq 1\\\indexSpace_{\yLabel}\geq 1}} \sum_{\indexVelocity\in\setVelIndexes\smallsetminus\{\labZeroVel\}} \kineticEntropy_{\indexVelocity}(\distributionFunction_{\indexVelocity, \vectorial{\indexSpace}}^{\indexTime, \collided})  (\testFunctionDiscrete_{\vectorial{\indexSpace} + \vectorial{\discreteVelocityLetter}_{\indexVelocity}/\latticeVelocity}^{\indexTime} - \testFunctionDiscrete_{\vectorial{\indexSpace}}^{\indexTime}) = \int_{0}^{\finalTime}\reduceSpaceDoubleInt\int_{(\spaceStep, +\infty)^2} \sum_{\indexVelocity\in\setVelIndexes\smallsetminus\{\labZeroVel\}} \kineticEntropy_{\indexVelocity}(\distributionFunctionsAsFunctionComponent{\indexVelocity}^{\collided}(\timeVariable, \vectorial{\spaceVariable}))\frac{\testFunction(\timeVariable, \vectorial{\spaceVariable}+\vectorial{\discreteVelocityLetter}_{\indexVelocity}\spaceStep/\latticeVelocity)-\testFunction(\timeVariable, \vectorial{\spaceVariable})}{\timeStep}\differential\vectorial{\spaceVariable}\differential\timeVariable\\
        \to\int_0^{\finalTime}\reduceSpaceDoubleInt \int_{(\reals_+^{*})^2} \sign(\limitConservedMoment(\timeVariable, \vectorial{\spaceVariable})-\krushkovParameter)((\flux_{\xLabel}(\limitConservedMoment(\timeVariable, \vectorial{\spaceVariable}))-\flux_{\xLabel}(\krushkovParameter))\partial_{\xLabel}\testFunction(\timeVariable, \vectorial{\spaceVariable})
        +(\flux_{\yLabel}(\limitConservedMoment(\timeVariable, \vectorial{\spaceVariable}))-\flux_{\yLabel}(\krushkovParameter))\partial_{\yLabel}\testFunction(\timeVariable, \vectorial{\spaceVariable}))\differential\vectorial{\spaceVariable}\differential\timeVariable,
    \end{multline*}
    thanks to the same way of proceeding as \cite{aregba2025monotonicity}.
    For the other terms
    \begin{multline*}
        \spaceStep^2\sum_{\indexTime=0}^{\numberTimeSteps-1} \sum_{\indexSpace_{\yLabel}\geq 1}  \Bigl ( \Bigl (\kineticEntropy_{\labPosX}(\distributionFunctionLetter_{\labPosX}^{\atEquilibrium}(\boundaryDatumWest{\indexSpace_{\yLabel}}^{\indexTime})) + \kineticEntropy_{\labPosY}(\distributionFunction_{\labPosY, 0, \indexSpace_{\yLabel} - 1}^{\indexTime, \collided}) + \kineticEntropy_{\labNegY}(\distributionFunction_{\labNegY, 0, \indexSpace_{\yLabel} + 1}^{\indexTime, \collided})  - \sum_{\indexVelocity\in\setVelIndexes\smallsetminus\{\labZeroVel\}} \kineticEntropy_{\indexVelocity}(\distributionFunction_{\indexVelocity, 0, \indexSpace_{\yLabel}}^{\indexTime, \collided})\Bigr )\testFunctionDiscrete_{0, \indexSpace_{\yLabel}}^{\indexTime} + \kineticEntropy_{\labPosX}(\distributionFunction_{\labPosX, 0, \indexSpace_{\yLabel}}^{\indexTime, \collided}) \testFunctionDiscrete_{1, \indexSpace_{\yLabel}}^{\indexTime} \Bigr )\\
        = \spaceStep^2\sum_{\indexTime=0}^{\numberTimeSteps-1} \Bigl (  \sum_{\indexSpace_{\yLabel}\geq 1} \Bigl (\kineticEntropy_{\labPosX}(\distributionFunctionLetter_{\labPosX}^{\atEquilibrium}(\boundaryDatumWest{\indexSpace_{\yLabel}}^{\indexTime}))\testFunctionDiscrete_{0, \indexSpace_{\yLabel}}^{\indexTime} + \kineticEntropy_{\labPosX}(\distributionFunction_{\labPosX, 0, \indexSpace_{\yLabel}}^{\indexTime, \collided}) (\testFunctionDiscrete_{1, \indexSpace_{\yLabel}}^{\indexTime}-\testFunctionDiscrete_{0, \indexSpace_{\yLabel}}^{\indexTime})  - \kineticEntropy_{\labNegX}(\distributionFunction_{\labNegX, 0, \indexSpace_{\yLabel}}^{\indexTime, \collided}) \testFunctionDiscrete_{0, \indexSpace_{\yLabel}}^{\indexTime} \Bigr ) \\
        +\kineticEntropy_{\labPosY}(\distributionFunction_{\labPosY, 0, 0}^{\indexTime, \collided})\testFunctionDiscrete_{0, 1}^{\indexTime} + \sum_{\indexSpace_{\yLabel}\geq 1} \kineticEntropy_{\labPosY}(\distributionFunction_{\labPosY, 0, \indexSpace_{\yLabel}}^{\indexTime, \collided})(\testFunctionDiscrete_{0, \indexSpace_{\yLabel}+1}^{\indexTime} - \testFunctionDiscrete_{0, \indexSpace_{\yLabel}}^{\indexTime}) + \sum_{\indexSpace_{\yLabel}\geq 1} \kineticEntropy_{\labNegY}(\distributionFunction_{\labNegY, 0, \indexSpace_{\yLabel}}^{\indexTime, \collided})(\testFunctionDiscrete_{0, \indexSpace_{\yLabel}-1}^{\indexTime} - \testFunctionDiscrete_{0, \indexSpace_{\yLabel}}^{\indexTime}) - \kineticEntropy_{\labNegY}(\distributionFunction_{\labNegY, 0, 1}^{\indexTime, \collided}) \testFunctionDiscrete_{0,0}^{\indexTime} \Bigr ) \\
        =\latticeVelocity\int_0^{\finalTime}\reduceSpaceDoubleInt\int_{\spaceStep}^{+\infty}\Bigl ( \kineticEntropy_{\labPosX}(\distributionFunctionLetter_{\labPosX}^{\atEquilibrium}(\boundaryFunction{\labelWest}(\timeVariable, \yLabel)))\dashint_0^{\spaceStep} \testFunction(\timeVariable, \xLabel, \yLabel)\differential\xLabel + \kineticEntropy_{\labPosX}(\distributionFunctionsAsFunctionComponent{\labPosX}^{\collided}(\timeVariable, \tfrac{\spaceStep}{2}, \yLabel))\Bigl ( \dashint_{\spaceStep}^{2\spaceStep} \testFunction(\timeVariable, \xLabel, \yLabel)\differential\xLabel-\dashint_{0}^{\spaceStep} \testFunction(\timeVariable, \xLabel, \yLabel)\differential\xLabel\Bigr )\\
        - \kineticEntropy_{\labNegX}(\distributionFunctionsAsFunctionComponent{\labNegX}^{\collided}(\timeVariable, \tfrac{\spaceStep}{2}, \yLabel)) \dashint_{0}^{\spaceStep} \testFunction(\timeVariable, \xLabel, \yLabel)\differential\xLabel + \kineticEntropy_{\labPosY}(\distributionFunctionsAsFunctionComponent{\labPosY}^{\collided}(\timeVariable, \tfrac{\spaceStep}{2}, \yLabel))\dashint_0^{\spaceStep}(\testFunction(\timeVariable, \xLabel, \yLabel+\spaceStep)-\testFunction(\timeVariable, \xLabel, \yLabel))\differential\xLabel\\
        + \kineticEntropy_{\labNegY}(\distributionFunctionsAsFunctionComponent{\labNegY}^{\collided}(\timeVariable, \tfrac{\spaceStep}{2}, \yLabel))\dashint_0^{\spaceStep}(\testFunction(\timeVariable, \xLabel, \yLabel-\spaceStep)-\testFunction(\timeVariable, \xLabel, \yLabel))\differential\xLabel \Bigr ) \differential\yLabel\differential\timeVariable+ \bigO{\spaceStep}\\
        \to\latticeVelocity\int_0^{\finalTime}\reduceSpaceDoubleInt\int_{0}^{+\infty} \Bigl ( \kineticEntropy_{\labPosX}(\distributionFunctionLetter_{\labPosX}^{\atEquilibrium}(\boundaryFunction{\labelWest}(\timeVariable, \yLabel)))- \kineticEntropy_{\labNegX}(\distributionFunctionLetter_{\labNegX}^{\atEquilibrium}(\traceOperator_{\labelWest}(\limitConservedMoment)(\timeVariable, \yLabel))) \Bigr )\testFunction(\timeVariable, 0, \yLabel)\differential\yLabel\differential\timeVariable.
    \end{multline*}
    We apply \cite[Lemma 4.4]{aregba2004kinetic} along the $\xLabel$-axis, which holds because consistency \eqref{eq:consistency} along this axis involves only discrete velocities along the same direction (\idEst{} $\labPosX$ and $\labNegX$).
    This lemma here reads 
    \begin{lemma}
        Under \eqref{eq:mononotonicityConditionsD2Q5}, we have that 
        \begin{align*}
            \latticeVelocity (\kineticEntropy_{\labPosX}(\distributionFunctionLetter_{\labPosX}^{\atEquilibrium}(\boundaryFunction{\labelWest}))- \kineticEntropy_{\labNegX}(\distributionFunctionLetter_{\labNegX}^{\atEquilibrium}(\traceOperator_{\labelWest}(\limitConservedMoment)))) \leq \sign(\boundaryFunction{\labelWest}-\krushkovParameter)(\latticeVelocity\distributionFunctionLetter_{\labPosX}^{\atEquilibrium}(\boundaryFunction{\labelWest})) - \latticeVelocity\distributionFunctionLetter_{\labNegX}^{\atEquilibrium}(\traceOperator_{\labelWest}(\limitConservedMoment)) - \flux_{\xLabel}(\krushkovParameter)),\\
            \latticeVelocity (\kineticEntropy_{\labPosY}(\distributionFunctionLetter_{\labPosY}^{\atEquilibrium}(\boundaryFunction{\labelSouth}))- \kineticEntropy_{\labNegY}(\distributionFunctionLetter_{\labNegY}^{\atEquilibrium}(\traceOperator_{\labelSouth}(\limitConservedMoment)))) \leq \sign(\boundaryFunction{\labelSouth}-\krushkovParameter)(\latticeVelocity\distributionFunctionLetter_{\labPosY}^{\atEquilibrium}(\boundaryFunction{\labelSouth})) - \latticeVelocity\distributionFunctionLetter_{\labNegY}^{\atEquilibrium}(\traceOperator_{\labelSouth}(\limitConservedMoment)) - \flux_{\yLabel}(\krushkovParameter)).
        \end{align*}
    \end{lemma} 
    We obtain, in the limit
    \begin{multline}\label{eq:limitEntropyIneqAlmostFinished}
        -\int_0^{\finalTime}\int_{(\reals_+^{*})^2}|\limitConservedMoment-\krushkovParameter |\partial_{\timeVariable}\testFunction\differential\vectorial{\spaceVariable}\differential\timeVariable + \int_{(\reals_+^{*})^2}|\conservedVariable^{\initial}(\vectorial{\spaceVariable})-\krushkovParameter | \testFunction(0, \vectorial{\spaceVariable})\differential\vectorial{\spaceVariable}\\
        \leq \int_0^{\finalTime}\reduceSpaceDoubleInt \int_{(\reals_+^{*})^2} \sign(\limitConservedMoment-\krushkovParameter)((\flux_{\xLabel}(\limitConservedMoment)-\flux_{\xLabel}(\krushkovParameter))\partial_{\xLabel}\testFunction
        +(\flux_{\yLabel}(\limitConservedMoment)-\flux_{\yLabel}(\krushkovParameter))\partial_{\yLabel}\testFunction)\differential\vectorial{\spaceVariable}\differential\timeVariable\\
        +\int_0^{\finalTime}\reduceSpaceDoubleInt\int_{0}^{+\infty} \sign(\boundaryFunction{\labelWest}(\timeVariable, \yLabel) - \krushkovParameter)\Bigl ( \latticeVelocity \distributionFunctionLetter_{\labPosX}^{\atEquilibrium}(\boundaryFunction{\labelWest}(\timeVariable, \yLabel))- \latticeVelocity \distributionFunctionLetter_{\labNegX}^{\atEquilibrium}(\traceOperator_{\labelWest}(\limitConservedMoment)(\timeVariable, \yLabel)) - \flux_{\xLabel}(\krushkovParameter)\Bigr )\testFunction(\timeVariable, 0, \yLabel)\differential\yLabel\differential\timeVariable\\
        +\int_0^{\finalTime}\reduceSpaceDoubleInt\int_{0}^{+\infty} \sign(\boundaryFunction{\labelSouth}(\timeVariable, \xLabel) - \krushkovParameter)\Bigl ( \latticeVelocity \distributionFunctionLetter_{\labPosY}^{\atEquilibrium}(\boundaryFunction{\labelSouth}(\timeVariable, \xLabel))- \latticeVelocity \distributionFunctionLetter_{\labNegY}^{\atEquilibrium}(\traceOperator_{\labelSouth}(\limitConservedMoment)(\timeVariable, \xLabel)) - \flux_{\yLabel}(\krushkovParameter)\Bigr )\testFunction(\timeVariable, \xLabel, 0)\differential\xLabel\differential\timeVariable.
    \end{multline}
    We now proceed as in the discussion of the proof of Theorem 4.2 in \cite{aregba2004kinetic}: we take a specific family of test functions $\testFunction(\timeVariable, \xLabel, \yLabel) = \rho(\timeVariable, \yLabel)\max(0, 1-\xLabel/\eta)$ for $\rho\in\smoothFunctionsSpace((0, \finalTime)\times\reals_+^{*})$, with $\rho \geq 0$, and $\eta>0$.
    Inequality \eqref{eq:limitEntropyIneqAlmostFinished} thus becomes
    \begin{multline*}
        - \int_0^{\finalTime}\int_{0}^{+\infty}\dashint_{0}^{\eta}\sign(\limitConservedMoment-\krushkovParameter)(\flux_{\xLabel}(\limitConservedMoment)-\flux_{\xLabel}(\krushkovParameter))\rho(\timeVariable, \yLabel)\differential\xLabel\differential\yLabel\differential\timeVariable \\
        +\int_0^{\finalTime}\reduceSpaceDoubleInt\int_{0}^{+\infty} \sign(\boundaryFunction{\labelWest}(\timeVariable, \yLabel) - \krushkovParameter)\Bigl ( \latticeVelocity \distributionFunctionLetter_{\labPosX}^{\atEquilibrium}(\boundaryFunction{\labelWest}(\timeVariable, \yLabel))- \latticeVelocity \distributionFunctionLetter_{\labNegX}^{\atEquilibrium}(\traceOperator_{\labelWest}(\limitConservedMoment)(\timeVariable, \yLabel)) - \flux_{\xLabel}(\krushkovParameter)\Bigr )\rho(\timeVariable, \yLabel)\differential\yLabel\differential\timeVariable + \bigO{\eta}\geq 0,
    \end{multline*}
    where the $\bigO{\eta}$ are integrals of bounded functions with compact support with respect to $\yLabel$ (their last argument) over $(0, \finalTime)\times (0, \eta)\times \reals_+^{*}$.
    Letting $\eta\to 0$, we end up with 
    \begin{multline*}
        - \int_0^{\finalTime}\int_{0}^{+\infty}\sign(\traceOperator_{\labelWest}(\limitConservedMoment)(\timeVariable, \yLabel)-\krushkovParameter)(\flux_{\xLabel}(\traceOperator_{\labelWest}(\limitConservedMoment)(\timeVariable, \yLabel))-\flux_{\xLabel}(\krushkovParameter))\rho(\timeVariable, \yLabel)\differential\yLabel\differential\timeVariable \\
        +\int_0^{\finalTime}\reduceSpaceDoubleInt\int_{0}^{+\infty} \sign(\boundaryFunction{\labelWest}(\timeVariable, \yLabel) - \krushkovParameter)\Bigl ( \latticeVelocity \distributionFunctionLetter_{\labPosX}^{\atEquilibrium}(\boundaryFunction{\labelWest}(\timeVariable, \yLabel))- \latticeVelocity \distributionFunctionLetter_{\labNegX}^{\atEquilibrium}(\traceOperator_{\labelWest}(\limitConservedMoment)(\timeVariable, \yLabel)) - \flux_{\xLabel}(\krushkovParameter)\Bigr )\rho(\timeVariable, \yLabel)\differential\yLabel\differential\timeVariable\geq 0.
    \end{multline*}
    Since $\rho$ is smooth, positive, and arbitrary, we deduce that 
    \begin{equation*}
        \sign(\traceOperator_{\labelWest}(\limitConservedMoment)(\timeVariable, \yLabel)-\krushkovParameter)(\flux_{\xLabel}(\traceOperator_{\labelWest}(\limitConservedMoment)(\timeVariable, \yLabel))-\flux_{\xLabel}(\krushkovParameter))\leq \sign(\boundaryFunction{\labelWest}(\timeVariable, \yLabel) - \krushkovParameter)\Bigl ( \latticeVelocity \distributionFunctionLetter_{\labPosX}^{\atEquilibrium}(\boundaryFunction{\labelWest}(\timeVariable, \yLabel))- \latticeVelocity \distributionFunctionLetter_{\labNegX}^{\atEquilibrium}(\traceOperator_{\labelWest}(\limitConservedMoment)(\timeVariable, \yLabel)) - \flux_{\xLabel}(\krushkovParameter)\Bigr ),
    \end{equation*}
    for almost every $(\timeVariable, \yLabel)\in (0, \finalTime)\times \reals_+^{*}$.
    Taking $\krushkovParameter<-\maximumInitialDatum$ once, and $\krushkovParameter>\maximumInitialDatum$ another time, we deduce the following equalities almost everywhere:
    \begin{align*}
        \latticeVelocity \distributionFunctionLetter_{\labPosX}^{\atEquilibrium}(\boundaryFunction{\labelWest}(\timeVariable, \yLabel))- \latticeVelocity \distributionFunctionLetter_{\labNegX}^{\atEquilibrium}(\traceOperator_{\labelWest}(\limitConservedMoment)(\timeVariable, \yLabel))  &= \flux_{\xLabel}(\traceOperator_{\labelWest}(\limitConservedMoment)(\timeVariable, \yLabel)), \\
       \latticeVelocity \distributionFunctionLetter_{\labPosY}^{\atEquilibrium}(\boundaryFunction{\labelSouth}(\timeVariable, \xLabel))- \latticeVelocity \distributionFunctionLetter_{\labNegY}^{\atEquilibrium}(\traceOperator_{\labelSouth}(\limitConservedMoment)(\timeVariable, \xLabel))  &= \flux_{\yLabel}(\traceOperator_{\labelSouth}(\limitConservedMoment)(\timeVariable, \xLabel)),
    \end{align*}
    where the second one has been obtained throught the same procedure on $\yLabel$.
    Back into \eqref{eq:limitEntropyIneqAlmostFinished}, we obtain
    \begin{multline*}
        -\int_0^{\finalTime}\int_{(\reals_+^{*})^2}|\limitConservedMoment-\krushkovParameter |\partial_{\timeVariable}\testFunction\differential\vectorial{\spaceVariable}\differential\timeVariable + \int_{(\reals_+^{*})^2}|\conservedVariable^{\initial}(\vectorial{\spaceVariable})-\krushkovParameter | \testFunction(0, \vectorial{\spaceVariable})\differential\vectorial{\spaceVariable}\\
        \leq \int_0^{\finalTime}\reduceSpaceDoubleInt \int_{(\reals_+^{*})^2} \sign(\limitConservedMoment-\krushkovParameter)((\flux_{\xLabel}(\limitConservedMoment)-\flux_{\xLabel}(\krushkovParameter))\partial_{\xLabel}\testFunction
        +(\flux_{\yLabel}(\limitConservedMoment)-\flux_{\yLabel}(\krushkovParameter))\partial_{\yLabel}\testFunction)\differential\vectorial{\spaceVariable}\differential\timeVariable\\
        +\int_0^{\finalTime}\reduceSpaceDoubleInt\int_{0}^{+\infty} \sign(\boundaryFunction{\labelWest}(\timeVariable, \yLabel) - \krushkovParameter)\Bigl ( \flux_{\xLabel}(\traceOperator_{\labelWest}(\limitConservedMoment)(\timeVariable, \yLabel)) - \flux_{\xLabel}(\krushkovParameter)\Bigr )\testFunction(\timeVariable, 0, \yLabel)\differential\yLabel\differential\timeVariable\\
        +\int_0^{\finalTime}\reduceSpaceDoubleInt\int_{0}^{+\infty} \sign(\boundaryFunction{\labelSouth}(\timeVariable, \xLabel) - \krushkovParameter)\Bigl ( \flux_{\yLabel}(\traceOperator_{\labelSouth}(\limitConservedMoment)(\timeVariable, \xLabel)) - \flux_{\yLabel}(\krushkovParameter)\Bigr )\testFunction(\timeVariable, \xLabel, 0)\differential\xLabel\differential\timeVariable,
    \end{multline*}
    which concludes the proof.
\end{proof}

\section{Numerical experiments}\label{sec:numericalExperiments}

We now test boundary conditions based on equilibria both in the case of scalar problems, which has been addressed in \Cref{sec:convergenceScalar}, and for systems.
In the former case, we test numerical schemes also in regimes where they do not fulfill the monotonicity conditions, \idEst{} \eqref{eq:mononotonicityConditionsD2Q5}, \eqref{eq:mononotonicityConditionsD2Q4}, \eqref{eq:mononotonicityConditionsD1Q3}, and \eqref{eq:mononotonicityConditionsD1Q2}.
We only use BGK schemes, since it is outside the scope of this paper to discuss the advantages and the limitations of the many different relaxation operators. 
We thus consider $\relaxationParameterAntiSymmetric = \relaxationParameterSymmetric = \relaxationParameter$.

\subsection{Resilience against ``wrong'' traces: possible boundary layers}

In the introduction, we have discussed the fact that the PDE at hand might need a boundary condition only on part of the boundary.
However, the numerical scheme needs to specify incoming information on the whole boundary. 
We have discussed under what conditions the numerical scheme can be shown to converge to the entropy solution.
We now want to check the speed of convergence when clearly wrong data are used as arguments of the equilibria in \eqref{eq:boundaryEquilibrium1}, which we call resilience against ``wrong'' traces.

\subsubsection{Linear case: transport equation}\label{sec:transportExp}

Let us first consider a linear problem, which is useful to understand things from a theoretical perspective.
Consider a 1D transport equation with negative transport velocity set on a segment.
By linearity, the left boundary is---regardless of the solution---an outflow, whereas the right boundary is always an inflow.
Here, the trace of the solution must be enforced \emph{via} a Dirichlet boundary condition, that we decide to take equal to zero.
The problem reads, for $\advectionVelocity<0$: 
\begin{align*}
    &\partial_{\timeVariable}\conservedVariable(\timeVariable, \spaceVariable) + \advectionVelocity\partial_{\spaceVariable}\conservedVariable(\timeVariable, \spaceVariable) = 0, \qquad &&(\timeVariable, \spaceVariable)\in (0, \tfrac{1}{2})\times (0, 1), \\
    &\conservedVariable(0, \spaceVariable) = \conservedVariable^{\initial}(\spaceVariable), \qquad &&\spaceVariable\in(0, 1), \\
    &\conservedVariable(\timeVariable, 1) = \boundaryFunction{\labelEast}(\timeVariable) = 0, \qquad &&\timeVariable\in(0, \tfrac{1}{2}).
\end{align*} 
Still, something must be done at $\spaceVariable = 0$ from the numerical point of view.

We utilize a \lbmScheme{1}{2} scheme, where the numerical boundary condition is compatible with $\conservedVariable(\timeVariable, 1) = 0$ on the right, but a possibly wrong one on the left boundary.
We take the worst-case scenario where we systematically enforce a time-constant state $\tilde{\conservedVariableDiscrete}_{\labelWest}^{\indexTime} = \tilde{\conservedVariableDiscrete}_{\labelWest}\in[-\maximumInitialDatum, \maximumInitialDatum]$ on the left boundary, so that \eqref{eq:boundaryEquilibrium1} becomes
\begin{equation}\label{eq:wrongTraceCondition}
    \distributionFunction_{\labPosX, -1}^{\indexTime, \collided} = \distributionFunctionLetter_{\labPosX}^{\atEquilibrium}(\tilde{\conservedVariableDiscrete}_{\labelWest}) = \frac{\tilde{\conservedVariableDiscrete}_{\labelWest}}{2}(1+\courantNumber), \qquad  \distributionFunction_{\labNegX, \numberCells}^{\indexTime, \collided} = \distributionFunctionLetter_{\labNegX}^{\atEquilibrium}(0) = 0,
\end{equation}
where $\courantNumber\definitionEquality\advectionVelocity/\latticeVelocity$ is the Courant number.
We consider the case of a strict CFL condition: $-1<\courantNumber<0$.
Notice that better boundary conditions on the outflow to the left would be obtained by extrapolating the trace of the numerical solution outside the domain and take the corresponding equilibrium:
\begin{equation}\label{eq:extrapolations}
    \distributionFunction_{\labPosX, -1}^{\indexTime, \collided} = \distributionFunctionLetter_{\labPosX}^{\atEquilibrium}(\tilde{\conservedVariableDiscrete}_{\labelWest}^{\indexTime}) = \frac{\tilde{\conservedVariableDiscrete}_{\labelWest}^{\indexTime}}{2}(1+\courantNumber), \qquad\text{with}\qquad 
    \underbrace{\tilde{\conservedVariableDiscrete}_{\labelWest}^{\indexTime} = \conservedVariableDiscrete_0^{\indexTime}}_{\text{1st-order}}, \qquad \underbrace{\tilde{\conservedVariableDiscrete}_{\labelWest}^{\indexTime} = 2\conservedVariableDiscrete_0^{\indexTime} - \conservedVariableDiscrete_1^{\indexTime}}_{\text{2nd-order}}, \qquad \cdots
\end{equation}
However, starting from second-order extrapolations, there is no guarantee that $\distributionFunction_{\labPosX, -1}^{\indexTime, \collided}\in[\distributionFunctionLetter_{\labPosX}^{\atEquilibrium}(-\maximumInitialDatum),\distributionFunctionLetter_{\labPosX}^{\atEquilibrium}(\maximumInitialDatum)]$, which is essential in the proof that we have conducted in \Cref{sec:convergenceScalar}.

\paragraph{A numerical simulation with non-zero initial datum}

\begin{figure}
    \begin{center}
        \includegraphics[width=0.99\textwidth]{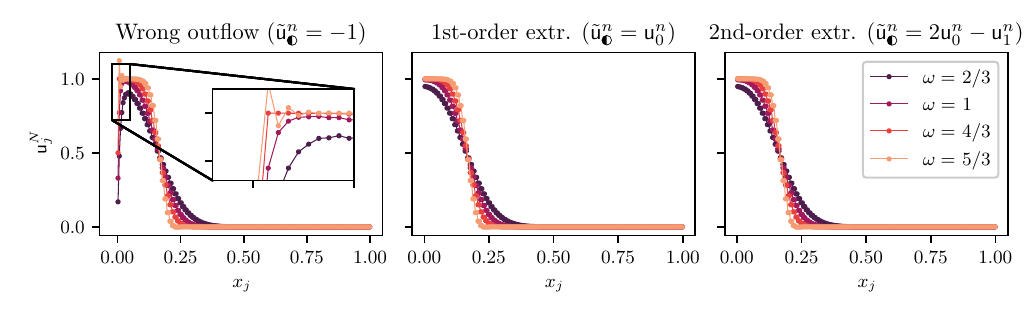}
    \end{center}\caption{\label{fig:D1Q2_Transport_Plot_Outflow}Solution at final time for the \lbmScheme{1}{2} scheme for the transport equation with different outflow ($\spaceVariable=0$) boundary conditions.}
\end{figure}

We propose a simulation with $\numberCells = 200$, using $\advectionVelocity = -1$ and $\latticeVelocity=2$.
The initial datum is $\conservedVariable^{\initial}(\spaceVariable) = \indicatorFunction{(1/3, 2/3)}(\spaceVariable)$ and the final time of the simulation is $\finalTime = \tfrac{1}{2}$. Within this framework, the \lbmScheme{1}{2} scheme is monotone for $\relaxationParameter\leq \tfrac{4}{3}$, \confer{} \eqref{eq:mononotonicityConditionsD1Q2}.
The results are presented in \Cref{fig:D1Q2_Transport_Plot_Outflow}.
We observe the expected boundary layer when \eqref{eq:wrongTraceCondition} is utilized.
We remark that this boundary layer looks geometrically damped for $\relaxationParameter<\tfrac{4}{3}$, and geometrically damped with oscillatory behavior when $\relaxationParameter>\tfrac{4}{3}$. When $\relaxationParameter = \tfrac{4}{3}$, the boundary layer is localized only on the first cell of the domain and does not extend inside.

This can be explained in the following way, following \cite{bellotti2024consistency}: away from the boundary, the conserved moment obeys the multi-step Finite Difference scheme $\conservedVariableDiscrete_{\indexSpace}^{\indexTime+1} = \tfrac{1}{2}(2-\relaxationParameter+\relaxationParameter\courantNumber)\conservedVariableDiscrete_{\indexSpace-1}^{\indexTime} + \tfrac{1}{2}(2-\relaxationParameter-\relaxationParameter\courantNumber)\conservedVariableDiscrete_{\indexSpace+1}^{\indexTime} + (\relaxationParameter-1) \conservedVariableDiscrete_{\indexSpace}^{\indexTime-1}$.
Inserting the ansatz  $\conservedVariableDiscrete_{\indexSpace}^{\indexTime}=\timeShiftOperator^{\indexTime}\fourierShift(\timeShiftOperator)^{\indexSpace}$, which boils down to consider a $\timeShiftOperator$-transform in time (see \cite{antoine:hal-00347884} for the details), yields
\begin{equation}\label{eq:charEquation}
    \tfrac{1}{2}(2-\relaxationParameter-\relaxationParameter\courantNumber) \fourierShift(\timeShiftOperator) + \tfrac{1}{2}(2-\relaxationParameter+\relaxationParameter\courantNumber) \fourierShift(\timeShiftOperator)^{-1}  = \timeShiftOperator + (1-\relaxationParameter)\timeShiftOperator^{-1}.
\end{equation}
This equation in $\fourierShift(\timeShiftOperator)$ has, when $-1<\courantNumber<0$ and except when $2-\relaxationParameter-\relaxationParameter\courantNumber= 0$ or $2-\relaxationParameter+\relaxationParameter\courantNumber= 0$, two non-trivial solutions $\solutionCharStable (\timeShiftOperator)$ (stable root) and $\solutionCharUnstable (\timeShiftOperator)$ (unstable root) such that $|\solutionCharStable (\timeShiftOperator)|<1$ and $|\solutionCharUnstable (\timeShiftOperator)|>1$ for $|\timeShiftOperator|>1$, which can be continuously extended up to $|\timeShiftOperator|=1$. 
When $2-\relaxationParameter+\relaxationParameter\courantNumber= 0$, the only non-trivial root is $\solutionCharUnstable (\timeShiftOperator)$ and $\solutionCharStable(\timeShiftOperator)\equiv 0$. 
The boundary layer originating from the left boundary at $\spaceVariable=0$ can only be propagated inside the domain by $\solutionCharStable$, which equals zero  when $2-\relaxationParameter+\relaxationParameter\courantNumber= 0$ ($\relaxationParameter=\tfrac{4}{3}$ for the considered setting).
Since the $\timeShiftOperator$-transform of a time-constant boundary source term has a pole at $\timeShiftOperator=1$, this is the value of $\timeShiftOperator$ at which one needs to study a residual to invert the $\timeShiftOperator$-transform, \confer{} \Cref{app:proofRelEstimate}:
\begin{equation}\label{eq:tmp3}
    \solutionCharStable(1) = \frac{2-\relaxationParameter+\relaxationParameter\courantNumber}{2-\relaxationParameter-\relaxationParameter\courantNumber},
\end{equation}
which is such that $\solutionCharStable(1)\in (0, 1)$ for $0<\relaxationParameter<\frac{2}{1-\courantNumber}$, $\solutionCharStable(1) = 0$ if $\relaxationParameter=\frac{2}{1-\courantNumber}$, and $\solutionCharStable(1)\in (-1, 0)$ if $\frac{2}{1-\courantNumber}<\relaxationParameter\leq 2$.
Since, see again \Cref{app:proofRelEstimate}, the boundary layer behaves $\propto\solutionCharStable(1)^{\indexSpace}$, this explains the behaviors of the numerical simulation observed on the first column of \Cref{fig:D1Q2_Transport_Plot_Outflow}. In particular, when $\relaxationParameter=\frac{2}{1-\courantNumber}$, we have $\solutionCharStable(1)^{\indexSpace} = \delta_{\indexSpace 0}$.

Finally, notice that extrapolations \eqref{eq:extrapolations} work fine and allow the initial datum to exit the domain from the left without boundary layer, as expected.

\paragraph{The structure of the boundary layer with zero initial datum}

\begin{figure}
    \begin{center}
        \includegraphics[width=0.99\textwidth]{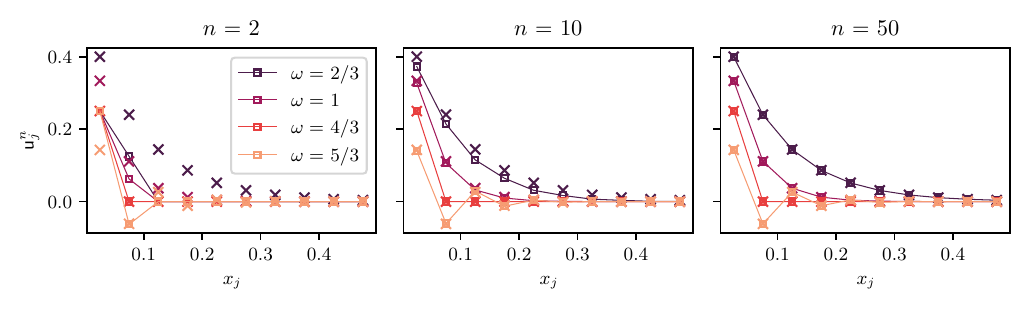}
    \end{center}\caption{\label{fig:D1Q2_Transport_Plot_BL_Description_Outflow}Solution at the iteration $\indexTime$ for the \lbmScheme{1}{2} scheme with boundary conditions \eqref{eq:wrongTraceCondition} with $\tilde{\conservedVariableDiscrete}_{\protect\labelWest} = 1$ and zero initial data. 
    Crosses correspond to the estimate \eqref{eq:estimationLongTimeRelaxationD1Q2WrongBC}.}
\end{figure}

The question now is: what kind of convergence rate can we hope for when a wrong trace on the outflow is imposed through \eqref{eq:wrongTraceCondition}? 
We know from \Cref{sec:convergenceScalar} that the scheme converges to the physical solution, which is insensitive of $\tilde{\conservedVariableDiscrete}_{\labelWest}$.
One must keep in mind that the convergence rate that one might expect for a problem without boundary, see \cite{brenner2006besov} for the theory on one-step linear Finite Difference schemes, in the $\lebesgueSpace{p}$ norm for a solution featuring shocks is $\bigO{\spaceStep^{1/(2p)}}$ when $\relaxationParameter\in(0, 2)$ (first-order scheme) and $\bigO{\spaceStep^{2/(3p)}}$ when $\relaxationParameter=2$ (second-order scheme).

For the sake of simplicity and thanks to the superposition principle, we consider the case of zero initial datum $\conservedVariable^{\initial}\equiv 0$.
In the case of a relaxation scheme, namely $\relaxationParameter = 1$, we can explicitly characterize the boundary layer profile.
When $\relaxationParameter\neq 1$, we can obtain the asymptotic profile for $\indexTime$ large enough.
This can help us figuring out the dependence of  $\lVert\conservedVariableDiscrete^{\indexTime}\rVert_{\petitLebesgueSpace{p}, \spaceStep}$ in $\spaceStep$, that is, the order of the scheme in presence of boundary layers induced by wrong outflow conditions.
\begin{proposition}[Description of the boundary layer for the linear \lbmScheme{1}{2}]\label{prop:EstimateWrongBoundaryRelaxationD1Q2}
    Let $\conservedVariableDiscrete_{\indexSpace}^{\indexTime}$ for $\indexTime\in\integerInterval{0}{\numberTimeSteps}$ and $\indexSpace\in\integerInterval{0}{\numberCells-1}$, be the solution of the \lbmScheme{1}{2} scheme with $\relaxationParameter=1$ (relaxation scheme) with $\conservedVariableDiscrete_{\indexSpace}^0 = 0$ for $\indexSpace\in\integerInterval{0}{\numberCells-1}$ and boundary conditions \eqref{eq:wrongTraceCondition}.
    Assume the strict CFL condition $-1<\courantNumber<0$.
    Then, for $\indexTime\geq 1$
    \begin{equation}\label{eq:estimateD1Q2Chebyshev}
        \conservedVariableDiscrete_{\indexSpace}^{\indexTime} = \frac{\tilde{\conservedVariableDiscrete}_{\labelWest}}{2}(1+\courantNumber) \Biggl (\delta_{\indexSpace 0} + \frac{\bigl ( \frac{1+\courantNumber}{1-\courantNumber}\bigr )^{\indexSpace/2}}{2\numberCells+2}\sum_{h=1}^{\lfloor\numberCells/2\rfloor}(4-\lambda_{\numberCells-h+1}^2) U_{\indexSpace}(\tfrac{1}{2}\lambda_{\numberCells-h+1}) \sum_{p=1}^{\indexTime-1} (1+(-1)^{\indexSpace+p})(\tfrac{1}{2}\sqrt{1-\courantNumber^2}\lambda_{\numberCells-h+1})^p \Biggr ),
    \end{equation}
    where $U_{\indexSpace}$ is the $\indexSpace$-th Chebyshev polynomial of second kind and $\lambda_h \definitionEquality -2\cos ( \frac{h\pi}{\numberCells+1} )$.

    Under the same conditions, except for the fact of allowing $\relaxationParameter\in (0, 2)$, when the number of cells $\numberCells$ is large ($\numberCells\gg 1$) as well as the discrete time $\indexTime$ ($\indexTime\gg1$), we have 
    \begin{equation}\label{eq:estimationLongTimeRelaxationD1Q2WrongBC}
        \conservedVariableDiscrete_{\indexSpace}^{\indexTime}\approx 
        \begin{cases}
            \tilde{\conservedVariableDiscrete}_{\labelWest} \frac{(2-\relaxationParameter)(1+\courantNumber)}{2-\relaxationParameter(1+\courantNumber)}\bigl ( \frac{2-\relaxationParameter+\relaxationParameter\courantNumber}{2-\relaxationParameter-\relaxationParameter\courantNumber}\bigr )^{\indexSpace}, \qquad &\text{if}\quad \relaxationParameter\neq\frac{2}{1-\courantNumber}, \\
            \tilde{\conservedVariableDiscrete}_{\labelWest}\frac{1+\courantNumber}{2}\delta_{\indexSpace 0}, \qquad &\text{if}\quad \relaxationParameter=\frac{2}{1-\courantNumber}.
        \end{cases}
    \end{equation}
\end{proposition}
The construction of \Cref{prop:EstimateWrongBoundaryRelaxationD1Q2} (proof given in \Cref{app:proofRelEstimate}) is completely analogous to the discussion from \cite{boutin2017stability} concerning Dirichlet boundary conditions for Finite Difference schemes and the boundary layer that these can generate, except for the fact that we limit the study to the long-time behavior. The work of \cite{boutin2017stability} deals with short times by considering the trace of the exact solution on the boundary.

To check that \eqref{eq:estimationLongTimeRelaxationD1Q2WrongBC} gives a good estimate on the behavior of the boundary layer as $\indexTime$ is large enough, consider a mesh made up of $\numberCells = 20$ cells and a zero initial datum, enforcing $\tilde{\conservedVariableDiscrete}_{\labelWest}=1$. 
The results in \Cref{fig:D1Q2_Transport_Plot_BL_Description_Outflow} confirm that  \eqref{eq:estimationLongTimeRelaxationD1Q2WrongBC} gives a good description of the boundary layer, for $\indexTime$ large enough.

\Cref{prop:EstimateWrongBoundaryRelaxationD1Q2} shows that we cannot expect convergence in the $\lebesgueSpace{\infty}$ norm when wrong traces are enforced.
This is different from the scheme (which is indeed a Lax-Friedrichs scheme when $\relaxationParameter = 1$) on the whole real line $\reals$, which is indeed stable and convergent in the $\lebesgueSpace{\infty}$ norm, see \cite{brenner2006besov}.
Under the assumptions by \Cref{prop:EstimateWrongBoundaryRelaxationD1Q2} and using \eqref{eq:estimationLongTimeRelaxationD1Q2WrongBC}, we deduce that on sufficiently refined grids and for $\indexTime$ large enough:
\begin{equation*}
    \lVert\conservedVariableDiscrete^{\indexTime}\rVert_{\petitLebesgueSpace{p}, \spaceStep} \approx  C(p) |\tilde{\conservedVariableDiscrete}_{\labelWest}| \spaceStep^{1/p}.
\end{equation*}
Assuming a first-order scheme, which is the case when $\relaxationParameter < 2$, we see that the error by the boundary layer in the $\lebesgueSpace{p}$-norm is $\bigO{\spaceStep^{1/p}}$, whereas the error of the inner scheme is $\bigO{\spaceStep^{1/(2p)}}$ if the solution features a shock, thus the latter always dominates. When data are smooth, we have to compare the error due to the boundary layer of order $\bigO{\spaceStep^{1/p}}$ with the error of the inner scheme, being $\bigO{\spaceStep}$. They are of the same order as $p=1$, but the boundary layer domaines when $p>1$.

\subsubsection{Non-linear case: Burgers equation}

\begin{figure}[h]
    \begin{center}
        Final time $\finalTime = \tfrac{1}{5}$\\
        \includegraphics[width=0.99\textwidth]{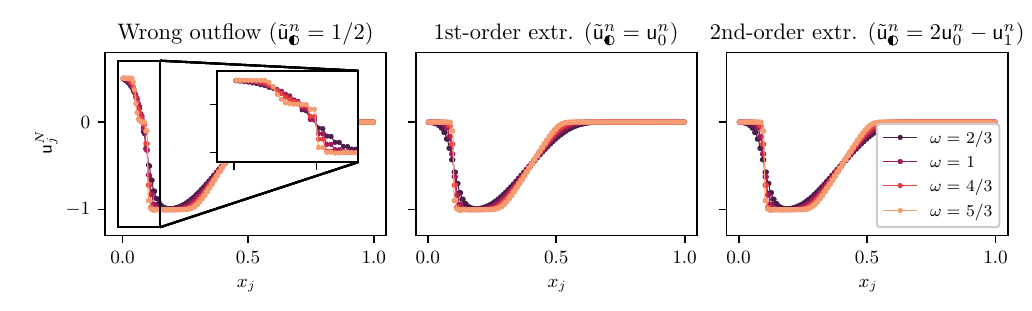}\\
        Final time $\finalTime = \tfrac{1}{2}$\\
        \includegraphics[width=0.99\textwidth]{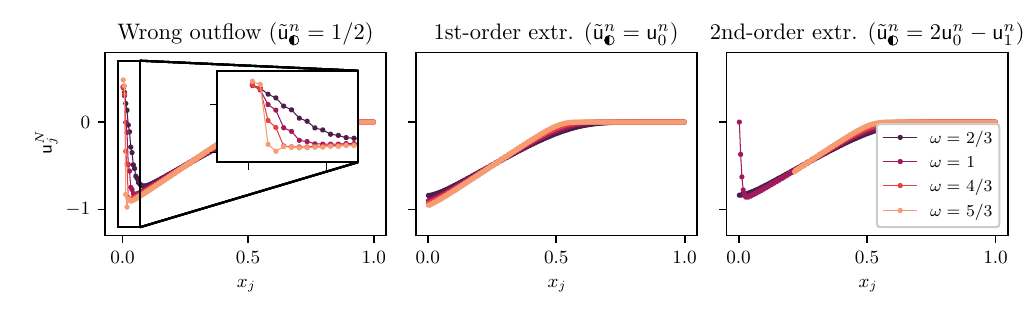}
    \end{center}\caption{\label{fig:D1Q2_Burgers_Plot_Outflow}Solution at final time for the \lbmScheme{1}{2} scheme for the Burgers equation with different outflow ($\spaceVariable=0$) boundary conditions.}
\end{figure}

We consider the same setting as \Cref{sec:transportExp}, except for the fact that the considered equation is a Burgers equation $\partial_{\timeVariable}\conservedVariable + \partial_{\spaceVariable}(\tfrac{1}{2}\conservedVariable^2)=0$.
We propose a simulation with $\numberCells = 200$, using $\latticeVelocity=2$.
The initial datum is $\conservedVariable^{\initial}(\spaceVariable) = -\indicatorFunction{(1/5, 1/2)}(\spaceVariable)$ and the final time of the simulation is either $\finalTime = \tfrac{1}{5}$ or $\finalTime=\tfrac{1}{2}$. Within this framework, the \lbmScheme{1}{2} scheme is monotone for $\relaxationParameter\leq \tfrac{4}{3}$.

The results in \Cref{fig:D1Q2_Burgers_Plot_Outflow} show a boundary layer at the outflow (at $\spaceVariable=0$) due to the wrong trace enforced at this boundary.
As long as the shock has not passed through the outflow, first and second-order extrapolations injected in the equilibrium work fine.
However, when the shock has left the domain, we see that the second-order extrapolation creates a boundary layer when $\relaxationParameter = 1$, and the simulation is unstable (data are missing) for $\relaxationParameter = \tfrac{4}{3}$ and $\tfrac{5}{3}$.
Indeed, as we have emphasized, there is no guarantee that $2\conservedVariableDiscrete_0^{\indexTime}-\conservedVariableDiscrete_1^{\indexTime}$ remains in $[-\maximumInitialDatum, \maximumInitialDatum]$ even when $\relaxationParameter \leq \tfrac{4}{3}$, which causes instabilities.

\subsection{A scalar 1D non-convex problem}

We now consider a test case based on the problem
\begin{align*}
    &\partial_{\timeVariable}\conservedVariable(\timeVariable, \spaceVariable) + \partial_{\spaceVariable} (\tfrac{1}{3}\conservedVariable(\timeVariable, \spaceVariable)^3 ) = 0, \qquad &&(\timeVariable, \spaceVariable) \in(0, 4)\times(0, 1), \\
    &\conservedVariable(0, \spaceVariable) = \conservedVariable^{\initial}(\spaceVariable) = 0, \qquad &&\spaceVariable\in(0, 1), \\
    &\conservedVariable(\timeVariable, 0) = \boundaryFunction{\labelWest}(\timeVariable) =  \sin(6\timeVariable), \qquad &&\timeVariable\in(0, 4), \\
    &\conservedVariable(\timeVariable, 1) = \boundaryFunction{\labelEast}(\timeVariable) =  0, \qquad &&\timeVariable\in(0, 4).
\end{align*}
We employ a \lbmScheme{1}{2} scheme with $\latticeVelocity = \tfrac{10}{7}$ as in \cite[Section 6.1]{aregba2004kinetic}, except that we fix it for the whole simulation without having the possibility to adapt it dynamically.
With this data, the scheme remains monotone for  $\relaxationParameter\leq \tfrac{20}{17}$.
Notice that whenever $\relaxationParameter = 1$, we obtain the two-velocities model---called ``XW''---from \cite{aregba2004kinetic}.
We compare to a Godunov scheme that has to be though as a benchmark reference scheme.

\begin{figure}[h]
    \begin{center}
        \includegraphics[width=0.99\textwidth]{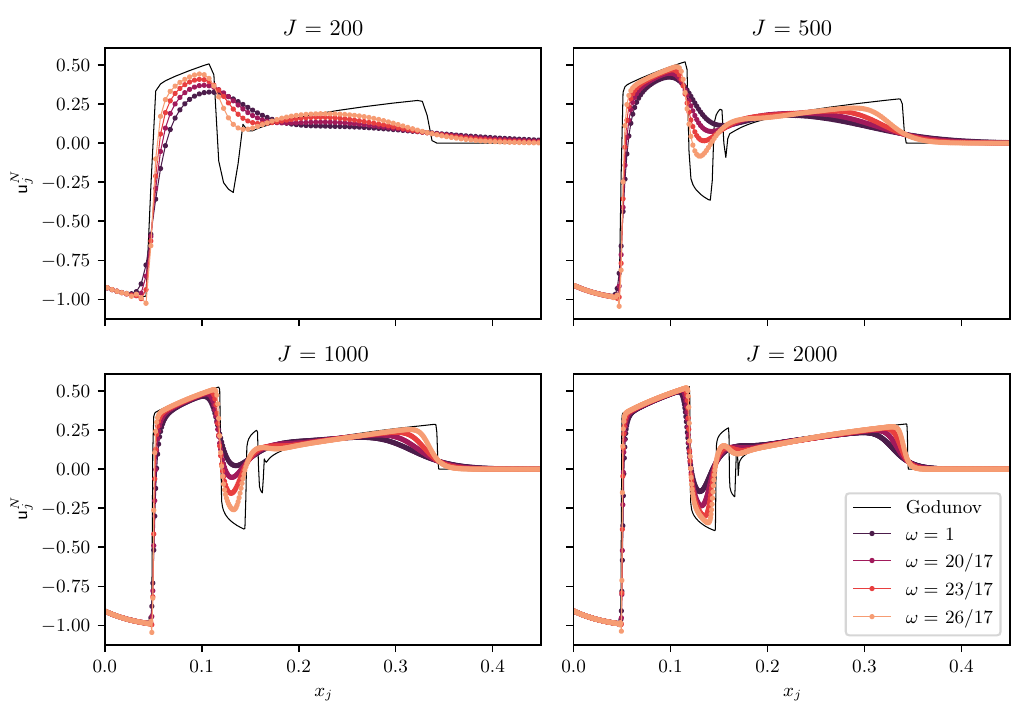}
    \end{center}\caption{\label{fig:D1Q2_XW_Plot}Solution at final time for the \lbmScheme{1}{2} scheme for the non-convex problem on part of the domain.}
\end{figure}

The results are given in \Cref{fig:D1Q2_XW_Plot} and show convergence to the same solution as the Godunov scheme.
The schemes we consider are more diffusive, and we see that by increasing $\relaxationParameter$, we can make the scheme more accurate compared to the relaxation scheme.
Moreover, we start observing oscillations when $\relaxationParameter> \tfrac{20}{17}$, and the scheme ceases to be monotone.

\subsection{Scalar 2D Burgers equation}

We now address a two-dimensional problem, which reads:
\begin{align*}
    &\partial_{\timeVariable}{\conservedVariable}(\timeVariable, \vectorial{\spaceVariable}) + \partial_{\xLabel}(\tfrac{1}{2}({\conservedVariable}(\timeVariable, \vectorial{\spaceVariable}))^2) + \partial_{\yLabel}(\tfrac{1}{2}({\conservedVariable}(\timeVariable, \vectorial{\spaceVariable}))^2) = 0, \qquad &&(\timeVariable, \vectorial{\spaceVariable})\in (0, \tfrac{1}{2})\times (0, 1)^2,\\
    &{\conservedVariable}(0, \vectorial{\spaceVariable}) = {\conservedVariable}^{\initial}(\vectorial{\spaceVariable}) = 0, \qquad &&\vectorial{\spaceVariable}\in(0, 1)^2, \\
    &{\conservedVariable}(\timeVariable, \xLabel = 0, \yLabel) = \tilde{{\conservedVariable}}_{\labelWest}(\timeVariable, \yLabel) = \indicatorFunction{\cos(\vartheta)(-\timeVariable/2)+\sin(\vartheta)(\yLabel-\timeVariable/2)\leq 0}, \qquad &&(\yLabel, \timeVariable)\in(0, 1)\times(0, \tfrac{1}{2}),\\
    &{\conservedVariable}(\timeVariable, \xLabel = 1, \yLabel) = \tilde{{\conservedVariable}}_{\labelEast}(\timeVariable, \yLabel)= \indicatorFunction{\cos(\vartheta)(1-\timeVariable/2)+\sin(\vartheta)(\yLabel-\timeVariable/2)\leq 0}, \qquad &&(\yLabel, \timeVariable)\in(0, 1)\times(0, \tfrac{1}{2}),\\
    &{\conservedVariable}(\timeVariable, \xLabel, \yLabel = 0) = \tilde{{\conservedVariable}}_{\labelSouth}(\timeVariable, \xLabel) =\indicatorFunction{\cos(\vartheta)(\xLabel-\timeVariable/2)+\sin(\vartheta)(-\timeVariable/2)\leq 0}, \qquad &&(\xLabel, \timeVariable)\in(0, 1)\times(0, \tfrac{1}{2}),\\
    &{\conservedVariable}(\timeVariable, \xLabel, \yLabel = 1) = \tilde{{\conservedVariable}}_{\labelNorth}(\timeVariable, \xLabel) = \indicatorFunction{\cos(\vartheta)(\xLabel-\timeVariable/2)+\sin(\vartheta)(1-\timeVariable/2)\leq 0}, \qquad &&(\xLabel, \timeVariable)\in(0, 1)\times(0, \tfrac{1}{2}),
\end{align*}
with $\vartheta\in[0, \pi/2]$.
The exact solution is a shock profile connecting constant states---equal to one before and equal to zero after the shock---moving in the direction $(\cos(\vartheta),\sin(\vartheta))$ according to the Rankine-Hugoniot relation.
This solution is ${\conservedVariable}(\timeVariable, \vectorial{\spaceVariable}) = \indicatorFunction{\cos(\vartheta)(\xLabel-\timeVariable/2)+\sin(\vartheta)(\yLabel-\timeVariable/2)\leq 0}$.
Since the initial datum is zero, this solution it totally built-up using boundary condition, which shows the relevance of this setting to evaluate our way of enforcing boundary conditions.
In the numerical simulations, we take $\vartheta = \tfrac{\pi}{3}$.

\begin{figure}[h]
    \begin{center}
        \includegraphics[width=0.99\textwidth]{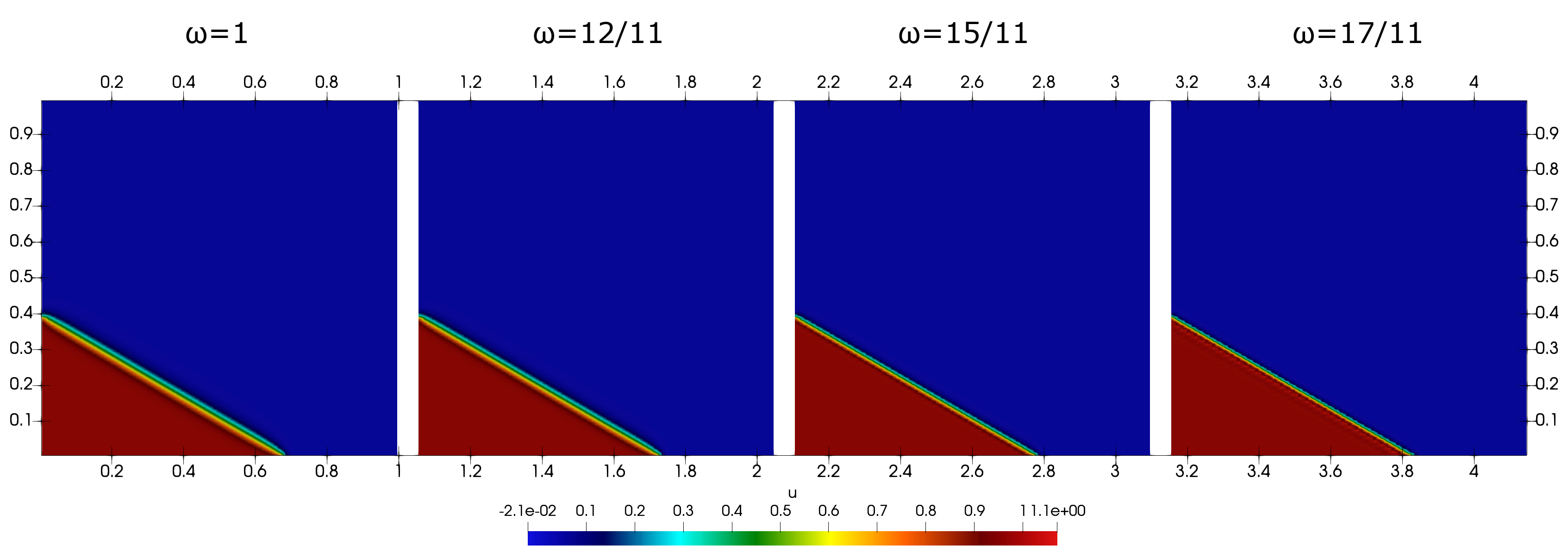}
    \end{center}\caption{\label{fig:D2Q4Burgers}Solution at final time $\conservedVariableDiscrete_{\vectorial{\indexSpace}}^{\numberTimeSteps}$ for the \lbmScheme{2}{4} for the 2D Burgers equation, for several values of $\relaxationParameter$.}
\end{figure}

We employ a \lbmScheme{2}{4} scheme with $\equilibriumCoefficientLinear_{\xLabel} = \equilibriumCoefficientLinear_{\yLabel} = \tfrac{1}{4}$.
Moreover, we take $\latticeVelocity = 3$, so that the scheme is monotone for its BGK version under $\relaxationParameter \leq \tfrac{12}{11}$.
Finally, we consider $\numberCells = 100$, see \eqref{eq:mononotonicityConditionsD2Q4}.
The solutions in \Cref{fig:D2Q4Burgers} show that the shock propagates at the right speed without deformation, thus the boundary conditions are able to reproduce the expected dynamics.
As expected, the shock undergoes less numerical diffusion for larger $\relaxationParameter$. However, this comes at the price of dispersive oscillations behind it, which are particualarly manifest when $\relaxationParameter = \tfrac{17}{11}$.
Moreover, since the shock is (indeed, on purpose) not parallel to the axis, we observe a small ``pinching'' at the level of the boundary, since the discrete velocities of the scheme are parallel to the axis.

\subsection{Euler equations in 2D: Double Mach 10 reflection}\label{sec:Mach10}

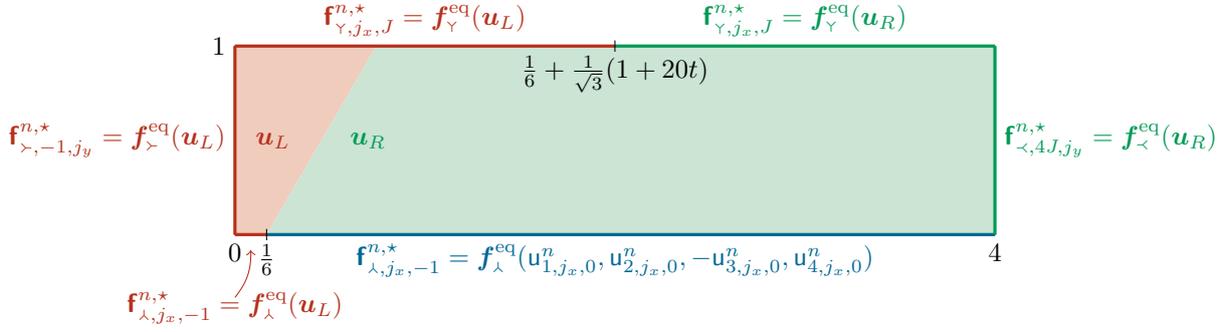
\begin{figure}[h]
    \begin{center}
        \begin{tikzpicture}[scale=2.5]
    \draw (0,0) rectangle (4,1);

    \def\xa{1/6} 
    \def\xb{\xa + 0.57735026919} 
     \fill[BrickRed!20] (0,0) -- (0,1) -- (\xb,1) -- (\xa,0) -- cycle;
     \fill[ForestGreen!20] (\xb,1) -- (\xa,0) -- (4, 0) -- (4, 1) -- cycle;

    \draw[BrickRed, very thick] (0,0) -- (0,1);
    
    \draw[BrickRed, very thick] (0,0) -- (1/6,0);
    
    \draw[MidnightBlue, very thick] (1/6,0) -- (4,0);
    
    \draw[BrickRed, very thick] (0,1) -- (2,1);
    
    \draw[ForestGreen, very thick] (2,1) -- (4,1);
    
    \draw[ForestGreen, very thick] (4,0) -- (4,1);
    
    \node[below] at (0,0) {0}; 
    \node[below] at (4,0) {4}; 
    \node[left] at (0,1) {1}; 
    \draw (1/6,-0.03) -- (1/6,0.03);
    \node[below] at (1/6,0) {\(\frac{1}{6}\)}; 
    \draw (2,1-0.03) -- (2,1+0.03);
    \node[below] at (2,1) {$\tfrac{1}{6}+\tfrac{1}{\sqrt{3}}(1+20\timeVariable)$}; 

    \node[color=BrickRed] at (0.2,0.5) {$\vectorial{\conservedVariable}_L$}; 
    \node[color=ForestGreen] at (0.7,0.5) {$\vectorial{\conservedVariable}_R$}; 

    \node[left, color=BrickRed] at (0,1/2) {$\vectorial{\distributionFunction}_{\labPosX, -1, \indexSpace_{\yLabel}}^{\indexTime, \collided} = \vectorial{\distributionFunctionLetter}_{\labPosX}^{\atEquilibrium}(\vectorial{\conservedVariable}_L)$};
    \node[above, color=BrickRed] at (1,1) {$\vectorial{\distributionFunction}_{\labNegY, \indexSpace_{\xLabel}, \numberCells}^{\indexTime, \collided} = \vectorial{\distributionFunctionLetter}_{\labNegY}^{\atEquilibrium}(\vectorial{\conservedVariable}_L)$};
    \node[below, color=BrickRed] at (0,-1/4) {$\vectorial{\distributionFunction}_{\labPosY, \indexSpace_{\xLabel}, -1}^{\indexTime, \collided} = \vectorial{\distributionFunctionLetter}_{\labPosY}^{\atEquilibrium}(\vectorial{\conservedVariable}_L)$};
    \draw[->, color=BrickRed] (0,-1/3) to[bend right=20] (1/12, -1/12);
    \node[above, color=ForestGreen] at (3,1) {$\vectorial{\distributionFunction}_{\labNegY, \indexSpace_{\xLabel}, \numberCells}^{\indexTime, \collided} = \vectorial{\distributionFunctionLetter}_{\labNegY}^{\atEquilibrium}(\vectorial{\conservedVariable}_R)$};
    \node[right, color=ForestGreen] at (4,0.5) {$\vectorial{\distributionFunction}_{\labNegX, 4\numberCells, \indexSpace_{\yLabel}}^{\indexTime, \collided} = \vectorial{\distributionFunctionLetter}_{\labNegX}^{\atEquilibrium}(\vectorial{\conservedVariable}_R)$};
    \node[below, color=MidnightBlue] at (2,0) {$\vectorial{\distributionFunction}_{\labPosY, \indexSpace_{\xLabel}, -1}^{\indexTime, \collided} = \vectorial{\distributionFunctionLetter}_{\labPosY}^{\atEquilibrium}(\conservedVariableDiscrete_{1, \indexSpace_{\xLabel}, 0}^{\indexTime}, \conservedVariableDiscrete_{2, \indexSpace_{\xLabel}, 0}^{\indexTime}, -\conservedVariableDiscrete_{3, \indexSpace_{\xLabel}, 0}^{\indexTime}, \conservedVariableDiscrete_{4, \indexSpace_{\xLabel}, 0}^{\indexTime})$};
    \end{tikzpicture}
    \end{center}\caption{\label{fig:BCEUler}Illustration of the boundary conditions enforced for the Mach 10 problem. Initial data, either $\vectorial{\conservedVariable}_L$ or $\vectorial{\conservedVariable}_R$ are presented in desaturated colors.}
\end{figure}

\begin{figure}[h]
    \begin{center}
        \includegraphics[width=0.99\textwidth]{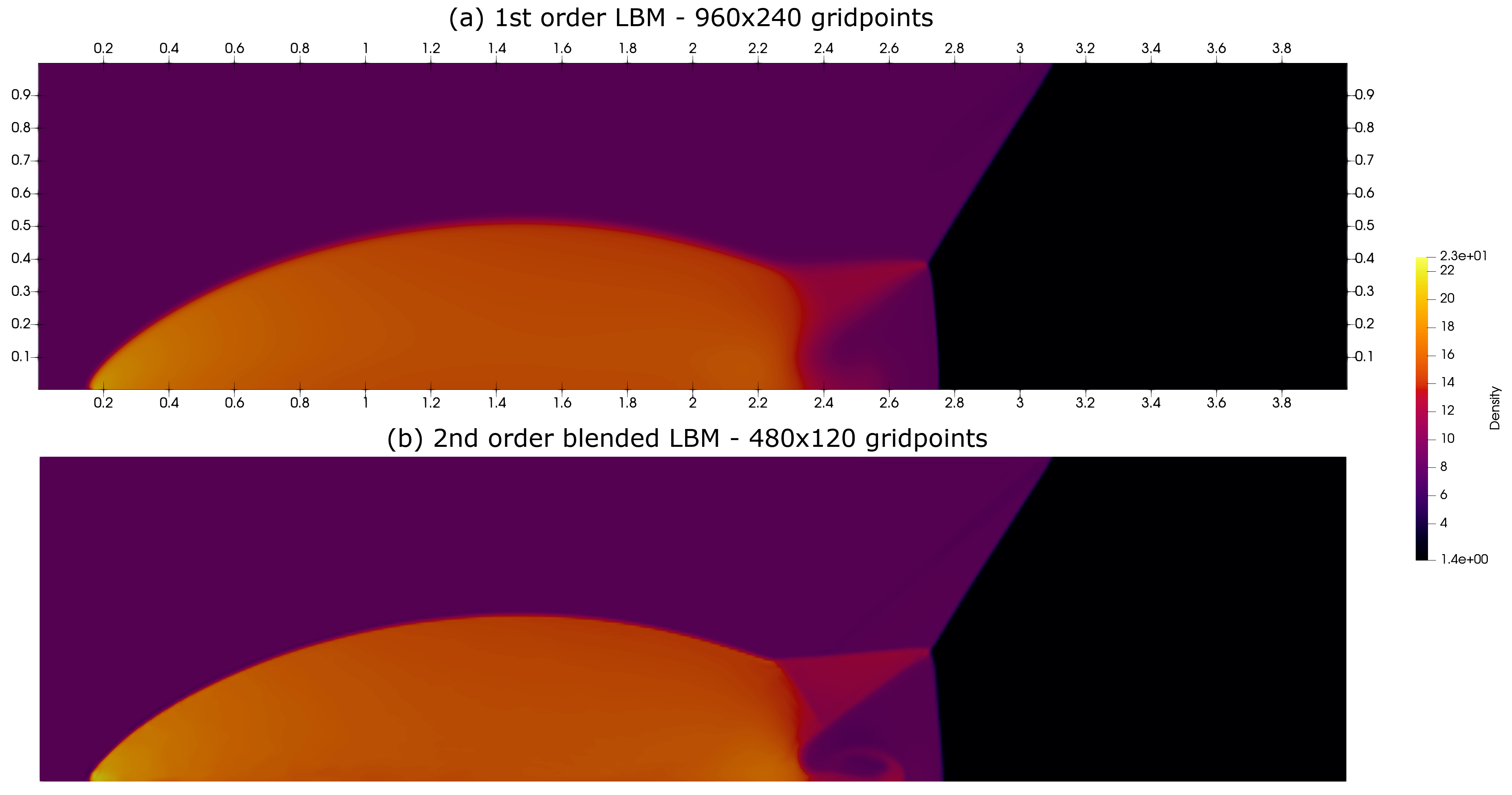}
    \end{center}\caption{\label{fig:Mach10}Density field at final time $\finalTime = \tfrac{1}{5}$ for the \lbmScheme{2}{4} scheme (a) and for the blended \lbmScheme{2}{5} scheme (b).}
\end{figure}

We finish by considering the test case first introduced by \cite{woodward1984numerical}: an oblique Mach 10 shock, forming a $60$-degree angle with the horizonal axis, and pre and post-shock states given by 
\begin{equation*}
    \vectorial{\conservedVariable}_L = (\rho_L, (\rho u)_L, (\rho v)_L, E_L) = (8, 57.16, -33, 563.52), \qquad \vectorial{\conservedVariable}_R  = (\rho_R, (\rho u)_R, (\rho v)_R, E_R) = (1.4, 0, 0, 2.5).
\end{equation*}
The flow variables $\vectorial{\conservedVariable} = (\rho, \rho u, \rho v, E)$ evolve under the Euler equations 
\begin{align*}
    \partial_{\timeVariable}\rho + \partial_{\xLabel}(\rho u) + \partial_{\yLabel}(\rho v)= 0, \qquad (\timeVariable, \vectorial{\spaceVariable})\in(0, \tfrac{1}{5})\times (0, 4)\times (0, 1),  \\
    \partial_{\timeVariable}(\rho u ) + \partial_{\xLabel}(\rho u^2 + p) + \partial_{\yLabel}(\rho uv)= 0, \qquad (\timeVariable, \vectorial{\spaceVariable})\in(0, \tfrac{1}{5})\times (0, 4)\times (0, 1),\\
    \partial_{\timeVariable}(\rho v ) + \partial_{\xLabel}(\rho u v) + \partial_{\yLabel}(\rho v^2 + p)= 0,\qquad (\timeVariable, \vectorial{\spaceVariable})\in(0, \tfrac{1}{5})\times (0, 4)\times (0, 1),\\
    \partial_{\timeVariable}E + \partial_{\xLabel}((E+p)u) + \partial_{\yLabel}((E+p)v)= 0,\qquad (\timeVariable, \vectorial{\spaceVariable})\in(0, \tfrac{1}{5})\times (0, 4)\times (0, 1),
\end{align*}
with pressure law $p = (\gamma - 1)(E - \tfrac{1}{2}\rho(u^2 + v^2))$.
We consider $\gamma = 1.4$ and use the initial datum
\begin{equation*}
    \vectorial{\conservedVariable}^{\initial}(\vectorial{\spaceVariable}) = \vectorial{\conservedVariable}_L\indicatorFunction{\yLabel-\sqrt{3}(\xLabel-\tfrac{1}{6})\geq 0} + \vectorial{\conservedVariable}_R\indicatorFunction{\yLabel-\sqrt{3}(\xLabel-\tfrac{1}{6})< 0}.
\end{equation*}
Boundary conditions are as follows, see \Cref{fig:BCEUler}.
On the western and eastern boundaries (labelled $\labelWest$ and $\labelEast$), we respectively enforce the trace $\vectorial{\conservedVariable}_L$ and $\vectorial{\conservedVariable}_R$, through the corresponding equilibria of $\vectorial{\distributionFunction}_{\labPosX}$ and $\vectorial{\distributionFunction}_{\labNegX}$.
On the northern boundary $\labelNorth$, we enforce the trace of the undisturbed shock in an infinite domain using the equilibrium of $\vectorial{\distributionFunction}_{\labNegY}$, that is, we impose $\vectorial{\conservedVariable}_L$ for $0\leq \xLabel\leq \tfrac{1}{6} + \tfrac{1}{\sqrt{3}}(1+20\timeVariable)$ and $\vectorial{\conservedVariable}_R$ for $\tfrac{1}{6} + \tfrac{1}{\sqrt{3}}(1+20\timeVariable)<\xLabel \leq 4$.
For the southern boundary $\labelSouth$, we enforce, using the equilibrium of $\vectorial{\distributionFunction}_{\labPosY}$, the state $\vectorial{\conservedVariable}_L$ for $0\leq \xLabel\leq \tfrac{1}{6}$.
On the other hand, for $\tfrac{1}{6}<\xLabel\leq 4$, we want this part of sourthern boundary be a reflective wall. This is obtained by using the equilibrium on the state where the vertical velocity is reversed. This state is computed using a first-order extrapolation in the direction of the normal vector to the boundary. Thus, on this part of the wall, we enforce 
\begin{equation*}
    \vectorial{\distributionFunction}_{\labPosY, \indexSpace_{\xLabel}, -1}^{\indexTime, \collided} = \vectorial{\distributionFunctionLetter}_{\labPosY}^{\atEquilibrium}(\conservedVariableDiscrete_{1, \indexSpace_{\xLabel}, 0}^{\indexTime}, \conservedVariableDiscrete_{2, \indexSpace_{\xLabel}, 0}^{\indexTime}, -\conservedVariableDiscrete_{3, \indexSpace_{\xLabel}, 0}^{\indexTime}, \conservedVariableDiscrete_{4, \indexSpace_{\xLabel}, 0}^{\indexTime}).
\end{equation*}

We compare two kinds of numerical scheme.
\begin{itemize}
    \item[(a)] A \lbmScheme{2}{4} with equilibria given by 
    \begin{align*}
        \vectorial{\distributionFunctionLetter}_{\labPosX}^{\atEquilibrium}(\vectorial{\conservedVariable}) &= \tfrac{1}{4} \vectorial{\conservedVariable} + \tfrac{1}{2\latticeVelocity}\vectorial{\flux}_{\xLabel}(\vectorial{\conservedVariable}), \qquad \vectorial{\distributionFunctionLetter}_{\labNegX}^{\atEquilibrium}(\vectorial{\conservedVariable}) = \tfrac{1}{4} \vectorial{\conservedVariable} - \tfrac{1}{2\latticeVelocity}\vectorial{\flux}_{\xLabel}(\vectorial{\conservedVariable}), \\
        \vectorial{\distributionFunctionLetter}_{\labPosY}^{\atEquilibrium}(\vectorial{\conservedVariable}) &= \tfrac{1}{4} \vectorial{\conservedVariable} + \tfrac{1}{2\latticeVelocity}\vectorial{\flux}_{\yLabel}(\vectorial{\conservedVariable}), \qquad \vectorial{\distributionFunctionLetter}_{\labNegY}^{\atEquilibrium}(\vectorial{\conservedVariable}) = \tfrac{1}{4} \vectorial{\conservedVariable} - \tfrac{1}{2\latticeVelocity}\vectorial{\flux}_{\yLabel}(\vectorial{\conservedVariable}),
    \end{align*}
    fulfilling \eqref{eq:consistency}.
    We employ a BGK collision operator with $\relaxationParameter = 1.35$, and $\latticeVelocity = 30$ due to the strength of the shock.
    \item[(b)] A \lbmScheme{2}{5} developed by \cite{wissocq2024positive} based on the equilibria
    \begin{align*}
        \vectorial{\distributionFunctionLetter}_{\labZeroVel}^{\atEquilibrium}(\vectorial{\conservedVariable}) = \tfrac{1}{2}\vectorial{\conservedVariable}, \qquad \vectorial{\distributionFunctionLetter}_{\labPosX}^{\atEquilibrium}(\vectorial{\conservedVariable}) &= \tfrac{1}{8} \vectorial{\conservedVariable} + \tfrac{1}{2\latticeVelocity}\vectorial{\flux}_{\xLabel}(\vectorial{\conservedVariable}), \qquad \vectorial{\distributionFunctionLetter}_{\labNegX}^{\atEquilibrium}(\vectorial{\conservedVariable}) = \tfrac{1}{8} \vectorial{\conservedVariable} - \tfrac{1}{2\latticeVelocity}\vectorial{\flux}_{\xLabel}(\vectorial{\conservedVariable}), \\
        \vectorial{\distributionFunctionLetter}_{\labPosY}^{\atEquilibrium}(\vectorial{\conservedVariable}) &= \tfrac{1}{8} \vectorial{\conservedVariable} + \tfrac{1}{2\latticeVelocity}\vectorial{\flux}_{\yLabel}(\vectorial{\conservedVariable}), \qquad \vectorial{\distributionFunctionLetter}_{\labNegY}^{\atEquilibrium}(\vectorial{\conservedVariable}) = \tfrac{1}{8} \vectorial{\conservedVariable} - \tfrac{1}{2\latticeVelocity}\vectorial{\flux}_{\yLabel}(\vectorial{\conservedVariable}),
    \end{align*}
    Without entering into the details, this scheme uses a BGK collision operator with $\relaxationParameter \in[1, 2]$ (always trying to keep it as close as possible to two to achieve second-order accuracy) which is locally adjusted (blended) to enforce admissibility of the numerical solution, \idEst{} positive density and positive pressure, and to minimize oscillations on shocks.
    This schemes dynamically adapts $\latticeVelocity$ in time.
\end{itemize}

In the results of \Cref{fig:Mach10}, we see that both for scheme (a) and (b), boundary conditions are correctly enforced.
Of course, scheme (a) is way more diffusive than scheme (b) even if utilized with twice the gridpoints per direction.
We see that our boundary conditions based on the equilibria can be successfully used with state-of-the-art schemes, such as (b).

\section*{Acknowledgments}

The authors are grateful to G. Wissocq, Y. Liu, and R. Abgrall for the enlighting discussion on their work and for providing a ready-to-use code to run the example of \Cref{sec:Mach10} with their numerical scheme \cite{wissocq2024positive}.

\bibliographystyle{apalike}
\bibliography{biblio}

\appendix

\section{Proofs of results in \Cref{sec:usefulBV}}\label{app:proofsBV}

\begin{proof}[Proof of \Cref{prop:totalVariationPiecewiseApprox}]
    Through changes of variable and basic properties of the Lebesgue integral, we obtain:
    \begin{align*}
        \totalVariation{\conservedVariable_{\discreteMark}}{(\realsPositive)^2} &= \spaceStep \sum_{\vectorial{\indexSpace}\in\naturals^2} \Bigl (\Bigl | \dashint_{\cell{\vectorial{\indexSpace}}}\conservedVariable(\xLabel, \yLabel)\differential\xLabel\differential\yLabel - \dashint_{\cell{\vectorial{\indexSpace}+\canonicalBasisVector{\xLabel}}}\conservedVariable(\xLabel, \yLabel)\differential\xLabel\differential\yLabel\Bigr|+\Bigl | \dashint_{\cell{\vectorial{\indexSpace}}}\conservedVariable(\xLabel, \yLabel)\differential\xLabel\differential\yLabel - \dashint_{\cell{\vectorial{\indexSpace}+\canonicalBasisVector{\yLabel}}}\conservedVariable(\xLabel, \yLabel)\differential\xLabel\differential\yLabel\Bigr|\Bigr )\\
        &=\frac{1}{\spaceStep}\sum_{\vectorial{\indexSpace}\in\naturals^2} \Bigl (\Bigl |\int_{\cell{\vectorial{\indexSpace}}} (\conservedVariable(\xLabel, \yLabel)-\conservedVariable(\xLabel+\spaceStep, \yLabel)) \differential\xLabel\differential\yLabel\Bigr | + \Bigl |\int_{\cell{\vectorial{\indexSpace}}} (\conservedVariable(\xLabel, \yLabel)-\conservedVariable(\xLabel, \yLabel+\spaceStep)) \differential\xLabel\differential\yLabel\Bigr | \Bigr )\\
        &\leq \frac{1}{\spaceStep} \int_{(\realsPositive)^2}|\conservedVariable(\xLabel, \yLabel)-\conservedVariable(\xLabel+\spaceStep, \yLabel)|\differential\xLabel\differential\yLabel + \frac{1}{\spaceStep} \int_{(\realsPositive)^2}|\conservedVariable(\xLabel, \yLabel)-\conservedVariable(\xLabel, \yLabel + \spaceStep)|\differential\xLabel\differential\yLabel.
    \end{align*}
    Now, consider $(\conservedVariable_k)_{k\in\naturals}$ in $\cKClass{\infty}((\realsPositive)^2)\cap \sobolevSpace{1}{1}((\realsPositive)^2)$ by \Cref{thm:approxThmBV}, hence such that 
    \begin{equation*}
        \conservedVariable_k\to \conservedVariable \quad \text{in}\quad \lebesgueSpace{1}((\realsPositive)^2), \qquad \text{and}\qquad \totalVariation{\conservedVariable_k}{(\realsPositive)^2} = \int_{(\realsPositive)^2}|\nabla \conservedVariable_k(\vectorial{\spaceVariable})|\differential\vectorial{\spaceVariable}\to\totalVariation{\conservedVariable}{(\realsPositive)^2}.
    \end{equation*}
    Let now $k\in\naturals$: we obtain 
    \begin{multline*}
        \totalVariation{\conservedVariable_{\discreteMark}}{(\realsPositive)^2}\leq \frac{1}{\spaceStep} \int_{(\realsPositive)^2}|\conservedVariable(\xLabel, \yLabel)-\conservedVariable_k(\xLabel, \yLabel)+\conservedVariable_k(\xLabel, \yLabel)-\conservedVariable_k(\xLabel+\spaceStep, \yLabel)+\conservedVariable_k(\xLabel+\spaceStep, \yLabel)-\conservedVariable(\xLabel+\spaceStep, \yLabel)|\differential\xLabel\differential\yLabel \\
        + \frac{1}{\spaceStep} \int_{(\realsPositive)^2}|\conservedVariable(\xLabel, \yLabel)-\conservedVariable_k(\xLabel, \yLabel)+\conservedVariable_k(\xLabel, \yLabel)-\conservedVariable_k(\xLabel, \yLabel + \spaceStep)+\conservedVariable_k(\xLabel, \yLabel + \spaceStep)-\conservedVariable(\xLabel, \yLabel + \spaceStep)|\differential\xLabel\differential\yLabel.
    \end{multline*}
    Thanks to the triangle inequality:
    \begin{align*}
        \totalVariation{\conservedVariable_{\discreteMark}}{(\realsPositive)^2}\leq \frac{4}{\spaceStep}\lVert\conservedVariable - \conservedVariable_k\rVert_{\lebesgueSpace{1}((\realsPositive)^2)} &+ \frac{1}{\spaceStep} \int_{(\realsPositive)^2}|\conservedVariable_k(\xLabel, \yLabel)-\conservedVariable_k(\xLabel+\spaceStep, \yLabel)|\differential\xLabel\differential\yLabel \\
        &+ \frac{1}{\spaceStep} \int_{(\realsPositive)^2}|\conservedVariable_k(\xLabel, \yLabel)-\conservedVariable_k(\xLabel, \yLabel + \spaceStep)|\differential\xLabel\differential\yLabel.
    \end{align*}
    By the fundamental theorem of calculus, which we can apply to $\conservedVariable_k$ thanks to its smoothness, we obtain 
    \begin{multline*}
        \totalVariation{\conservedVariable_{\discreteMark}}{(\realsPositive)^2}\leq \frac{4}{\spaceStep}\lVert\conservedVariable - \conservedVariable_k\rVert_{\lebesgueSpace{1}} + \frac{1}{\spaceStep} \int_{(\realsPositive)^2}\Bigl |\int_{\xLabel}^{\xLabel+\spaceStep} \partial_{\xLabel}\conservedVariable_k(\eta, \yLabel)\differential\eta \Bigr |\differential\xLabel\differential\yLabel 
        + \frac{1}{\spaceStep} \int_{(\realsPositive)^2}\Bigl |\int_{\yLabel}^{\yLabel+\spaceStep} \partial_{\yLabel}\conservedVariable_k(\xLabel, \eta)\differential\eta \Bigr |\differential\yLabel\differential\xLabel\\
        \leq  \frac{4}{\spaceStep}\lVert\conservedVariable - \conservedVariable_k\rVert_{\lebesgueSpace{1}} + \frac{1}{\spaceStep} \int_0^{+\infty}\reduceSpaceDoubleInt\int_0^{+\infty}\reduceSpaceDoubleInt\int_{\xLabel}^{\xLabel+\spaceStep} |\partial_{\xLabel}\conservedVariable_k(\eta, \yLabel)|\differential\eta \differential\xLabel\differential\yLabel 
        + \frac{1}{\spaceStep} \int_0^{+\infty}\reduceSpaceDoubleInt\int_0^{+\infty}\reduceSpaceDoubleInt\int_{\yLabel}^{\yLabel+\spaceStep} |\partial_{\yLabel}\conservedVariable_k(\xLabel, \eta)|\differential\eta \differential\yLabel\differential\xLabel.\\
    \end{multline*}
    Let us deal with the penultimate addendum on the right-hand side: the last term is handled analoguously. 
    We apply the Fubini's theorem to switch the order of integration in the last two integrals of this term, with the area on which the double integral is computed depicted in \Cref{fig:fubini}.
    Therefore, for any fixed $\yLabel>0$
    \begin{align*} 
        \int_0^{+\infty}\reduceSpaceDoubleInt\int_{\xLabel}^{\xLabel+\spaceStep} |\partial_{\xLabel}\conservedVariable_k(\eta, \yLabel)|\differential\eta \differential\xLabel &= \int_{0}^{\spaceStep}\int_0^{\eta}|\partial_{\xLabel}\conservedVariable_k(\eta, \yLabel)| \differential\xLabel \differential\eta + \int_{\spaceStep}^{+\infty}\int_{\eta-\spaceStep}^{\eta}|\partial_{\xLabel}\conservedVariable_k(\eta, \yLabel)| \differential\xLabel \differential\eta\\
        &= \int_{0}^{\spaceStep}\eta |\partial_{\xLabel}\conservedVariable_k(\eta, \yLabel)| \differential\eta + \int_{\spaceStep}^{+\infty}\spaceStep|\partial_{\xLabel}\conservedVariable_k(\eta, \yLabel)|  \differential\eta \leq \spaceStep\int_{0}^{+\infty}|\partial_{\xLabel}\conservedVariable_k(\eta, \yLabel)|\differential\eta.
    \end{align*}
    This entails 
    \begin{align*}
        \totalVariation{\conservedVariable_{\discreteMark}}{(\realsPositive)^2}
        &\leq \frac{4}{\spaceStep}\lVert\conservedVariable - \conservedVariable_k\rVert_{\lebesgueSpace{1}} + \int_{(\realsPositive)^2}|\partial_{\xLabel}\conservedVariable_k(\eta, \yLabel)|\differential\eta\differential\yLabel 
        + \int_{(\realsPositive)^2}|\partial_{\yLabel}\conservedVariable_k(\xLabel, \eta)|\differential\eta \differential\xLabel\\
        &= \frac{4}{\spaceStep}\lVert\conservedVariable - \conservedVariable_k\rVert_{\lebesgueSpace{1}((\realsPositive)^2)} + \int_{(\realsPositive)^2}|\nabla \conservedVariable_k(\vectorial{\spaceVariable})|\differential\vectorial{\spaceVariable}.
    \end{align*}
    Letting $k\to+\infty$ provides $\totalVariation{\conservedVariable_{\discreteMark}}{(\realsPositive)^2}\leq \totalVariation{\conservedVariable}{(\realsPositive)^2}$.

    \begin{figure}[h]
        \begin{center}
            \begin{tikzpicture}[->, thick]

                \fill[Cyan, opacity=0.5] 
                (0, 0) -- (3.5, 3.5) -- (3, 4) -- (0, 1) -- cycle;
            
            
                    \draw[->] (-0.5, 0) -- (5, 0) node[right] {\(x\)};
                    \draw[->] (0, -0.5) -- (0, 4) node[above] {\(\eta\)};
                    
                    \draw[-] (0, 0) -- (3.5, 3.5) node[right] {$\eta = x$};
                    \draw[-] (0, 1) -- (3, 4) node[right] {$\eta = x + \spaceStep$};
            
                    \draw[-, dotted] (0, 1) -- (1, 1);
                    \draw[-, dotted] (1, 1) -- (1, 0) node[below] {$\spaceStep$};

                    \node[left] at (0, 1) {$\spaceStep$};
              
              \end{tikzpicture}
        \end{center}\caption{\label{fig:fubini}Area of integration (light blue) when employing the Fubini's theorem in the proof of \Cref{prop:totalVariationPiecewiseApprox}.}
    \end{figure}
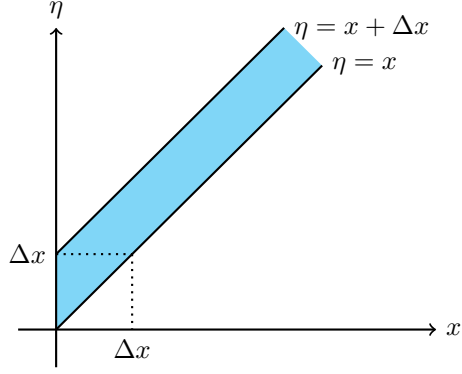

\end{proof}

\begin{proof}[Proof of \Cref{prop:totalVariationObsoluteValue}]
    The property of $|\conservedVariable|$ being in $\lebesgueSpace{\infty}(\Omega)\cap \lebesgueSpace{1}(\Omega)$ is trivial.
    Let us discuss the $\boundedVariationSpace$-properties. Consider $(\conservedVariable_k)_{k\in\naturals}$ in $\cKClass{\infty}(\Omega)\cap \sobolevSpace{1}{1}(\Omega)$ with the approximation properties ensured by \Cref{thm:approxThmBV}.
    By the reverse triangle inequality, one obtains that $|\conservedVariable_k|\to |\conservedVariable|$ in $\lebesgueSpace{1}(\Omega)$.
    Let now $\testFunction \in\smoothFunctionsSpaceWithReg{1}(\Omega;\reals^{\spatialDimensionality})$ such that $\lVert\testFunction\rVert_{\lebesgueSpace{\infty}(\Omega)}\leq 1$: we obtain 
    \begin{equation*}
        \int_{\Omega}|\conservedVariable_k(\vectorial{\spaceVariable})|\divergence(\testFunction (\vectorial{\spaceVariable}))\differential\vectorial{\spaceVariable}\to \int_{\Omega}|\conservedVariable(\vectorial{\spaceVariable})|\divergence(\testFunction(\vectorial{\spaceVariable}))\differential\vectorial{\spaceVariable}.
    \end{equation*}
    As $\nabla(|\conservedVariable_k|) = \nabla(\sign(\conservedVariable_k)\conservedVariable_k) =\sign(\conservedVariable_k) \nabla(\conservedVariable_k)$ in the distributional sense, we deduce that $\nabla(|\conservedVariable_k|)\in\lebesgueSpace{1}(\Omega)$, for $\conservedVariable_k\in\sobolevSpace{1}{1}(\Omega)$.
    This entails 
    \begin{equation*}
        \int_{\Omega}|\nabla(|\conservedVariable_k(\vectorial{\spaceVariable})|)|\differential\vectorial{\spaceVariable} = \totalVariation{|\conservedVariable_k|}{\Omega} = \int_{\Omega}|\nabla(\conservedVariable_k(\vectorial{\spaceVariable}))|\differential\vectorial{\spaceVariable} = \totalVariation{\conservedVariable_k}{\Omega},
    \end{equation*}
    which---by the definition of total variation---gives that $\int_{\Omega}|\conservedVariable_k(\vectorial{\spaceVariable})|\divergence(\testFunction(\vectorial{\spaceVariable}))\differential\vectorial{\spaceVariable}\leq \totalVariation{\conservedVariable_k}{\Omega}$.
    Taking the limit for $k\to+\infty$ gives 
    \begin{equation*}
        \int_{\Omega}|\conservedVariable(\vectorial{\spaceVariable})|\divergence(\testFunction(\vectorial{\spaceVariable}))\differential\vectorial{\spaceVariable}\leq \totalVariation{\conservedVariable}{\Omega},
    \end{equation*}
    hence a supremum over the test functions $\testFunction$ gives the final result.
\end{proof}

\begin{proof}[Proof of \Cref{prop:restrictionTotalVariation}]
    Using the Fubini's theorem, we have that $v\in\lebesgueSpace{1}(a_{\xLabel}, b_{\xLabel})$. Now let $\testFunction\in\smoothFunctionsSpaceWithReg{1}(a_{\xLabel}, b_{\xLabel})$ and observe that the function $(\xLabel, \yLabel)\mapsto \conservedVariable(\xLabel, \yLabel)\testFunction'(\xLabel)$ is integrable on $\Omega$.
    This yields 
    \begin{align*}
        \totalVariation{v}{(a_{\xLabel}, b_{\xLabel})} &= \sup \Bigl \{ \int_{a_{\xLabel}}^{b_{\xLabel}} v(\xLabel)\testFunction'(\xLabel)\differential\xLabel, \qquad \testFunction\in\smoothFunctionsSpaceWithReg{1}(a_{\xLabel}, b_{\xLabel}), \quad \lVert\testFunction\rVert_{\lebesgueSpace{\infty}(a_{\xLabel}, b_{\xLabel})}\leq 1 \Bigr \}\\
        &= \sup \Bigl \{ \int_{a_{\xLabel}}^{b_{\xLabel}}\Bigl (  \int_{a_{\yLabel}}^{b_{\yLabel}} \conservedVariable(\xLabel, \yLabel)\differential\yLabel\Bigr )\testFunction'(\xLabel)\differential\xLabel, \qquad \testFunction\in\smoothFunctionsSpaceWithReg{1}(a_{\xLabel}, b_{\xLabel}), \quad \lVert\testFunction\rVert_{\lebesgueSpace{\infty}(a_{\xLabel}, b_{\xLabel})}\leq 1\Bigr \}\\
        &= \sup \Bigl \{ \int_{a_{\yLabel}}^{b_{\yLabel}}\Bigl (  \int_{a_{\xLabel}}^{b_{\xLabel}} \conservedVariable(\xLabel, \yLabel)\testFunction'(\xLabel)\differential\xLabel\Bigr )\differential\yLabel, \qquad \testFunction\in\smoothFunctionsSpaceWithReg{1}(a_{\xLabel}, b_{\xLabel}), \quad \lVert\testFunction\rVert_{\lebesgueSpace{\infty}(a_{\xLabel}, b_{\xLabel})}\leq 1\Bigr \}\\
        &\leq \int_{a_{\yLabel}}^{b_{\yLabel}}\sup \Bigl \{ \Bigl (  \int_{a_{\xLabel}}^{b_{\xLabel}}\conservedVariable(\xLabel, \yLabel)\testFunction'(\xLabel)\differential\xLabel\Bigr ), \qquad \testFunction\in\smoothFunctionsSpaceWithReg{1}(a_{\xLabel}, b_{\xLabel}), \quad \lVert\testFunction\rVert_{\lebesgueSpace{\infty}(a_{\xLabel}, b_{\xLabel})}\leq 1\Bigr \}\differential\yLabel\\
        &=\int_{a_{\yLabel}}^{b_{\yLabel}}\totalVariation{\conservedVariable(\cdot, \yLabel)}{(a_{\xLabel}, b_{\xLabel})}\differential\yLabel = \totalVariationAlongAxis{\conservedVariable}{\Omega}{\xLabel}.
    \end{align*}
    The last inequality in the claim comes from \Cref{rem:1Dvs2D}.
\end{proof}

\begin{proof}[Proof of \Cref{prop:restrictionLinF}]
    Thanks to \Cref{prop:restrictionTotalVariation}, we have that $v\in\lebesgueSpace{1}(a_{\xLabel}, b_{\xLabel})\cap \boundedVariationSpace(a_{\xLabel}, b_{\xLabel})$.
    Using \cite[Pages 134 and 135]{ambrosio2000functions}, we deduce the claim since $\lebesgueSpace{1}(a_{\xLabel}, b_{\xLabel})\cap \boundedVariationSpace(a_{\xLabel}, b_{\xLabel})\subset \lebesgueSpace{\infty}(a_{\xLabel}, b_{\xLabel})$.
\end{proof}

\begin{proof}[Proof of \Cref{prop:VerticalIntegrationBV}]
    We have 
    \begin{equation*}
        \Bigl  | \int_{c_{\xLabel}}^{c_{\xLabel}+\eta}v(\xLabel)\differential\xLabel\Bigr  |\leq \int_{c_{\xLabel}}^{c_{\xLabel}+\eta}\lVert v \rVert_{\lebesgueSpace{\infty}(a_{\xLabel}, b_{\xLabel})}\differential\xLabel  = \eta \lVert v \rVert_{\lebesgueSpace{\infty}(a_{\xLabel}, b_{\xLabel})},
    \end{equation*}
    where the right-hand side is finite by virtue of \Cref{prop:restrictionLinF}.
\end{proof}

\section{Proof of \Cref{prop:EstimateWrongBoundaryRelaxationD1Q2}}\label{app:proofRelEstimate}

\begin{proof}[Proof of \Cref{prop:EstimateWrongBoundaryRelaxationD1Q2}]
    Let us start with the first part of the claim: as we consider $\relaxationParameter = 1$, the scheme can be rewritten as a one-step scheme only on the conserved moment $\conservedVariableDiscrete$, which reads $\conservedVariableDiscrete_{\indexSpace}^0 = 0$ for $\indexSpace\in\integerInterval{0}{\numberCells-1}$ and, for $\indexTime\in\naturals$
    \begin{equation*}
        \begin{cases}
            \conservedVariableDiscrete_{0}^{\indexTime + 1} = \tfrac{1}{2}(1+\courantNumber) \tilde{\conservedVariableDiscrete}_{\labelWest} + \tfrac{1}{2}(1-\courantNumber) \conservedVariableDiscrete_{1}^{\indexTime}, \\
            \conservedVariableDiscrete_{\indexSpace}^{\indexTime + 1} = \tfrac{1}{2}(1+\courantNumber) \conservedVariableDiscrete_{\indexSpace - 1}^{\indexTime} + \tfrac{1}{2}(1-\courantNumber) \conservedVariableDiscrete_{\indexSpace + 1}^{\indexTime}, \qquad \indexSpace\in\integerInterval{1}{\numberCells-2}, \\
            \conservedVariableDiscrete_{\numberCells-1}^{\indexTime + 1} = \tfrac{1}{2}(1+\courantNumber) \conservedVariableDiscrete_{\numberCells-2}^{\indexTime}.
        \end{cases}
    \end{equation*}
    The solution $\conservedVariableDiscrete^{\indexTime}\definitionEquality\transpose{(\conservedVariableDiscrete_{0}^{\indexTime}, \dots, \conservedVariableDiscrete_{\numberCells-1}^{\indexTime})}$ hence satisfies the recurrence $\conservedVariableDiscrete^{0} = \vectorial{0}_{\numberCells}$ and $\conservedVariableDiscrete^{\indexTime + 1} = \matricial{A} \conservedVariableDiscrete^{\indexTime} + \tfrac{1}{2}(1+\courantNumber) \tilde{\conservedVariableDiscrete}_{\labelWest}\canonicalBasisVector{1}$ for $\indexTime\in\naturals$, where $\matricial{A}$ is the $\numberCells\times\numberCells$ tridiagonal matrix with $\tfrac{1}{2}(1+\courantNumber)$ on the subdiagonal, zero on the diagonal, and $\tfrac{1}{2}(1-\courantNumber)$ supradiagonal, whereas $\canonicalBasisVector{1}$ is the first vector of the canonical basis of $\reals^{\numberCells}$.
    We thus obtain, for $\indexTime\geq 1$:
    \begin{equation}\label{eq:tmp2}
        \conservedVariableDiscrete^{\indexTime} = \frac{\tilde{\conservedVariableDiscrete}_{\labelWest}}{2}(1+\courantNumber) \Bigl (\sum_{p=0}^{\indexTime-1}\matricial{A}^p \Bigr )\canonicalBasisVector{1}.
    \end{equation}
    To obtain \eqref{eq:estimateD1Q2Chebyshev}, we write the powers $\matricial{A}^p$ in \eqref{eq:tmp2} using the formula from \cite[Theorem 2]{gutierrez2008powers}.

    For the second part of the claim, having $\relaxationParameter\in(0, 2)$, we consider a domain infinite to the right, and adopt a GKS (Gustafsson, Kreiss, and Sundstr\"om) construction based on the $\timeShiftOperator$-transform, see \cite{bellotti2024consistency}. The $\timeShiftOperator$-transformed boundary equations read 
    \begin{equation}\label{eq:zTransBoundary}
        \begin{cases*}
            \timeShiftOperator \tilde{\conservedVariableDiscrete}_0(\timeShiftOperator) - \tfrac{1}{2}(1-\relaxationParameter\courantNumber) \tilde{\conservedVariableDiscrete}_1(\timeShiftOperator) + \tfrac{1}{2\latticeVelocity}(1-\relaxationParameter)\tilde{\discrete{v}}_1(\timeShiftOperator) = \tfrac{1}{2}(1+\courantNumber) \tilde{\conservedVariableDiscrete}_{\labelWest} \frac{\timeShiftOperator}{\timeShiftOperator-1}, \\
            \timeShiftOperator \tilde{\discrete{v}}_0(\timeShiftOperator) + \tfrac{\latticeVelocity}{2}(1-\relaxationParameter\courantNumber) \tilde{\conservedVariableDiscrete}_1(\timeShiftOperator) - \tfrac{1}{2}(1-\relaxationParameter)\tilde{\discrete{v}}_1(\timeShiftOperator) = \tfrac{\latticeVelocity}{2}(1+\courantNumber) \tilde{\conservedVariableDiscrete}_{\labelWest} \frac{\timeShiftOperator}{\timeShiftOperator-1},
        \end{cases*}
    \end{equation}
    where $\discrete{v}$ is the first-order moment of the distribution functions.
    For the inner points, that is for $\indexSpace\geq 1$, the scheme reads
    \begin{equation}\label{eq:inner}
        \begin{cases*}
            \timeShiftOperator \tilde{\conservedVariableDiscrete}_{\indexSpace}(\timeShiftOperator) - \tfrac{1}{2}(1+\relaxationParameter\courantNumber) \tilde{\conservedVariableDiscrete}_{\indexSpace-1}(\timeShiftOperator) - \tfrac{1}{2}(1-\relaxationParameter\courantNumber) \tilde{\conservedVariableDiscrete}_{\indexSpace+1}(\timeShiftOperator) - \tfrac{1}{2\latticeVelocity}(1-\relaxationParameter)\tilde{\discrete{v}}_{\indexSpace - 1}(\timeShiftOperator) + \tfrac{1}{2\latticeVelocity}(1-\relaxationParameter)\tilde{\discrete{v}}_{\indexSpace + 1}(\timeShiftOperator) = 0, \\
            \timeShiftOperator \tilde{\discrete{v}}_{\indexSpace}(\timeShiftOperator) - \tfrac{\latticeVelocity}{2}(1+\relaxationParameter\courantNumber) \tilde{\conservedVariableDiscrete}_{\indexSpace- 1}(\timeShiftOperator) + \tfrac{\latticeVelocity}{2}(1-\relaxationParameter\courantNumber) \tilde{\conservedVariableDiscrete}_{\indexSpace + 1}(\timeShiftOperator)  - \tfrac{1}{2}(1-\relaxationParameter)\tilde{\discrete{v}}_{\indexSpace - 1}(\timeShiftOperator) - \tfrac{1}{2}(1-\relaxationParameter)\tilde{\discrete{v}}_{\indexSpace + 1}(\timeShiftOperator) =0.
        \end{cases*}
    \end{equation}
    Calling 
    \begin{equation*}
        \matricial{L}_{\timeShiftOperator}(\fourierShift)=
        \begin{pmatrix}
            \timeShiftOperator \fourierShift - \tfrac{1}{2}(1+\relaxationParameter\courantNumber) - \tfrac{1}{2}(1-\relaxationParameter\courantNumber) \fourierShift^2 & - \tfrac{1}{2\latticeVelocity}(1-\relaxationParameter) + \tfrac{1}{2\latticeVelocity}(1-\relaxationParameter) \fourierShift^2\\
            - \tfrac{\latticeVelocity}{2}(1+\relaxationParameter\courantNumber) + \tfrac{\latticeVelocity}{2}(1-\relaxationParameter\courantNumber) \fourierShift^2 & \timeShiftOperator \fourierShift- \tfrac{1}{2}(1-\relaxationParameter) - \tfrac{1}{2}(1-\relaxationParameter)\fourierShift^2
        \end{pmatrix},
    \end{equation*}
    from the theory of linear recurrences \cite[Chapter 8]{gohberg2005matrix}, the $\ell^2$-stable-in-space (on $\naturals$) general solution to \eqref{eq:inner} reads, for $|\timeShiftOperator|\geq 1$
    \begin{equation}\label{eq:ansatzFull}
        \begin{pmatrix}
            \tilde{\conservedVariableDiscrete}_{\indexSpace}(\timeShiftOperator)\\
            \tilde{\discrete{v}}_{\indexSpace}(\timeShiftOperator)
        \end{pmatrix}
         = C_-(\timeShiftOperator)\vectorial{\varphi}_-(\timeShiftOperator)\solutionCharStable(\timeShiftOperator)^{\indexSpace} + C_0(\timeShiftOperator)\vectorial{\varphi}_0(\timeShiftOperator)\delta_{\indexSpace 0}, \qquad \indexSpace\in\naturals,
    \end{equation}
    where $\solutionCharStable(\timeShiftOperator)$ is the stable solution of \eqref{eq:charEquation}, $\vectorial{\varphi}_-(\timeShiftOperator)\in\textnormal{ker}(\matricial{L}_{\timeShiftOperator}(\solutionCharStable(\timeShiftOperator)))$, $\vectorial{\varphi}_0(\timeShiftOperator)\in\textnormal{ker}(\matricial{L}_{\timeShiftOperator}(0))$, and the constants $C_-(\timeShiftOperator)$ and $C_0(\timeShiftOperator)$ are determined using the boundary conditions \eqref{eq:zTransBoundary}.
    The inverse $\timeShiftOperator$-transform is 
    \begin{equation}\label{eq:inverseZ}
        \conservedVariableDiscrete_{\indexSpace}^{\indexTime} = \frac{1}{2\pi i}\oint_{C}\timeShiftOperator^{\indexTime-1}\tilde{\conservedVariableDiscrete}_{\indexSpace}(\timeShiftOperator)\differential\timeShiftOperator = \sum_{\substack{\tilde{\timeShiftOperator}\text{ s.t.}\\
        \timeShiftOperator^{\indexTime-1}\tilde{\conservedVariableDiscrete}_{\indexSpace}(\timeShiftOperator)\\
        \text{singular}}} \textnormal{Res}_{\timeShiftOperator = \tilde{\timeShiftOperator}}[\timeShiftOperator^{\indexTime-1}\tilde{\conservedVariableDiscrete}_{\indexSpace}(\timeShiftOperator)]\approx  \sum_{\substack{|\tilde{\timeShiftOperator}|\geq1\text{ s.t.}\\
        \timeShiftOperator^{\indexTime-1}\tilde{\conservedVariableDiscrete}_{\indexSpace}(\timeShiftOperator)\\
        \text{singular}}} \textnormal{Res}_{\timeShiftOperator =\tilde{\timeShiftOperator}}[\timeShiftOperator^{\indexTime-1}\tilde{\conservedVariableDiscrete}_{\indexSpace}(\timeShiftOperator)],
    \end{equation}
    where $C$ is a contour, directed in the counterclockwise sense, large enough to enclose all the poles, and the approximation is a long-time one, thus for $\indexTime\gg 1$.
    Under this approximation, we now understand why the study of the $\timeShiftOperator$-transformed discrete solution for $|{\timeShiftOperator}|\geq 1$ is enough.

    Inserting \eqref{eq:ansatzFull} into \eqref{eq:zTransBoundary}, we notice that the only pole of both $\tilde{\conservedVariableDiscrete}_{\indexSpace}(\timeShiftOperator)$ and $\tilde{\discrete{v}}_{\indexSpace}(\timeShiftOperator)$ is at $\timeShiftOperator = 1$, and that in this vicinity 
    \begin{equation*}
        \vectorial{\varphi_{-}}(\timeShiftOperator) = 
        \begin{pmatrix}
        1\\
        \frac{\relaxationParameter\advectionVelocity}{\relaxationParameter-2}
        \end{pmatrix} + \bigO{\timeShiftOperator-1}, \quad C_-(\timeShiftOperator) = \tilde{\conservedVariableDiscrete}_{\labelWest} \frac{(2-\relaxationParameter)(1+\courantNumber)}{2-\relaxationParameter(1+\courantNumber)} \frac{\timeShiftOperator}{\timeShiftOperator - 1} + \bigO{1}, \qquad \vectorial{\varphi}_{0}(\timeShiftOperator) = 
        \begin{pmatrix}
            1\\
            \frac{\latticeVelocity + \relaxationParameter\advectionVelocity}{\relaxationParameter-1}
        \end{pmatrix}
        , \quad C_0(\timeShiftOperator) = \bigO{\timeShiftOperator-1}.
    \end{equation*}
    From \eqref{eq:inverseZ}, we therefore obtain 
    \begin{equation*}
        \conservedVariableDiscrete_{\indexSpace}^{\indexTime}\approx \tilde{\conservedVariableDiscrete}_{\labelWest} \frac{(2-\relaxationParameter)(1+\courantNumber)}{2-\relaxationParameter(1+\courantNumber)} \textnormal{Res}_{\timeShiftOperator = 1}\Bigl [\frac{\timeShiftOperator^{\indexTime}\solutionCharStable(\timeShiftOperator)^{\indexSpace}}{\timeShiftOperator - 1}\Bigr ] = \tilde{\conservedVariableDiscrete}_{\labelWest} \frac{(2-\relaxationParameter)(1+\courantNumber)}{2-\relaxationParameter(1+\courantNumber)}\Bigl ( \frac{2-\relaxationParameter+\relaxationParameter\courantNumber}{2-\relaxationParameter-\relaxationParameter\courantNumber}\Bigr )^{\indexSpace}.
    \end{equation*}

    Notice that in the case $\relaxationParameter = 1$, we can deduce \eqref{eq:estimationLongTimeRelaxationD1Q2WrongBC} in another way. Equation \eqref{eq:tmp2} entails that
    \begin{equation*}
        \lim_{\indexTime\to+\infty}\conservedVariableDiscrete^{\indexTime} = \frac{\tilde{\conservedVariableDiscrete}_{\labelWest}}{2}(1+\courantNumber) \Bigl (\sum_{p=0}^{+\infty}\matricial{A}^p \Bigr )\canonicalBasisVector{1} = \frac{\tilde{\conservedVariableDiscrete}_{\labelWest}}{2}(1+\courantNumber) (\identityMatrix{\numberCells} - \matricial{A})^{-1} \canonicalBasisVector{1},
    \end{equation*}
    where we used the Neumann series. Remark that $\identityMatrix{\numberCells} - \matricial{A}$ is a tridiagonal matrix with $-\tfrac{1}{2}(1+\courantNumber)$ on the subdiagonal, $1$ on the diagonal, and $-\tfrac{1}{2}(1-\courantNumber)$ supradiagonal.
    We can therefore employ \cite[Corollary 4.1]{DAFONSECA20017} to make the expression of $(\identityMatrix{\numberCells} - \matricial{A})^{-1}$ explicit:
    \begin{equation*}
        \lim_{\indexTime\to+\infty}\conservedVariableDiscrete^{\indexTime}_{\indexSpace} = \frac{\tilde{\conservedVariableDiscrete}_{\labelWest}}{2}(1+\courantNumber)(-1)^{\indexSpace} \frac{\Bigl ( -\tfrac{1}{2}(1+\courantNumber)\Bigr )^{\indexSpace}}{\Bigl ( \tfrac{1}{2}\sqrt{1-\courantNumber^2}\Bigr )^{\indexSpace + 1}} \frac{U_{\numberCells-\indexSpace-1}(\frac{1}{\sqrt{1-\courantNumber^2}})}{U_{\numberCells}(\frac{1}{\sqrt{1-\courantNumber^2}})}.
    \end{equation*}
    Notice that under the strict CFL condition, we have that $\frac{1}{\sqrt{1-\courantNumber^2}}\in(1, +\infty)$, hence the previous expression never vanish, nor the denominator become zero.
    Proceeding as in the proof of \cite[Theorem 2]{bellotti2024consistency}, we obtain that 
    \begin{align*}
        \lim_{\numberCells\to+\infty} \lim_{\indexTime\to+\infty}\conservedVariableDiscrete^{\indexTime}_{\indexSpace} &= \frac{\tilde{\conservedVariableDiscrete}_{\labelWest}}{2}(1+\courantNumber)(-1)^{\indexSpace} \frac{\Bigl ( -\tfrac{1}{2}(1+\courantNumber)\Bigr )^{\indexSpace}}{\Bigl ( \tfrac{1}{2}\sqrt{1-\courantNumber^2}\Bigr )^{\indexSpace + 1}} \lim_{\numberCells\to+\infty}  \frac{U_{\numberCells-\indexSpace-1}(\frac{1}{\sqrt{1-\courantNumber^2}})}{U_{\numberCells}(\frac{1}{\sqrt{1-\courantNumber^2}})} \\
        &= \frac{\tilde{\conservedVariableDiscrete}_{\labelWest}}{2}(1+\courantNumber)(-1)^{\indexSpace} \frac{\Bigl ( -\tfrac{1}{2}(1+\courantNumber)\Bigr )^{\indexSpace}}{\Bigl ( \tfrac{1}{2}\sqrt{1-\courantNumber^2}\Bigr )^{\indexSpace + 1}} \frac{(1+\courantNumber)^{\indexSpace+1}}{(\sqrt{1-\courantNumber^2})^{\indexSpace+1}} = \tilde{\conservedVariableDiscrete}_{\labelWest}\Bigl ( \frac{1+\courantNumber}{1-\courantNumber}\Bigr )^{\indexSpace + 1}.
    \end{align*}

\end{proof}

\end{document}